\documentclass[10pt]{amsart}

\usepackage[T1]{fontenc}
\usepackage[lf]{Baskervaldx} 
\usepackage[bigdelims,vvarbb]{newtxmath} 
\usepackage[cal=boondoxo]{mathalfa} 

\usepackage{graphicx}
\usepackage[all]{xy}
\usepackage{a4wide}
\usepackage{mathrsfs} 

\usepackage{enumitem}

\usepackage{caption}
\usepackage{subcaption}

\usepackage{multirow}
\usepackage[dvipsnames]{xcolor}
\usepackage[colorlinks,final,hyperindex]{hyperref}
\usepackage[noabbrev,capitalize]{cleveref}
\usepackage{tikz}
\usepackage{tikz-cd}
\tikzset{math3d/.style=
    {x= {(-0.353cm,-0.353cm)}, z={(0cm,1cm)},y={(1cm,0cm)}}}
\tikzset{JLL3d/.style=
    {x= {(0.4cm,-0.2cm)}, z={(0cm,1cm)},y={(-1cm,0cm)}}}
\tikzset{J4/.style=
   {x= {(-0.87cm,-0.5cm)}, z={(0cm,1cm)},y={(0.87cm,-0.5cm)}}}

\calclayout

\definecolor{Chocolat}{rgb}{0.36, 0.2, 0.09}
\definecolor{BleuTresFonce}{rgb}{0.215, 0.215, 0.36}
\definecolor{RougeTresFonce}{rgb}{0.36, 0.215, 0.215}
\hypersetup{citecolor=BleuTresFonce, linkcolor=black}

\newtheorem{definition}{Definition}[section]
\newtheorem{proposition}[definition]{Proposition}
\newtheorem{lemma}[definition]{Lemma}
\newtheorem{theorem}{Theorem}

\newtheorem*{problem}{Problem}

\theoremstyle{remark}
\newtheorem{example}[definition]{\sc Example}
\newtheorem{remark}[definition]{\sc Remark}

\newcommand{\RR}{\mathbb{R}}

\newcommand{\K}{\mathrm{K}}
\newcommand{\J}{\mathrm{J}}
\newcommand{\La}{\mathcal{L}}
\newcommand{\Tam}[1]{\mathrm{PBT}_{#1}}

\newcommand{\CT}[1]{\mathrm{CT}_{#1}}
\newcommand{\CMT}[1]{\mathrm{CAT}_{#1}}

\newcommand{\PolySub}{\mathsf{Poly}}

\DeclareMathOperator{\tp}{top}
\DeclareMathOperator{\bm}{bot}

\DeclareMathOperator{\conv}{conv}

\newcommand{\tr}{\mathrm{tr}}
\newcommand{\id}{\mathrm{id}}

\newcommand{\B}{\mathrm{B}}
\newcommand{\T}{\mathrm{T}}
\newcommand{\iBT}{\mathrm{i}^\T_\B}
\newcommand{\cat}[1]{\mathcal{#1}}
\newcommand{\gra}{\ell}
\newcommand{\Ainf}{\ensuremath{\mathrm{A}_\infty}}
\newcommand{\Linf}{\ensuremath{\mathrm{L}_\infty}}
\newcommand{\Ainfdeux}{\ensuremath{\mathrm{A}_\infty^2}}
\newcommand{\uAinf}{\ensuremath{\mathrm{uA}_\infty}}
\newcommand{\AAinf}{\mathrm{AA}_\infty} 
\newcommand{\Minf}{\ensuremath{\mathrm{M}_\infty}} 
\newcommand{\Aalg}{\ensuremath{\mathrm{A}_\infty\text{-\hspace{2pt}}\mathsf{alg}}} 
\newcommand{\Acog}{\ensuremath{\mathrm{A}_\infty\text{-\hspace{2pt}}\mathsf{cog}}} 
\newcommand{\infAalg}{\ensuremath{\infty\text{-}\mathrm{A}_\infty\text{-\hspace{1pt}}\mathsf{alg}}} 
\newcommand{\infAcog}{\ensuremath{\infty\text{-}\mathrm{A}_\infty\text{-}\mathsf{cog}}} 
\newcommand{\ide}{\ensuremath{\mathrm{id}}}
\newcommand{\mono}{\ensuremath{\mathrm{mono}}}
\newcommand{\comp}{\ensuremath{\mathrm{comp}}}

\newcommand{\diagainf}{\ensuremath{\triangle^{K}}}
\newcommand{\diagminf}{\ensuremath{\triangle^{J}}}

\newcommand{\OEIS}[1]{{\rm \href{http://oeis.org/#1}{\texttt{#1}}}}

\newcommand{\black}[1]{\boldsymbol{\left(\right.} #1 \boldsymbol{\left.\right)}}
\newcommand{\blue}[1]{\textcolor{MidnightBlue}{\boldsymbol{\left(\right.}} #1 \textcolor{MidnightBlue}{\boldsymbol{\left.\right)}}} 
\newcommand{\red}[1]{\textcolor{Red!60}{\boldsymbol{\left(\right.}} #1 \textcolor{Red!60}{\boldsymbol{\left.\right)}}} 
\newcommand{\purple}[1]{\textcolor{Purple!80}{\boldsymbol{\left[\right.}} #1 \textcolor{Purple!80}{\boldsymbol{\left.\right]}}}

\newcommand{\pointbullet}{
\begin{tikzpicture}[very thick]
\draw[fill,black] (0,0) circle (0.5) ;
\end{tikzpicture}}

\newcommand{\TreeLa}
{\vcenter{\hbox{
\begin{tikzpicture}[scale=0.25]
\draw[very thick, MidnightBlue] (0,-0.5)--(0,0) -- (-2,2)--(-2,2.5);
\draw[very thick, MidnightBlue] (-1,1)--(0,2)--(0,2.5) ;
\draw[very thick, MidnightBlue] (0,0)--(1,1)--(1,2.5) ;
\draw[very thick, Red!60] (0,-0.5)--(0,-1) ;
\draw (-0.5,-0.5) --(0.5, -0.5);
\end{tikzpicture}}}}

\newcommand{\TreeLb}
{\vcenter{\hbox{
\begin{tikzpicture}[scale=0.25]
\draw[very thick, MidnightBlue] (-0.5,0.5)--(-2,2)--(-2,2.5);
\draw[very thick, MidnightBlue] (-1,1)--(0,2)--(0,2.5) ;
\draw[very thick, MidnightBlue] (0.5,0.5)--(1,1)--(1,2.5) ;
\draw[very thick, Red!60] (0,0)--(0,-0.5) ;
\draw[very thick, Red!60] (0,0)--(-0.5, 0.5) ;
\draw[very thick, Red!60] (0,0)--(0.5,0.5) ;
\draw (-1,0.5) --(1, 0.5);
\end{tikzpicture}}}}

\newcommand{\TreeLc}
{\vcenter{\hbox{
\begin{tikzpicture}[scale=0.25]
\draw[very thick, MidnightBlue] (-2, 2.5)--(-2, 3) ;
\draw[very thick, MidnightBlue] (0, 2.5)--(0, 3) ;
\draw[very thick, MidnightBlue] (1, 2.5)--(1, 3) ;
\draw[very thick, Red!60] (0,-0.5)--(0,0) -- (-2,2)--(-2,2.5);
\draw[very thick, Red!60] (-1,1)--(0,2)--(0,2.5) ;
\draw[very thick, Red!60] (0,0)--(1,1)--(1,2.5) ;
\draw (-2.5,2.5) --(1.5, 2.5);
\end{tikzpicture}}}}

\newcommand{\TreeRa}
{\vcenter{\hbox{
\begin{tikzpicture}[scale=0.25]
\draw[very thick, MidnightBlue] (0,-0.5)--(0,0) -- (2,2)--(2,2.5);
\draw[very thick, MidnightBlue] (1,1)--(0,2)--(0,2.5) ;
\draw[very thick, MidnightBlue] (0,0)--(-1,1)--(-1,2.5) ;
\draw[very thick, Red!60] (0,-0.5)--(0,-1) ;
\draw (-0.5,-0.5) --(0.5, -0.5);
\end{tikzpicture}}}}

\newcommand{\TreeRb}
{\vcenter{\hbox{
\begin{tikzpicture}[scale=0.25]
\draw[very thick, MidnightBlue] (0.5,0.5) -- (2,2)--(2,2.5);
\draw[very thick, MidnightBlue] (1,1)--(0,2)--(0,2.5) ;
\draw[very thick, MidnightBlue] (-0.5,0.5)--(-1,1)--(-1,2.5) ;
\draw[very thick, Red!60] (0,-0.5)--(0,0) ;
\draw[very thick, Red!60] (0,0)--(0.5,0.5) ;
\draw[very thick, Red!60] (0,0)--(-0.5,0.5) ;
\draw (-1,0.5) --(1, 0.5);
\end{tikzpicture}}}}

\newcommand{\TreeRc}
{\vcenter{\hbox{
\begin{tikzpicture}[scale=0.25]
\draw[very thick, Red!60] (0,-0.5)--(0,0) -- (2,2)--(2,2.5);
\draw[very thick, Red!60] (1,1)--(0,2)--(0,2.5) ;
\draw[very thick, Red!60] (0,0)--(-1,1)--(-1,2.5) ;
\draw[very thick, MidnightBlue] (2,2.5)--(2,3) ;
\draw[very thick, MidnightBlue] (0,2.5)--(0,3) ;
\draw[very thick, MidnightBlue] (-1,2.5)--(-1,3) ;
\draw (-1.5,2.5) --(2.5, 2.5);
\end{tikzpicture}}}}

\newcommand{\TreeLab}
{\vcenter{\hbox{
\begin{tikzpicture}[scale=0.25]
\draw[very thick, MidnightBlue] (0,0) -- (-2,2)--(-2,2.5);
\draw[very thick, MidnightBlue] (-1,1)--(0,2)--(0,2.5) ;
\draw[very thick, MidnightBlue] (0,0)--(1,1)--(1,2.5) ;
\draw[very thick, Red!60] (0,0)--(0,-0.5) ;
\draw (-0.5,0) --(0.5, 0);
\end{tikzpicture}}}}

\newcommand{\TreeLbc}
{\vcenter{\hbox{
\begin{tikzpicture}[scale=0.25]
\draw[very thick, MidnightBlue] (-1,1) -- (-2,2)--(-2,2.5);
\draw[very thick, MidnightBlue] (-1,1)--(0,2)--(0,2.5) ;
\draw[very thick, MidnightBlue] (1,1)--(1,2.5) ;
\draw[very thick, Red!60] (0,-0.5)--(0,0) ;
\draw[very thick, Red!60] (0,0)--(-1,1) ;
\draw[very thick, Red!60] (0,0)--(1,1) ;
\draw (-1.5,1) --(1.5, 1);
\end{tikzpicture}}}}

\newcommand{\TreeRab}
{\vcenter{\hbox{
\begin{tikzpicture}[scale=0.25]
\draw[very thick, MidnightBlue] (0,0) -- (2,2)--(2,2.5);
\draw[very thick, MidnightBlue] (1,1)--(0,2)--(0,2.5) ;
\draw[very thick, MidnightBlue] (0,0)--(-1,1)--(-1,2.5) ;
\draw[very thick, Red!60] (0,-0.5)--(0,0) ;
\draw (-0.5,0) --(0.5, 0);
\end{tikzpicture}}}}

\newcommand{\TreeRbc}
{\vcenter{\hbox{
\begin{tikzpicture}[scale=0.25]
\draw[very thick, MidnightBlue] (1,1) -- (2,2)--(2,2.5);
\draw[very thick, MidnightBlue] (1,1)--(0,2)--(0,2.5) ;
\draw[very thick, MidnightBlue] (-1,1)--(-1,2.5) ;
\draw[very thick, Red!60] (0,-0.5)--(0,0) ;
\draw[very thick, Red!60] (0,0)--(1,1) ;
\draw[very thick, Red!60] (0,0)--(-1,1) ;
\draw (-1.5,1) --(1.5, 1);
\end{tikzpicture}}}}

\newcommand{\TreeCa}
{\vcenter{\hbox{
\begin{tikzpicture}[scale=0.25]
\draw[very thick, MidnightBlue] (0,-0.5) -- (0,1.5);
\draw[very thick, MidnightBlue] (0,0) -- (-1,1)--(-1,1.5);
\draw[very thick, MidnightBlue] (0,0) -- (1,1)--(1,1.5);
\draw[very thick, Red!60] (0,-0.5)--(0,-1) ;
\draw (-0.5,-0.5) --(0.5, -0.5);
\end{tikzpicture}}}}

\newcommand{\TreeCb}
{\vcenter{\hbox{
\begin{tikzpicture}[scale=0.25]
\draw[very thick, Red!60] (0,-0.5) -- (0,1.5);
\draw[very thick, Red!60] (0,0) -- (-1,1)--(-1,1.5);
\draw[very thick, Red!60] (0,0) -- (1,1)--(1,1.5);
\draw[very thick, MidnightBlue] (-1,1.5)--(-1,2) ;
\draw[very thick, MidnightBlue] (0,1.5)--(0,2) ;
\draw[very thick, MidnightBlue] (1,1.5)--(1,2) ;
\draw (-1.5,1.5) --(1.5, 1.5);
\end{tikzpicture}}}}

\newcommand{\TreeCab}
{\vcenter{\hbox{
\begin{tikzpicture}[scale=0.25]
\draw[very thick, MidnightBlue] (0,0) -- (0,1.5);
\draw[very thick, MidnightBlue] (0,0) -- (-1,1)--(-1,1.5);
\draw[very thick, MidnightBlue] (0,0) -- (1,1)--(1,1.5);
\draw[very thick, Red!60] (0,-0.5)--(0,0) ;
\draw (-0.5,0) --(0.5, 0);
\end{tikzpicture}}}}

\newcommand{\TreeIab}
{\vcenter{\hbox{
\begin{tikzpicture}[scale=0.35]
\draw[very thick, MidnightBlue] (0,0.5)--(0,0);
\draw[very thick, Red!60] (0,-0.5)--(0,0) ;
\draw (-0.25,0) --(0.25, 0);
\end{tikzpicture}}}}

\newcommand{\TreeBa}
{\vcenter{\hbox{
\begin{tikzpicture}[scale=0.25]
\draw[very thick, MidnightBlue] (0,0.5)--(0,1);
\draw[very thick, MidnightBlue] (0,1)--(-0.5,1.5)--(-0.5,2);
\draw[very thick, MidnightBlue] (0,1)--(0.5,1.5)--(0.5,2);
\draw[very thick, Red!60] (0,0)--(0,0.5);
\draw (-0.5,0.5) --(0.5, 0.5);\end{tikzpicture}}}}

\newcommand{\TreeBb}
{\vcenter{\hbox{
\begin{tikzpicture}[scale=0.25]
\draw[very thick, MidnightBlue] (-0.5,2)--(-0.5,2.5);
\draw[very thick, MidnightBlue] (0.5,2)--(0.5,2.5);
\draw[very thick, Red!60] (0,0.5)--(0,1);
\draw[very thick, Red!60] (0,1)--(-0.5,1.5)--(-0.5,2);
\draw[very thick, Red!60] (0,1)--(0.5,1.5)--(0.5,2);
\draw (-1,2) --(1, 2);
\end{tikzpicture}}}}

\usepackage{graphicx,calc}
\newlength\myheight
\newlength\mydepth
\settototalheight\myheight{Xygp}
\newcommand*\inlinegraphics[1]{%
  \settototalheight\myheight{Xygp}%
  \settodepth\mydepth{Xygp}%
  \raisebox{-\mydepth}{\includegraphics[height=\myheight]{#1}}%
}

\DeclareMathOperator{\Ima}{Im} 
\DeclareMathOperator{\cone}{Cone} 

\newcommand{\End}{\ensuremath{\mathrm{End}}}
\newcommand{\coEnd}{\ensuremath{\mathrm{coEnd}}}
\newcommand{\Hom}{\ensuremath{\mathrm{Hom}}}
\newcommand{\coHom}{\ensuremath{\mathrm{coHom}}}

\newcommand{\arbreoptrois}{
\begin{tikzpicture}[scale=0.15,baseline=(pt_base.base)]

\coordinate (pt_base) at (0,-0.5) ;

\draw (0,1.25) -- (0,0) ;
\draw (-1.25,1.25) -- (0,0) ;
\draw (1.25,1.25) -- (0,0) ;
\draw (0,-1.25) -- (0,0) ;

\end{tikzpicture}
}

\newcommand{\premiertermecobarbarA}{
\begin{tikzpicture}[scale=0.15,baseline=(pt_base.base)]

\coordinate (pt_base) at (0,-0.5) ;

\draw (-1.25,1.25) -- (0,0) ;
\draw (1.25,1.25) -- (0,0) ;
\draw (0,-1.25) -- (0,0) ;

\end{tikzpicture}
}

\newcommand{\premiertermecobarbarB}{
\begin{tikzpicture}[scale = 0.15,baseline = (O.base)]
\node (K) at (-0.625,0.625) {} ;
\node (L) at (0,0) {} ;
\draw (-0.625,0.625) -- (0,1.25) ;
\draw (-1.25,1.25) -- (0,0) ;
\draw (1.25,1.25) -- (0,0) ;
\draw (0,-1.25) -- (0,0) ;
\node (O) at (0,-0.5) {} ;
\end{tikzpicture}
}

\newcommand{\premiertermecobarbarC}{
\begin{tikzpicture}[scale = 0.15,baseline = (O.base)]
\node (K) at (0.625,0.625) {} ;
\node (L) at (0,0) {} ;
\draw (0.625,0.625) -- (0,1.25) ;
\draw (-1.25,1.25) -- (0,0) ;
\draw (1.25,1.25) -- (0,0) ;
\draw (0,-1.25) -- (0,0) ;
\node (O) at (0,-0.5) {} ;
\end{tikzpicture}
}

\newcommand{\premiertermecobarbarD}{
\begin{tikzpicture}[scale=0.15,baseline=(pt_base.base)]

\coordinate (pt_base) at (0,-0.5) ;

\draw (0,1.25) -- (0,0) ;
\draw (-1.25,1.25) -- (0,0) ;
\draw (1.25,1.25) -- (0,0) ;
\draw (0,-1.25) -- (0,0) ;

\end{tikzpicture}
}

\newcommand{\arbreopdeuxcol}[1]{
\begin{tikzpicture}[scale=0.15,baseline=(pt_base.base)]

\coordinate (pt_base) at (0,-0.5) ;

\draw[#1] (-1.25,1.25) -- (0,0) ;
\draw[#1] (1.25,1.25) -- (0,0) ;
\draw[#1] (0,-1.25) -- (0,0) ;

\end{tikzpicture}
}

\newcommand{\arbreoptroiscol}[1]{
\begin{tikzpicture}[scale=0.15,baseline=(pt_base.base)]

\coordinate (pt_base) at (0,-0.5) ;

\draw[#1] (0,1.25) -- (0,0) ;
\draw[#1] (-1.25,1.25) -- (0,0) ;
\draw[#1] (1.25,1.25) -- (0,0) ;
\draw[#1] (0,-1.25) -- (0,0) ;

\end{tikzpicture}
}

\usetikzlibrary{decorations.pathreplacing,calligraphy}
\usetikzlibrary{decorations.pathmorphing}

\newcommand{\arbreoprouge}[1]{
\begin{tikzpicture}[scale=#1,baseline=(pt_base.base)]

\coordinate (pt_base) at (0,-0.5) ;

\draw[Red!60] (-2,1.25) -- (0,0) ;
\draw[Red!60] (-1,1.25) -- (0,0) ;
\draw[Red!60] (1,1.25) -- (0,0) ;
\draw[Red!60] (2,1.25) -- (0,0) ;
\draw[Red!60] (0,-1.25) -- (0,0) ;

\draw[Red!60] [densely dotted] (-0.5,0.9) -- (0.5,0.9) ;

\end{tikzpicture}}

\newcommand{\arbreopbleu}[1]{
\begin{tikzpicture}[scale=#1,baseline=(pt_base.base)]

\coordinate (pt_base) at (0,-0.5) ;

\draw[MidnightBlue] (-2,1.25) -- (0,0) ;
\draw[MidnightBlue] (-1,1.25) -- (0,0) ;
\draw[MidnightBlue] (1,1.25) -- (0,0) ;
\draw[MidnightBlue] (2,1.25) -- (0,0) ;
\draw[MidnightBlue] (0,-1.25) -- (0,0) ;

\draw[MidnightBlue] [densely dotted] (-0.5,0.9) -- (0.5,0.9) ;

\end{tikzpicture}}

\newcommand{\arbreopmorph}[1]{
\begin{tikzpicture}[scale=#1,baseline=(pt_base.base)]

\coordinate (pt_base) at (0,-0.5) ;

\draw[MidnightBlue] (-2,1.25) -- (0,0) ;
\draw[MidnightBlue] (-1,1.25) -- (0,0) ;
\draw[MidnightBlue] (1,1.25) -- (0,0) ;
\draw[MidnightBlue] (2,1.25) -- (0,0) ;

\draw[Red!60] (0,-1.25) -- (0,0) ;

\draw[MidnightBlue] [densely dotted] (-0.5,0.9) -- (0.5,0.9) ;

\draw  (-1.5,0) -- (1.5,0) ;

\end{tikzpicture}}

\newcommand{\arbreopmorphcompun}{
\begin{tikzpicture}[scale=0.15,baseline=(pt_base.base)]

\coordinate (pt_base) at (0,-0.5) ;

\draw[MidnightBlue] (-2,1.25) -- (0,0) ;
\draw[MidnightBlue] (-1,1.25) -- (0,0) ;
\draw[MidnightBlue] (1,1.25) -- (0,0) ;
\draw[MidnightBlue] (2,1.25) -- (0,0) ;

\draw[OliveGreen!80] (0,-1.25) -- (0,0) ;

\draw[MidnightBlue] [densely dotted] (-0.5,0.9) -- (0.5,0.9) ;

\draw  (-1.5,0) -- (1.5,0) ;

\end{tikzpicture}}

\newcommand{\arbreopmorphcompdeux}{
\begin{tikzpicture}[scale=0.15,baseline=(pt_base.base)]

\coordinate (pt_base) at (0,-0.5) ;

\draw[OliveGreen!80] (-2,1.25) -- (0,0) ;
\draw[OliveGreen!80] (-1,1.25) -- (0,0) ;
\draw[OliveGreen!80] (1,1.25) -- (0,0) ;
\draw[OliveGreen!80] (2,1.25) -- (0,0) ;

\draw[Red!60] (0,-1.25) -- (0,0) ;

\draw[OliveGreen!80] [densely dotted] (-0.5,0.9) -- (0.5,0.9) ;

\draw  (-1.5,0) -- (1.5,0) ;

\end{tikzpicture}}

\newcommand{\arbreopdeuxmorph}{
\begin{tikzpicture}[scale=0.15,baseline=(pt_base.base)]

\coordinate (pt_base) at (0,-0.5) ;

\draw[MidnightBlue] (-1.25,1.25) -- (0,0) ;
\draw[MidnightBlue] (1.25,1.25) -- (0,0) ;
\draw[Red!60] (0,-1.25) -- (0,0) ;

\draw (-1.25,0) -- (1.25,0) ;

\end{tikzpicture}
}

\newcommand{\arbreopdeuxmorphcompun}{
\begin{tikzpicture}[scale=0.15,baseline=(pt_base.base)]

\coordinate (pt_base) at (0,-0.5) ;

\draw[MidnightBlue] (-1.25,1.25) -- (0,0) ;
\draw[MidnightBlue] (1.25,1.25) -- (0,0) ;
\draw[OliveGreen!80] (0,-1.25) -- (0,0) ;

\draw (-1.25,0) -- (1.25,0) ;

\end{tikzpicture}
}

\newcommand{\arbreopdeuxmorphcompdeux}{
\begin{tikzpicture}[scale=0.15,baseline=(pt_base.base)]

\coordinate (pt_base) at (0,-0.5) ;

\draw[OliveGreen!80] (-1.25,1.25) -- (0,0) ;
\draw[OliveGreen!80] (1.25,1.25) -- (0,0) ;
\draw[Red!60] (0,-1.25) -- (0,0) ;

\draw (-1.25,0) -- (1.25,0) ;

\end{tikzpicture}
}

\newcommand{\arbreopcompun}{
\begin{tikzpicture}[scale=0.15,baseline=(pt_base.base)]

\coordinate (pt_base) at (0,-2) ;

\draw[MidnightBlue] (-1.25,1.25) -- (0,0) ;
\draw[MidnightBlue] (1.25,1.25) -- (0,0) ;
\draw[OliveGreen!80] (0,-2.5) -- (0,0) ;
\draw[Red!60] (0,-3.75) -- (0,-2.5) ;

\draw (-1.25,0) -- (1.25,0) ;

\draw (-1.25,-2.5) -- (1.25,-2.5) ;

\end{tikzpicture}
}

\newcommand{\arbreopcompdeux}{
\begin{tikzpicture}[scale=0.15,baseline=(pt_base.base)]

\coordinate (pt_base) at (0,0.5) ;

\draw[MidnightBlue] (-1.25,3.75) -- (-1.25,2.5) ;
\draw[MidnightBlue] (1.25,3.75) -- (1.25,2.5) ;
\draw[OliveGreen!80] (-1.25,1.25) -- (-1.25,2.5) ;
\draw[OliveGreen!80] (1.25,1.25) -- (1.25,2.5) ;
\draw[OliveGreen!80] (-1.25,1.25) -- (0,0) ;
\draw[OliveGreen!80] (1.25,1.25) -- (0,0) ;
\draw[Red!60] (0,-1.25) -- (0,0) ;

\draw (-1.25,0) -- (1.25,0) ;
\draw (-1.75,2.5) -- (-0.75,2.5) ;
\draw (1.75,2.5) -- (0.75,2.5) ;

\end{tikzpicture}
}

\newcommand{\arbreoptroismorph}{
\begin{tikzpicture}[scale=0.15,baseline=(pt_base.base)]

\coordinate (pt_base) at (0,-0.5) ;

\draw[MidnightBlue] (0,1.25) -- (0,0) ;
\draw[MidnightBlue] (-1.25,1.25) -- (0,0) ;
\draw[MidnightBlue] (1.25,1.25) -- (0,0) ;

\draw[Red!60] (0,-1.25) -- (0,0) ;

\draw (-1.25,0) -- (1.25,0) ;

\end{tikzpicture}
}

\newcommand{\arbreoptroismorphcompun}{
\begin{tikzpicture}[scale=0.15,baseline=(pt_base.base)]

\coordinate (pt_base) at (0,-0.5) ;

\draw[MidnightBlue] (0,1.25) -- (0,0) ;
\draw[MidnightBlue] (-1.25,1.25) -- (0,0) ;
\draw[MidnightBlue] (1.25,1.25) -- (0,0) ;

\draw[OliveGreen!80] (0,-1.25) -- (0,0) ;

\draw (-1.25,0) -- (1.25,0) ;

\end{tikzpicture}
}

\newcommand{\arbreoptroismorphcompdeux}{
\begin{tikzpicture}[scale=0.15,baseline=(pt_base.base)]

\coordinate (pt_base) at (0,-0.5) ;

\draw[OliveGreen!80] (0,1.25) -- (0,0) ;
\draw[OliveGreen!80] (-1.25,1.25) -- (0,0) ;
\draw[OliveGreen!80] (1.25,1.25) -- (0,0) ;

\draw[Red!60] (0,-1.25) -- (0,0) ;

\draw (-1.25,0) -- (1.25,0) ;

\end{tikzpicture}
}

\newcommand{\arbreopquatremorph}{
\begin{tikzpicture}[scale=0.15,baseline=(pt_base.base)]

\coordinate (pt_base) at (0,-0.5) ;

\draw[MidnightBlue] (-1.5,1.25) -- (0,0) ;
\draw[MidnightBlue] (1.5,1.25) -- (0,0) ;
\draw[MidnightBlue] (-0.5,1.25) -- (0,0) ;
\draw[MidnightBlue] (0.5,1.25) -- (0,0) ;

\draw[Red!60] (0,-1.25) -- (0,0) ;

\draw (-1.5,0) -- (1.5,0) ;

\end{tikzpicture}
}

\newcommand{\arbreopquatremorphcompun}{
\begin{tikzpicture}[scale=0.15,baseline=(pt_base.base)]

\coordinate (pt_base) at (0,-0.5) ;

\draw[MidnightBlue] (-1.5,1.25) -- (0,0) ;
\draw[MidnightBlue] (1.5,1.25) -- (0,0) ;
\draw[MidnightBlue] (-0.5,1.25) -- (0,0) ;
\draw[MidnightBlue] (0.5,1.25) -- (0,0) ;

\draw[OliveGreen!80] (0,-1.25) -- (0,0) ;

\draw (-1.5,0) -- (1.5,0) ;

\end{tikzpicture}
}

\newcommand{\arbreopquatremorphcompdeux}{
\begin{tikzpicture}[scale=0.15,baseline=(pt_base.base)]

\coordinate (pt_base) at (0,-0.5) ;

\draw[OliveGreen!80] (-1.5,1.25) -- (0,0) ;
\draw[OliveGreen!80] (1.5,1.25) -- (0,0) ;
\draw[OliveGreen!80] (-0.5,1.25) -- (0,0) ;
\draw[OliveGreen!80] (0.5,1.25) -- (0,0) ;

\draw[Red!60] (0,-1.25) -- (0,0) ;

\draw (-1.5,0) -- (1.5,0) ;

\end{tikzpicture}
}

\newcommand{\eqainfmorphun}{
\begin{tikzpicture}[scale=0.2,baseline=(pt_base.base)]

\coordinate (pt_base) at (0,1.25) ;

\draw [decorate,    decoration = {calligraphic brace}] (-4,1.5) --  (-0.5,1.5) ;
\draw [decorate,    decoration = {calligraphic brace}] (0.5,1.5) --  (4,1.5) ;
\draw [decorate,    decoration = {calligraphic brace}] (-2,4) --  (2,4) ;

\node[above] at (-2.3,1.5) {\tiny $p$} ; 
\node[above] at (0,4) {\tiny $q$} ; 
\node[above] at (2.3,1.5) {\tiny $r$} ; 

\draw[MidnightBlue] (-4,1.25) -- (0,0) ;
\draw[MidnightBlue] (-3,1.25) -- (0,0) ;
\draw[MidnightBlue] (4,1.25) -- (0,0) ;
\draw[MidnightBlue] (3,1.25) -- (0,0) ;
\draw[MidnightBlue] (0,1.25) -- (0,0) ;

\draw[Red!60] (0,-1.25) -- (0,0) ; 

\draw[MidnightBlue] (-2,3.75) -- (0,2.5) ;
\draw[MidnightBlue] (-1,3.75) -- (0,2.5) ;
\draw[MidnightBlue] (1,3.75) -- (0,2.5) ;
\draw[MidnightBlue] (2,3.75) -- (0,2.5) ;
\draw[MidnightBlue] (0,1.25) -- (0,2.5) ;

\draw[MidnightBlue] [densely dotted] (-0.5,3.65) -- (0.5,3.65) ;
\draw[MidnightBlue] [densely dotted] (-0.5,0.9) -- (-1.5,0.9) ;
\draw[MidnightBlue] [densely dotted] (0.5,0.9) -- (1.5,0.9) ;

\draw (-2,0) -- (2,0) ;

\end{tikzpicture}}

\newcommand{\compainf}{
\begin{tikzpicture}[scale=0.2,baseline=(pt_base.base)]

\coordinate (pt_base) at (0,1.25) ;

\draw [decorate,    decoration = {calligraphic brace}] (-5.5,4) --  (-2.5,4) ;
\draw [decorate,    decoration = {calligraphic brace}] (2.5,4) --  (5.5,4) ;

\node[above] at (-4,4) {\tiny $i_1$} ; 
\node[above] at (4,4) {\tiny $i_k$} ;

\draw[OliveGreen!80] (-4,1.25) -- (0,0) ;
\draw[OliveGreen!80] (-2,1.25) -- (0,0) ;
\draw[OliveGreen!80] (4,1.25) -- (0,0) ;
\draw[OliveGreen!80] (2,1.25) -- (0,0) ;
\draw[Red!60] (0,-1.25) -- (0,0) ; 

\draw (-3,0) -- (3,0) ;

\draw[MidnightBlue] (2.5,3.75) -- (4,2.5) ;
\draw[MidnightBlue] (3.5,3.75) -- (4,2.5) ;
\draw[MidnightBlue] (4.5,3.75) -- (4,2.5) ;
\draw[MidnightBlue] (5.5,3.75) -- (4,2.5) ;
\draw[OliveGreen!80] (4,1.25) -- (4,2.5) ;

\draw (3,2.5) -- (5,2.5) ;

\draw[MidnightBlue] (-2.5,3.75) -- (-4,2.5) ;
\draw[MidnightBlue] (-3.5,3.75) -- (-4,2.5) ;
\draw[MidnightBlue] (-4.5,3.75) -- (-4,2.5) ;
\draw[MidnightBlue] (-5.5,3.75) -- (-4,2.5) ;
\draw[OliveGreen!80] (-4,1.25) -- (-4,2.5) ;

\draw (-3,2.5) -- (-5,2.5) ;

\draw [densely dotted,OliveGreen!80] (-1,0.9) -- (1,0.9) ;
\draw [densely dotted] (2,2.75) -- (-2,2.75) ; 
\draw [densely dotted,MidnightBlue] (3.8,3.7) -- (4.2,3.7) ;
\draw [densely dotted,MidnightBlue] (-3.8,3.7) -- (-4.2,3.7) ;

\end{tikzpicture}}

\newcommand{\eqainfmorphdeux}{
\begin{tikzpicture}[scale=0.2,baseline=(pt_base.base)]

\coordinate (pt_base) at (0,1.25) ;

\draw [decorate,    decoration = {calligraphic brace}] (-5.5,4) --  (-2.5,4) ;
\draw [decorate,    decoration = {calligraphic brace}] (2.5,4) --  (5.5,4) ;

\node[above] at (-4,4) {\tiny $i_1$} ; 
\node[above] at (4,4) {\tiny $i_k$} ;

\draw[Red!60] (-4,1.25) -- (0,0) ;
\draw[Red!60] (-2,1.25) -- (0,0) ;
\draw[Red!60] (4,1.25) -- (0,0) ;
\draw[Red!60] (2,1.25) -- (0,0) ;
\draw[Red!60] (0,-1.25) -- (0,0) ; 

\draw[MidnightBlue] (2.5,3.75) -- (4,2.5) ;
\draw[MidnightBlue] (3.5,3.75) -- (4,2.5) ;
\draw[MidnightBlue] (4.5,3.75) -- (4,2.5) ;
\draw[MidnightBlue] (5.5,3.75) -- (4,2.5) ;
\draw[Red!60] (4,1.25) -- (4,2.5) ;

\draw (3,2.5) -- (5,2.5) ;

\draw[MidnightBlue] (-2.5,3.75) -- (-4,2.5) ;
\draw[MidnightBlue] (-3.5,3.75) -- (-4,2.5) ;
\draw[MidnightBlue] (-4.5,3.75) -- (-4,2.5) ;
\draw[MidnightBlue] (-5.5,3.75) -- (-4,2.5) ;
\draw[Red!60] (-4,1.25) -- (-4,2.5) ;

\draw (-3,2.5) -- (-5,2.5) ;

\draw [densely dotted,Red!60] (-1,0.9) -- (1,0.9) ;
\draw [densely dotted] (2,2.75) -- (-2,2.75) ; 
\draw [densely dotted,MidnightBlue] (3.8,3.7) -- (4.2,3.7) ;
\draw [densely dotted,MidnightBlue] (-3.8,3.7) -- (-4.2,3.7) ;

\end{tikzpicture}}

\newcommand{\arbreopunmorph}{
\begin{tikzpicture}[scale=0.15,baseline=(pt_base.base)]

\coordinate (pt_base) at (0,-0.5) ;

\draw[MidnightBlue] (0,1.25) -- (0,0) ;
\draw[Red!60] (0,-1.25) -- (0,0) ;

\draw (-1.25,0) -- (1.25,0) ;

\end{tikzpicture}
}

\newcommand{\arbreopunmorphcompun}{
\begin{tikzpicture}[scale=0.15,baseline=(pt_base.base)]

\coordinate (pt_base) at (0,-0.5) ;

\draw[MidnightBlue] (0,1.25) -- (0,0) ;
\draw[OliveGreen!80] (0,-1.25) -- (0,0) ;

\draw (-1.25,0) -- (1.25,0) ;

\end{tikzpicture}
}

\newcommand{\arbreopunmorphcompdeux}{
\begin{tikzpicture}[scale=0.15,baseline=(pt_base.base)]

\coordinate (pt_base) at (0,-0.5) ;

\draw[OliveGreen!80] (0,1.25) -- (0,0) ;
\draw[Red!60] (0,-1.25) -- (0,0) ;

\draw (-1.25,0) -- (1.25,0) ;

\end{tikzpicture}
}

\newcommand{\thickmidarrow}{
\begin{tikzpicture}

\draw[ultra thick] (-2,1) -- (0,0) ;
\draw[ultra thick] (-2,-1) -- (0,0) ;
\end{tikzpicture}}

\newcommand{\eqainf}{
\begin{tikzpicture}[scale=0.2,baseline=(pt_base.base)]

\coordinate (pt_base) at (0,1.25) ;

\draw [decorate,    decoration = {calligraphic brace}] (-4,1.5) --  (-0.5,1.5) ;
\draw [decorate,    decoration = {calligraphic brace}] (0.5,1.5) --  (4,1.5) ;
\draw [decorate,    decoration = {calligraphic brace}] (-2,4) --  (2,4) ;

\node[above] at (-2.3,1.5) {\tiny $p$} ; 
\node[above] at (0,4) {\tiny $q$} ; 
\node[above] at (2.3,1.5) {\tiny $r$} ; 

\draw (-4,1.25) -- (0,0) ;
\draw (-3,1.25) -- (0,0) ;
\draw (4,1.25) -- (0,0) ;
\draw (3,1.25) -- (0,0) ;
\draw (0,1.25) -- (0,0) ;
\draw(0,-1.25) -- (0,0) ; 

\draw (-2,3.75) -- (0,2.5) ;
\draw (-1,3.75) -- (0,2.5) ;
\draw (1,3.75) -- (0,2.5) ;
\draw (2,3.75) -- (0,2.5) ;
\draw (0,1.25) -- (0,2.5) ;

\draw [densely dotted] (-0.5,3.65) -- (0.5,3.65) ;
\draw [densely dotted] (-0.5,0.9) -- (-1.5,0.9) ;
\draw [densely dotted] (0.5,0.9) -- (1.5,0.9) ;

\end{tikzpicture}}

\newcommand{\arbreopquatrecol}[1]{
\begin{tikzpicture}[scale=0.15,baseline=(pt_base.base)]

\coordinate (pt_base) at (0,-0.5) ;

\draw[#1] (-1.5,1.25) -- (0,0) ;
\draw[#1] (1.5,1.25) -- (0,0) ;
\draw[#1] (-0.5,1.25) -- (0,0) ;
\draw[#1] (0.5,1.25) -- (0,0) ;
\draw[#1] (0,-1.25) -- (0,0) ;

\end{tikzpicture}
}

\newcommand{\arbreop}[1]{
\begin{tikzpicture}[scale=#1,baseline=(pt_base.base)]

\coordinate (pt_base) at (0,-0.5) ;

\draw (-2,1.25) -- (0,0) ;
\draw (-1,1.25) -- (0,0) ;
\draw (1,1.25) -- (0,0) ;
\draw (2,1.25) -- (0,0) ;
\draw (0,-1.25) -- (0,0) ;

\draw [densely dotted] (-0.5,0.9) -- (0.5,0.9) ;

\end{tikzpicture}}

\newcommand{\arbreopdeux}{
\begin{tikzpicture}[scale=0.15,baseline=(pt_base.base)]

\coordinate (pt_base) at (0,-0.5) ;

\draw (-1.25,1.25) -- (0,0) ;
\draw (1.25,1.25) -- (0,0) ;
\draw (0,-1.25) -- (0,0) ;

\end{tikzpicture}
}

\newcommand{\arbreopquatre}{
\begin{tikzpicture}[scale=0.15,baseline=(pt_base.base)]

\coordinate (pt_base) at (0,-0.5) ;

\draw (-1.5,1.25) -- (0,0) ;
\draw (1.5,1.25) -- (0,0) ;
\draw (-0.5,1.25) -- (0,0) ;
\draw (0.5,1.25) -- (0,0) ;
\draw (0,-1.25) -- (0,0) ;

\end{tikzpicture}
}

\newcommand{\exampleleftlevelwisetwo}{
\begin{tikzpicture}[scale=0.3,baseline=(pt_base.base)]

\coordinate (pt_base) at (0,-0.5) ;

\draw (-1.5,1.25) -- (0,0) ;
\draw (1.5,1.25) -- (0,0) ;
\draw (-0.5,1.25) -- (0,0) ;
\draw (0.5,1.25) -- (0,0) ;
\draw (0,-1.25) -- (0,0) ;

\begin{scope}[xshift = -1.5cm , yshift = 2.5cm]
\draw (0,1.25) -- (0,0) ;
\draw (-1.25,1.25) -- (0,0) ;
\draw (1.25,1.25) -- (0,0) ;
\draw (0,-1.25) -- (0,0) ;
\end{scope}

\begin{scope}[xshift = 1.5cm , yshift = 2.5cm]
\draw (-1.5,1.25) -- (0,0) ;
\draw (1.5,1.25) -- (0,0) ;
\draw (-0.5,1.25) -- (0,0) ;
\draw (0.5,1.25) -- (0,0) ;
\draw (0,-1.25) -- (0,0) ;
\end{scope}

\end{tikzpicture}
}

\newcommand{\exampleleftlevelwiseone}{
\begin{tikzpicture}[scale=0.3,baseline=(pt_base.base)]

\coordinate (pt_base) at (0,-0.5) ;

\draw (-1.5,1.25) -- (0,0) ;
\draw (1.5,1.25) -- (0,0) ;
\draw (-0.5,1.25) -- (0,0) ;
\draw (0.5,1.25) -- (0,0) ;
\draw (0,-1.25) -- (0,0) ;

\begin{scope}[xshift = 0.5cm , yshift = 2.5cm]
\draw (-1.5,1.25) -- (0,0) ;
\draw (1.5,1.25) -- (0,0) ;
\draw (-0.5,1.25) -- (0,0) ;
\draw (0.5,1.25) -- (0,0) ;
\draw (0,-1.25) -- (0,0) ;
\end{scope}

\begin{scope}[xshift = -1cm , yshift = 5cm]
\draw (0,1.25) -- (0,0) ;
\draw (-1.25,1.25) -- (0,0) ;
\draw (1.25,1.25) -- (0,0) ;
\draw (0,-1.25) -- (0,0) ;
\end{scope}

\end{tikzpicture}
}

\title[The diagonal of the multiplihedra]{The diagonal of the multiplihedra and \\ the tensor product of $\Ainf$-morphisms}

\author{Guillaume Laplante-Anfossi}
\address{School of Mathematics and Statistics, University of Melbourne, Victoria, Australia}
\email{guillaume.laplanteanfossi@unimelb.edu.au}

\author{Thibaut Mazuir}
\address{Institut für Mathematik, Humboldt Universität zu Berlin, Germany.}
\email{thibaut.mazuir@hu-berlin.de}

 \date{\today}



 \thanks{The first author was supported by the European Union's Horizon 2020 research and innovation program under the Marie Sklodowska-Curie grant agreement No 754362 \inlinegraphics{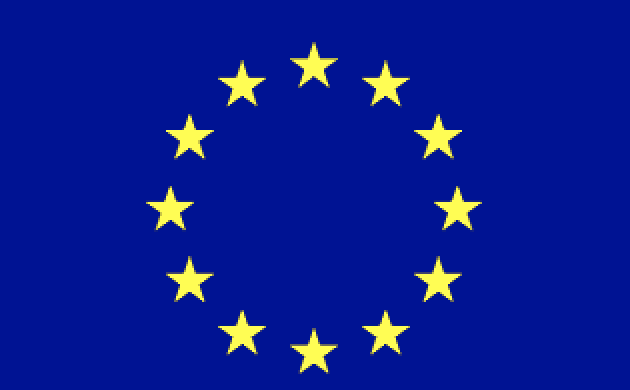}, by the Natural Sciences and Engineering Research Council of Canada (NSERC) and by the ANR-20-CE40-0016 (Higher Algebra, Geometry and Topology). The second author was supported by the ANR-21-CE40-0002 (New challenges in Symplectic and Contact Topology).}
 
\begin{document}

\begin{abstract}
We define a cellular approximation for the diagonal of the Forcey--Loday realizations of the multiplihedra, and endow them with a compatible topological cellular operadic bimodule structure over the Loday realizations of the associahedra. 
This provides us with a model for topological and algebraic \Ainf -morphisms, as well as a universal and explicit formula for their tensor product.
We study the monoidal properties of this newly defined tensor product and conclude by outlining several applications, notably in algebraic and symplectic topology.
\end{abstract}

\maketitle

\begin{figure}[h!]
\centering
\includegraphics[width=0.6\linewidth]{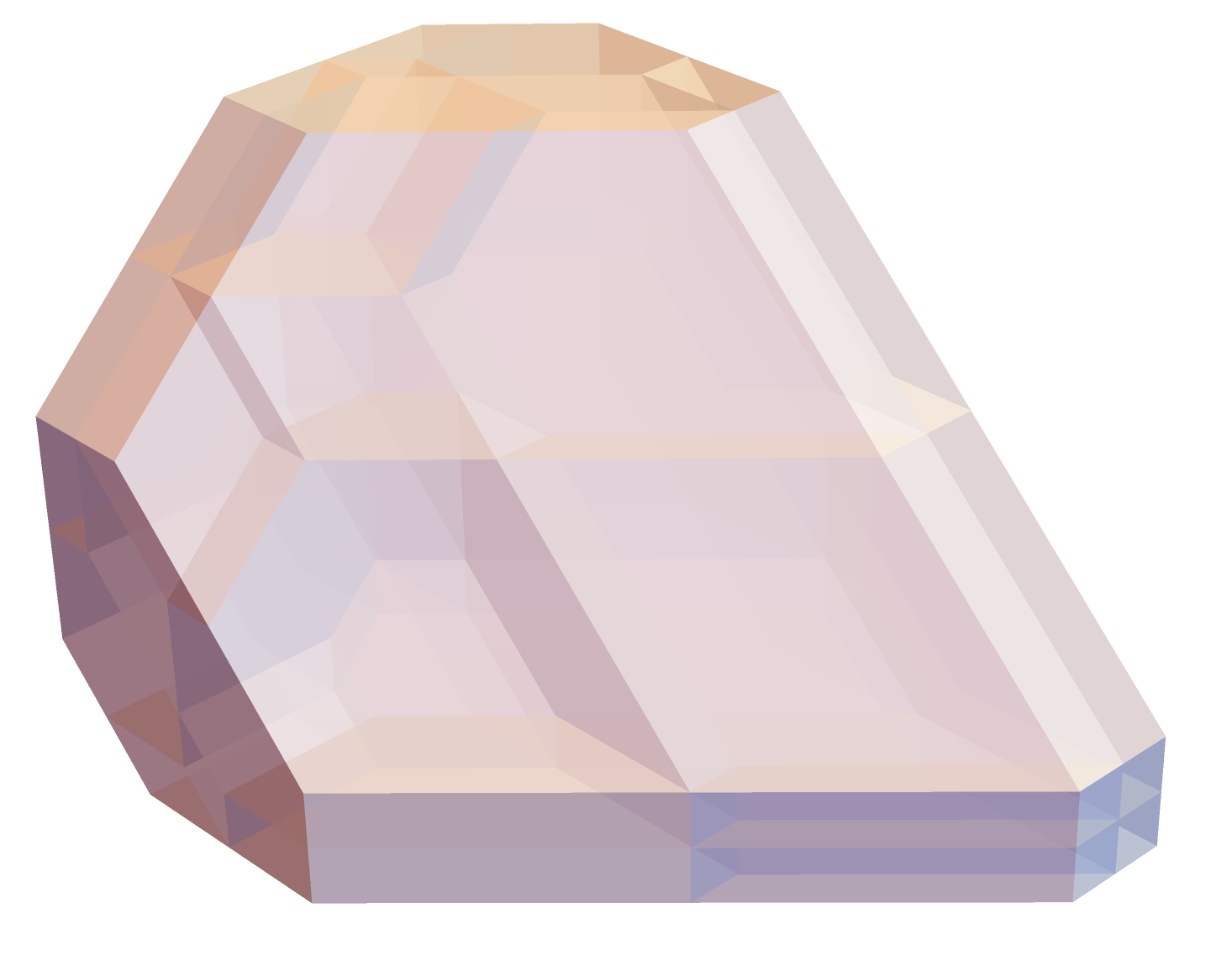} 
\label{Fig5:J4}
\end{figure}

\newpage

\setcounter{tocdepth}{1}
\tableofcontents


\section*{Introduction}

The $n$-dimensional associahedron, a polytope whose faces are in bijection with planar trees with $n+2$ leaves, was first introduced as a topological cell complex by J. Stasheff to describe algebras whose product is associative up to homotopy \cite{Stasheff63}.
The problem of giving polytopal realizations of these CW-complexes has a rich history \cite{CeballosZiegler12}, and the algebras that they encode, called $\mathrm{A}_\infty$-algebras, have been extensively studied in various branches of mathematics. They were used in algebraic topology for the study of iterated loop spaces \cite{May72,BoardmanVogt73} or the study of homotopy theory of differential graded associative algebras \cite{LefevreHasegawa03,Vallette14} ; in symplectic topology to define Fukaya categories of symplectic manifolds \cite{Seidel08,fo3-I,fo3-II}, through the interpretation of the associahedra as moduli spaces of disks with marked boundary points; and more recently, in mathematical physics, mirror symmetry, Galois cohomology or non-commutative probability.

The $n$-dimensional multiplihedron is a polytope whose faces are in bijection with 2-colored planar trees with $n+1$ leaves. It was first introduced as a topological cell complex by J. Stasheff to describe morphisms between $\mathrm{A}_\infty$-algebras \cite{Stasheff70}.
It was only recently realized as a convex polytope in the work of S. Forcey \cite{Forcey08}, followed by the work of S. Forcey and S. Devadoss \cite{DevadossForcey08}, F. Ardila and J. Doker \cite{AD13}, and F. Chapoton and V. Pilaud \cite{CP22}.
The multiplihedra were studied in algebraic topology \cite{BoardmanVogt73}, as well as in symplectic topology \cite{MauWoodward10,mau-wehrheim-woodward} and Morse theory \cite{mazuir-I,mazuir-II}, as they can be respectively realized as moduli spaces of quilted disks with marked boundary points and as moduli spaces of 2-colored metric trees. 

In this paper, we define and study a cellular approximation of the diagonal of the multiplihedra. 
The need for such an approximation comes from the fact that the standard thin diagonal $\triangle_P:P\to P\times P, x\mapsto (x,x)$ of a polytope $P$ is not cellular in general, i.e. its image is not a union of faces of $P\times P$. 
A cellular approximation of the diagonal is a cellular map $\triangle_P^{\textrm{cell}} : P \to P\times P$ which is homotopic to $\triangle_P$ and which agrees with $\triangle_P$ on the vertices of $P$.

The Alexander--Whitney map \cite{EilenbergMacLane53} and the Serre diagonal \cite{Serre51} respectively define cellular approximations for the diagonal of the simplices and for the diagonal of the cubes, yielding the cup product in singular cohomology and the cup product in cubical cohomology.
A cellular approximation for the diagonal of the associahedra was first defined at the level of cellular chains in \cite{SaneblidzeUmble04}. 
A second construction was given in \cite{MarklShnider06}, and was recently shown to coincide with the first one \cite{saneblidzeComparingDiagonalsAssociahedra2022a}. 
A topological map was given for the first time in \cite{MTTV19} and was shown to recover the previous constructions at the cellular level. 
Such a diagonal yields a universal formula for the tensor product of two $\Ainf$-algebras.
By the term \textit{universal}, we mean that the same formula applies uniformly to any pair of $\Ainf$-algebras.
In a similar fashion, a cellular approximation of the diagonal of the multiplihedra will be used to define a universal tensor product of $\Ainf$-morphisms in this paper.
Such a map was given at the level of chains in \cite{SaneblidzeUmble04}. We provide the first topological map, which is moreover distinct from the one of \cite{SaneblidzeUmble04} at the cellular level, see \cref{ss:about} for more details on this point.
Our main results can be summarized as follows.
\begin{enumerate}
  \item We define a cellular approximation of the diagonal on Forcey--Loday realizations of the multiplihedra (\cref{def:diagonal-multipl-forcey-loday}).
  \item We endow them with a compatible operadic bimodule structure over the Loday realizations of the associahedra (\cref{thm:MainOperad}).
  \item We compute explicitly the associated combinatorial formula for the cellular image of the diagonal (\cref{thm:formuladiagonal}).
  \item We apply the cellular chains functor to the diagonal in order to define a universal tensor product of $\mathrm{A}_\infty$-morphisms (\cref{prop:diagonal-polytopale-m-infini}), and we study its properties (\cref{ss:homotopy-properties}).
\end{enumerate}

To achieve these goals, we use the theory of cellular approximations of diagonals developed by the first author in \cite{LA21}, which is based on the theory of fiber polytopes of \cite{BilleraSturmfels92} and the method introduced in \cite{MTTV19}.
We prove that the Forcey--Loday realizations of the multiplihedra \cite{Forcey08} can be obtained from the Ardila--Doker realization of the multiplihedra \cite{AD13} by projection (\cref{prop:lifting}).
These last realizations are generalized permutahedra, in the sense of A. Postnikov \cite{Postnikov09}, which allows us to apply the results of \cite{LA21} directly, both to define a cellular approximation of the diagonal and to describe its cellular image combinatorially.

The tensor product of $\Ainf$-morphisms defined by this diagonal does not however define a symmetric monoidal structure on the category $\infAalg$ of $\Ainf$-algebras and their $\Ainf$-morphisms, since it is not strictly compatible with the composition. 
This is not a defect of our construction: in \cref{thm:nofunctorial}, we prove that there is no tensor product of $\Ainf$-morphisms which is strictly compatible with the composition of $\Ainf$-morphisms. 
This proposition should be compared to a similar result by M. Markl and S. Shnider, saying that there is no strictly associative tensor product of $\Ainf$-algebras \cite[Theorem 13]{MarklShnider06}.
The preceding two properties are in fact always satisfied up to homotopy (see \cref{th:homotopy-properties}), which points towards the idea that the category $\infAalg$ should possess some kind of \textit{homotopy} symmetric monoidal structure. 
An analogous phenomenon was already observed for the category of homotopy representations of an algebraic group \cite{AriasAbadCrainicDherin11,poliakova2020cellular}.

Our results can be readily applied to different fields. 
The operadic bimodule structure of Point~(2) above was used in the work of the second author, in order to realize $\mathrm{A}_\infty$-algebras and $\mathrm{A}_\infty$-morphisms in Morse theory \cite{mazuir-I,mazuir-II}. 
The algebraic tensor product in Point~(4) has applications in Heegaard Floer homology and could be used to relate the Fukaya categories of products of symplectic manifolds via Lagrangian correspondences, see \cref{ss:diag-symp}.
We also expect future applications of our work to the computation of the homology of fibered spaces, using the construction of the convolution $\Ainf$-algebra associated to an $\Ainf$-coalgebra and an $\Ainf$-algebra in \cref{prop:convolution-ainf}.
This last construction can also be related to the deformation theory of $\infty$-morphisms developed in \cite{RobertNicoudWierstraI,RobertNicoudWierstraII}, see \cref{sec:RNW}.
Moreover, our geometric methods shed a new light on a result of M. Markl and S. Shnider \cite{MarklShnider06}, pointing towards possible links with discrete and continuous Morse theory (\cref{rem:Morse}). 


Finally, the results of this paper can be straightforwardly extended to the "multiploperahedra", a family of polytopes which is to the operahedra of \cite{LA21} what the multiplihedra are to the associahedra. 
They belong at the same time to the families of graph-multiplihedra \cite{DevadossForcey08} and of nestomultiplihedra \cite{AD13}.  
Together with the results of \cite[Section 4]{LA21}, one would obtain a tensor product of $\infty$-morphisms between homotopy operads, defined by explicit formul\ae. 

\subsection*{Layout} 
We introduce the Forcey--Loday and the Ardila-Doker realizations of the multiplihedra in \cref{sec:I}. 
We define a cellular approximation of their diagonal and endow the Forcey--Loday multiplihedra with an operadic bimodule structure over the Loday associahedra in \cref{sec:II}.
We compute explicitly the associated combinatorial formula for the image of our diagonal in \cref{sec:III}. 
We define a tensor product of \Ainf -algebras and of \Ainf -morphisms and study its properties in \cref{sec:IV}.
We finally sketch future applications of our work in \cref{sec:V}.

\subsection*{Conventions} 
We use the conventions and notations of \cite{Ziegler95} for convex polytopes and the ones of \cite{LodayVallette12} for operads. 
The word operad will always mean non-symmetric operad \cite[Section 5.2.8]{LodayVallette12} in this paper. 
We denote by $[n]\coloneqq \{1,\ldots,n\}$ and by $\{ e_i\}_{i \in [n]}$ the standard basis of $\RR^n$.
The abbreviation "dg" will stand for the words "differential graded". 

\subsection*{Acknowledgements} 
We would like to thank Bruno Vallette for numerous discussions and for his careful reading of our paper, as well as Alexandru Oancea and Eric Hoffbeck for their comments on earlier versions.
We are also indebted to Lino Amorim and Robert Lipshitz, for explaining to us their work and for their detailed insights on possible applications of our results in symplectic topology. 
We express our gratitude to Sushmita Venugopalan, for taking the time to discuss potential connections between our work and results on toric varieties, and to Daniel Robert-Nicoud, for discussing his work with us and suggesting new directions of research. 
Finally, we would like to thank the anonymous referee whose detailed comments and suggestions helped improve the exposition.
This work was partially written at Institut Mittag-Leffler in Sweden during the semester \emph{Higher algebraic structures in algebra, topology and geometry}, supported by the Swedish Research Council under grant no. 2016-06596.


\section{Realizations of the multiplihedra} 
\label{sec:I}

Drawing from the work of Forcey in \cite{Forcey08}, we define the weighted Forcey--Loday realizations of the multiplihedra and describe their geometric properties in \cref{prop:PropertiesKLoday}.
We then show how they can be recovered from the Ardila--Doker realizations of the multiplihedra, which are in particular generalized permutahedra.


\subsection{2-colored trees and multiplihedra}

\subsubsection{2-colored trees}

We consider in this section \textit{planar rooted trees}, which we simply abbreviate as \textit{trees}. The term \emph{edge} refers to both internal and external edges. The external edges will sometimes be called leaves. 

\begin{definition}[Cut]
A \emph{cut} of a tree is a subset of edges or vertices which contains precisely one edge or vertex in each path from a leaf to the root.
\end{definition}

A cut divides a tree into an upper part that we color in blue and a lower part that we color in red. 
The edges and vertices of the cut are represented by drawing a black line over them, as pictured in \cref{Fig2:InclusionOrder}. 

\begin{definition}[2-colored tree] \label{def:2coloredtree}
A \emph{2-colored tree} is a tree together with a cut. We call \emph{2-colored atomic tree} a 2-colored binary tree whose cut is made of edges only. 
\end{definition}
We denote by $\CT{n}$ (resp. $\CMT{n}$) the set of 2-colored trees (resp. 2-colored atomic trees) with $n$ leaves, for $n\geq 1$. 

\begin{definition}[Face order]\leavevmode
The \emph{face order} $s\subset t$ on 2-colored trees is defined as follows: a 2-colored tree $s$ is less than a 2-colored tree $t$ if $t$ can be obtained from $s$ by a sequence of contractions of monochrome edges or moves of the cut from a family of edges to an adjacent vertex.

\begin{figure}[h]
\[\vcenter{\hbox{
\begin{tikzpicture}[yscale=0.7,xscale=1]
\draw[very thick, MidnightBlue] (-2.5,3.5)--(-2,2.5);
\draw[very thick, MidnightBlue] (-1.5,3.5)--(-2,2.5);
\draw[very thick, MidnightBlue] (-2,2.5) -- (-1.75,2);
\draw[very thick, MidnightBlue] (-1.25, 2) -- (-1,2.5);
\draw[very thick, MidnightBlue] (-0.5,2.5) -- (0,1.5);
\draw[very thick, MidnightBlue] (0,1.5)--(0,2.5);
\draw[very thick, MidnightBlue] (0,1.5)--(0.5,2.5);
\draw[very thick, MidnightBlue] (1,1)--(1.5,1.5);
\draw[very thick, MidnightBlue] (1.5,1.5)--(1,2.5);
\draw[very thick, MidnightBlue] (1.5,1.5)--(2,2.5);
\draw[very thick, MidnightBlue] (2,2.5)--(1.5,3.5);
\draw[very thick, MidnightBlue] (2,2.5)--(2,3.5);
\draw[very thick, MidnightBlue] (2,2.5)--(2.5,3.5);
\draw[very thick, Red!60] (0,-1)--(0, 1.5); 
\draw[very thick, Red!60] (0,0)--(-1.5,1.5);
\draw[very thick, Red!60] (-1.5,1.5)--(-1.75, 2); 
\draw[very thick, Red!60] (-1.5,1.5)--(-1.25, 2); 
\draw[very thick, Red!60] (0,0)--(1, 1);
\draw (-2,2) to (-1,2); 
\draw (-0.25,1.5)-- (0.25, 1.5) ; 
\draw (0.75,1) to (1.25,1);
\end{tikzpicture}}}
\quad \subset \quad
\vcenter{\hbox{
\begin{tikzpicture}[yscale=0.7,xscale=1]
\draw[very thick, MidnightBlue] (-1.5,1.5)--(-1.5,2.5);
\draw[very thick, MidnightBlue] (-2,2.5) -- (-1.75,2);
\draw[very thick, MidnightBlue] (-1.25, 2) -- (-1,2.5);
\draw[very thick, MidnightBlue] (-0.5,2.5) -- (0,1.5);
\draw[very thick, MidnightBlue] (0,1.5)--(0,2.5);
\draw[very thick, MidnightBlue] (0,1.5)--(0.5,2.5);
\draw[very thick, MidnightBlue] (1,1)--(1.5,1.5);
\draw[very thick, MidnightBlue] (1.5,1.5)--(1,2.5);
\draw[very thick, MidnightBlue] (1.5,1.5)--(1.33,2.5);
\draw[very thick, MidnightBlue] (1.5,1.5)--(1.66,2.5);
\draw[very thick, MidnightBlue] (1.5,1.5)--(2,2.5);
\draw[very thick, MidnightBlue] (-1.5,1.5)--(-1.75, 2); 
\draw[very thick, MidnightBlue] (-1.5,1.5)--(-1.25, 2); 
\draw[very thick, Red!60] (0,-1)--(0, 1.5); 
\draw[very thick, Red!60] (0,0)--(-1.5,1.5);
\draw[very thick, Red!60] (0,0)--(1, 1);
\draw (-1.75,1.5) to (-1.25,1.5); 
\draw (-0.25,1.5) to (0.25,1.5); 
\draw (0.75,1) to (1.25,1);
\end{tikzpicture}}}
\]
\caption{Two 2-colored trees, related by the face order.}
\label{Fig2:InclusionOrder}
\end{figure}
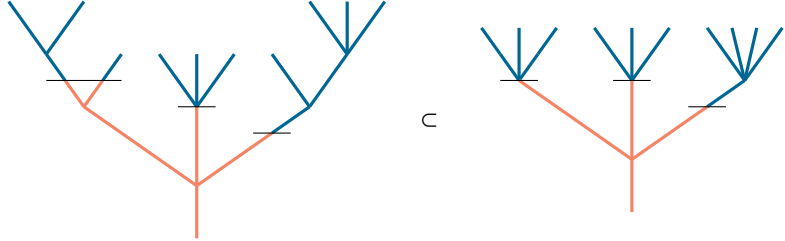

\end{definition}

The set of 2-colored trees endowed with the face order $\subset$ and augmented with a global minimum element $\emptyset_n$ form a poset $(\CT{n}, \subset)$ for which 2-colored atomic trees are atoms.

\begin{definition}[Tamari-type order]\leavevmode
The \emph{Tamari-type order} $s<t$ on  2-colored atomic trees is generated by the following three covering relations: 
\[
{\vcenter{\hbox{
\begin{tikzpicture}[yscale=0.45,xscale=0.45]
\draw[very thick, MidnightBlue] (0,-0.5)--(0,0) -- (-2,2)--(-2,2.5);
\draw[very thick, MidnightBlue] (-1,1)--(0,2)--(0,2.5) ;
\draw[very thick, MidnightBlue] (0,0)--(1,1)--(1,2.5) ;
\draw[very thick, Red!60] (0,-0.5)--(0,-1) ;
\draw (-0.5,-0.5) --(0.5, -0.5);
\draw (-2,2.5) node[above] {$t_1$}; 
\draw (0,2.5) node[above] {$t_2$}; 
\draw (1,2.5) node[above] {$t_3$}; 
\draw (0,-1) node[below] {$t_4$}; 
\end{tikzpicture}}}}
\prec
{\vcenter{\hbox{
\begin{tikzpicture}[yscale=0.45,xscale=0.45]
\draw[very thick, MidnightBlue] (0,-0.5)--(0,0) -- (2,2)--(2,2.5);
\draw[very thick, MidnightBlue] (1,1)--(0,2)--(0,2.5) ;
\draw[very thick, MidnightBlue] (0,0)--(-1,1)--(-1,2.5) ;
\draw[very thick, Red!60] (0,-0.5)--(0,-1) ;
\draw (-0.5,-0.5) --(0.5, -0.5);
\draw (2,2.5) node[above] {$t_3$}; 
\draw (0,2.5) node[above] {$t_2$}; 
\draw (-1,2.5) node[above] {$t_1$}; 
\draw (0,-1) node[below] {$t_4$}; 
\end{tikzpicture}}}}\ , \quad 
{\vcenter{\hbox{
\begin{tikzpicture}[yscale=0.45,xscale=0.45]
\draw[very thick, MidnightBlue] (-2, 2.5)--(-2, 3) ;
\draw[very thick, MidnightBlue] (0, 2.5)--(0, 3) ;
\draw[very thick, MidnightBlue] (1, 2.5)--(1, 3) ;
\draw[very thick, Red!60] (0,-0.5)--(0,0) -- (-2,2)--(-2,2.5);
\draw[very thick, Red!60] (-1,1)--(0,2)--(0,2.5) ;
\draw[very thick, Red!60] (0,0)--(1,1)--(1,2.5) ;
\draw (-2.5,2.5) --(1.5, 2.5);
\draw (-2,3) node[above] {$t_1$}; 
\draw (0,3) node[above] {$t_2$}; 
\draw (1,3) node[above] {$t_3$}; 
\draw (0,-0.5) node[below] {$t_4$}; 
\end{tikzpicture}}}}
\prec
{\vcenter{\hbox{
\begin{tikzpicture}[yscale=0.45,xscale=0.45]
\draw[very thick, Red!60] (0,-0.5)--(0,0) -- (2,2)--(2,2.5);
\draw[very thick, Red!60] (1,1)--(0,2)--(0,2.5) ;
\draw[very thick, Red!60] (0,0)--(-1,1)--(-1,2.5) ;
\draw[very thick, MidnightBlue] (2,2.5)--(2,3) ;
\draw[very thick, MidnightBlue] (0,2.5)--(0,3) ;
\draw[very thick, MidnightBlue] (-1,2.5)--(-1,3) ;
\draw (-1.5,2.5) --(2.5, 2.5);
\draw (2,3) node[above] {$t_3$}; 
\draw (0,3) node[above] {$t_2$}; 
\draw (-1,3) node[above] {$t_1$}; 
\draw (0,-0.5) node[below] {$t_4$}; 
\end{tikzpicture}}}}\ , \quad 
{\vcenter{\hbox{
\begin{tikzpicture}[yscale=0.45,xscale=0.45]
\draw[very thick, MidnightBlue] (0,0.5)--(0,1);
\draw[very thick, MidnightBlue] (0,1)--(-0.5,1.5)--(-0.5,2);
\draw[very thick, MidnightBlue] (0,1)--(0.5,1.5)--(0.5,2);
\draw[very thick, Red!60] (0,0)--(0,0.5);
\draw (-0.5,0.5) --(0.5, 0.5);
\draw (-0.5,2) node[above] {$t_1$}; 
\draw (0.5,2) node[above] {$t_2$}; 
\draw (0,0) node[below] {$t_3$}; 
\end{tikzpicture}}}}
\prec
{\vcenter{\hbox{
\begin{tikzpicture}[yscale=0.45,xscale=0.45]
\draw[very thick, MidnightBlue] (-0.5,2)--(-0.5,2.5);
\draw[very thick, MidnightBlue] (0.5,2)--(0.5,2.5);
\draw[very thick, Red!60] (0,0.5)--(0,1);
\draw[very thick, Red!60] (0,1)--(-0.5,1.5)--(-0.5,2);
\draw[very thick, Red!60] (0,1)--(0.5,1.5)--(0.5,2);
\draw (-1,2) --(1, 2);
\draw (-0.5,2.5) node[above] {$t_1$}; 
\draw (0.5,2.5) node[above] {$t_2$}; 
\draw (0,0.5) node[below] {$t_3$}; 
\end{tikzpicture}}}}
\ ,\]
where each $t_i$, $1\leq i\leq 4$, is a binary tree of the appropriate color. 
\end{definition}

\begin{proposition}
The posets $(\CT{n}, \subset)$ and $(\CMT{n}, <)$ are lattices. 
\end{proposition}

\begin{proof}
The poset of 2-colored trees was proven in \cite{Forcey08} to be isomorphic to the face lattice of a polytope, the multiplihedron; see Point~(3) of \cref{prop:PropertiesKLoday}. 
The Hasse diagram of the poset of 2-colored atomic trees was proven to be isomorphic to the oriented 1-skeleton of the multiplihedron, and also to be the Hasse diagram of a lattice in \cite[Proposition 117]{CP22}.
\end{proof}

\begin{remark}
F. Chapoton and V. Pilaud introduced in \cite{CP22} the shuffle of two generalized permutahedra (see \cref{sec:generalizedpermutahedra} for definition and examples).
The fact that the poset $(\CMT{n}, <)$ is a lattice follows from the fact that the multiplihedron arises as the shuffle of the associahedron and the interval, which both have the lattice property, and that the shuffle operation preserves the lattice property in this case, see \cite[Corollary 95]{CP22}. 
\end{remark}

\subsubsection{Grafting of trees} \label{sss:grafting}

We will denote the operation of grafting a planar tree $v$ at the $i^{\rm th}$-leaf of a 2-colored tree $u$ by $u \circ_i v$. 
We will also denote the grafting of a level of 2-colored trees $v_1, \ldots, v_k$ on the $k$ leaves of a planar tree by $u(v_1, \ldots, v_k)$. 
We denote by $c^{\T}_n$ and by $c^{\B}_n$ the corollae with $n$ leaves fully painted with the upper and the lower color respectively; we denote by $c_n$ the corolla with $n$ leaves with frontier color at the vertex. 
It is straightforward to see that these two grafting operations on corollae generate all the coatoms in the poset of 2-colored trees: we call $(\B)$, for ``bottom'', the first type of 2-colored trees $c_{p+1+r}\circ_{p+1} c^\T_q$, with $p+q+r=n$ and $2\leq q\leq n$, and we call  $(\T)$, for ``top'', the second type of 2-colored trees $c^\B_k(c_1, \ldots, c_k)$, with $i_1+\cdots+i_k=n$, $i_1, \ldots,i_k\geq 1$, and $k\geq 2$.

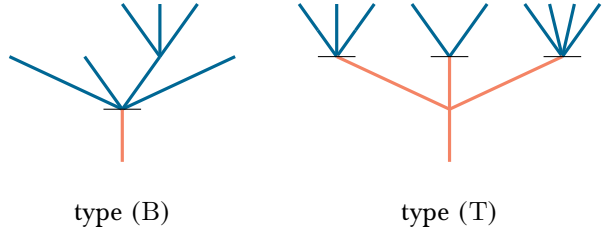
\begin{figure}[h]
\[\vcenter{\hbox{
\begin{tikzpicture}[yscale=0.7,xscale=1]
\draw[very thick, MidnightBlue] (0.5,1)--(0,2);
\draw[very thick, MidnightBlue] (0.5,1)--(0.5,2);
\draw[very thick, MidnightBlue] (0.5,1)--(1,2);
\draw[very thick, MidnightBlue] (0,0)--(0.5, 1); 
\draw[very thick, MidnightBlue] (0,0)--(-0.5, 1); 
\draw[very thick, MidnightBlue] (0,0)--(-1.5,1);
\draw[very thick, MidnightBlue] (0,0)--(1.5, 1);
\draw[very thick, Red!60] (0,-1)--(0, 0); 
\draw (-0.25,0) to (0.25,0); 
\draw (0,-2) node {type $(\B)$};
\end{tikzpicture}}}\qquad \vcenter{\hbox{
\begin{tikzpicture}[yscale=0.7,xscale=1]
\draw[very thick, MidnightBlue] (-1.5,1)--(-1.5,2);
\draw[very thick, MidnightBlue] (-2,2) -- (-1.75,1.5);
\draw[very thick, MidnightBlue] (-1.25, 1.5) -- (-1,2);
\draw[very thick, MidnightBlue] (-0.5,2) -- (0,1);
\draw[very thick, MidnightBlue] (0,1)--(0.5,2);
\draw[very thick, MidnightBlue] (1.5,1)--(1,2);
\draw[very thick, MidnightBlue] (1.5,1)--(1.33,2);
\draw[very thick, MidnightBlue] (1.5,1)--(1.66,2);
\draw[very thick, MidnightBlue] (1.5,1)--(2,2);
\draw[very thick, MidnightBlue] (-1.5,1)--(-1.75, 1.5); 
\draw[very thick, MidnightBlue] (-1.5,1)--(-1.25, 1.5); 
\draw[very thick, Red!60] (0,-1)--(0, 1); 
\draw[very thick, Red!60] (0,0)--(-1.5,1);
\draw[very thick, Red!60] (0,0)--(1.5, 1);
\draw (-1.75,1) to (-1.25,1); 
\draw (-0.25,1) to (0.25,1); 
\draw (1.25,1) to (1.75,1);
\draw (0,-2) node {type $(\T)$};
\end{tikzpicture}}}\]
\caption{Examples of 2-colored trees of type $(\B)$ and $(\T)$ respectively. }
\label{Fig5:FacetsColoredTrees}
\end{figure}


\subsubsection{Multiplihedra} \label{sec:multiplihedra}

\begin{definition}[Multiplihedra]
For any $n\geq 1$, an \emph{$(n-1)$-dimensional multiplihedron} is a polytope of dimension $(n-1)$ whose face lattice is isomorphic to the lattice 
$(\CT{n}, \subset)$
of 2-colored trees with $n$ leaves. 
\end{definition}

\begin{figure}[h]
\[
\begin{tikzpicture}[xscale=0.8,yscale=1]
\draw[fill, opacity=0.12] (-2,2)--(2,2)--(4,0)--(2,-2)--(-2,-2)--(-4,0)--cycle;
\draw (-2,2) node[above left] {$\TreeLa$};
\draw (2,2) node[above right] {$\TreeRa$};
\draw (-2,-2) node[below left] {$\TreeLc$};
\draw (2,-2) node[below right] {$\TreeRc$};
\draw (-4.2,0) node[left] {$\TreeLb$};
\draw (4,0) node[right] {$\TreeRb$};
\draw (-3,1) node[above left] {$\TreeLab$};
\draw (-3,-1) node[below left] {$\TreeLbc$};
\draw (3,1) node[above right] {$\TreeRab$};
\draw (3,-1) node[below right] {$\TreeRbc$};
\draw (0,2.1) node[above] {$\TreeCa$};
\draw (0,-2.1) node[below] {$\TreeCb$};
\draw (0,0) node  {$\TreeCab$};

\draw (-1.99,1.99) node  {$\bullet$};
\draw (1.99,-1.99) node  {$\bullet$};

\draw[thick] (-2,2)--(2,2) node[midway,sloped,allow upside down,scale=0.1]{\thickmidarrow};
\draw[thick] (2,2)--(4,0) node[midway,sloped,allow upside down,scale=0.1]{\thickmidarrow};
\draw[thick] (4,0)--(2,-2) node[midway,sloped,allow upside down,scale=0.1]{\thickmidarrow};

\draw[thick] (-2,2)--(-4,0) node[midway,sloped,allow upside down,scale=0.1]{\thickmidarrow};
\draw[thick] (-4,0)--(-2,-2) node[midway,sloped,allow upside down,scale=0.1]{\thickmidarrow};
\draw[thick] (-2,-2)--(2,-2) node[midway,sloped,allow upside down,scale=0.1]{\thickmidarrow};
\end{tikzpicture}
\]
\caption{A 2-dimensional multiplihedron and the Tamari-type poset $(\CMT{3}, <)$ on its oriented 1-skeleton.}
\label{Fig4:J3}
\end{figure}
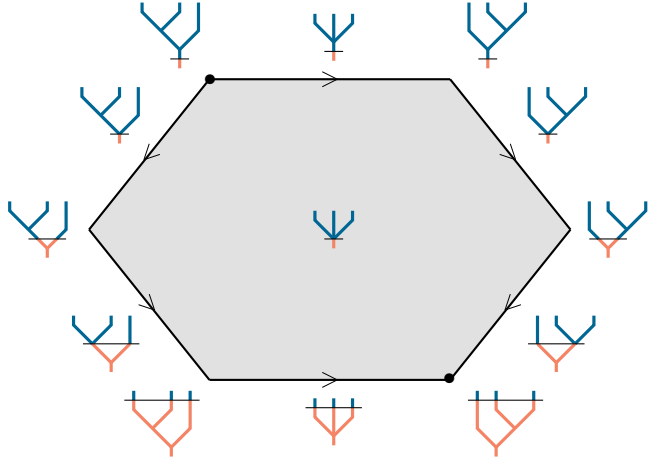

The dimension of a face labeled by a 2-colored tree is given by the sum of the degrees of its vertices defined by 
\[
\left|{\vcenter{\hbox{
\begin{tikzpicture}[scale=0.5]
\draw[very thick, MidnightBlue] (0,-0.5) -- (0,1.5);
\draw[very thick, MidnightBlue] (0,0) -- (-1,1)--(-1,1.5);
\draw[very thick, MidnightBlue] (0,0) -- (1,1)--(1,1.5);
\draw (1,1.5) node[above] {$k$};
\draw (-1,1.5) node[above] {$1$};
\draw (0,1.5) node[above] {$\cdots$};
\end{tikzpicture}}}}\right|=k-2\ , \quad 
\left|{\vcenter{\hbox{
\begin{tikzpicture}[scale=0.5]
\draw[very thick, Red!60] (0,-0.5) -- (0,1.5);
\draw[very thick, Red!60] (0,0) -- (-1,1)--(-1,1.5);
\draw[very thick, Red!60] (0,0) -- (1,1)--(1,1.5);
\draw (1,1.5) node[above] {$k$};
\draw (-1,1.5) node[above] {$1$};
\draw (0,1.5) node[above] {$\cdots$};
\end{tikzpicture}}}}\right|=k-2\ , \quad 
\left|{\vcenter{\hbox{
\begin{tikzpicture}[scale=0.5]
\draw[very thick, MidnightBlue] (0,0) -- (0,1.5);
\draw[very thick, MidnightBlue] (0,0) -- (-1,1)--(-1,1.5);
\draw[very thick, MidnightBlue] (0,0) -- (1,1)--(1,1.5);
\draw[very thick, Red!60] (0,0) -- (0,-0.5);
\draw (-0.5,0)--(0.5,0);
\draw (1,1.5) node[above] {$k$};
\draw (-1,1.5) node[above] {$1$};
\draw (0,1.5) node[above] {$\cdots$};
\end{tikzpicture}}}}\right|=k-1\ .
\]
The codimension of a 2-colored tree is then equal to the number of blue and red vertices. 
In the example of the 2-colored tree depicted on the left of \cref{Fig2:InclusionOrder}, the dimension is equal to 4 and the codimension is equal to 5. 
As proven in \cite[Proposition 117]{CP22}, the oriented $1$-skeleton of a multiplihedron is the Hasse diagram of the Tamari-type poset. 


\subsection{Forcey--Loday realizations of the multiplihedra}
Jean-Louis Loday gave realizations of the associahedra in the form of polytopes with integer coordinates in \cite{Loday04a}. 
Stefan Forcey generalized this construction in \cite{Forcey08} in order to give similar realizations for the multiplihedra. 

\begin{definition}[Weighted 2-colored atomic tree]
A \emph{weighted 2-colored atomic tree} is a pair $(t, \omega)$ made up of a 2-colored atomic tree $t\in \CMT{n}$ with $n$ leaves with a weight $\omega= (\omega_1, \ldots, \omega_n) \in \RR_{>0}^n$. 
We call $\omega$ the \emph{weight} and $n$ the \emph{length} of the weight $\omega$.
\end{definition}

Let $(t, \omega)$ be a weighted 2-colored atomic tree with $n$ leaves. We order its $n-1$ vertices from left to right. At the $i^{\rm th}$ vertex, we consider the sum $\alpha_i$ of the weights of the leaves supported by its left input and 
 the sum $\beta_i$ of the weights of the leaves supported by its right input. 
If the $i^{\rm th}$ vertex is colored by the upper color, we consider the product $\alpha_i\beta_i$ and if the 
$i^{\rm th}$ vertex is colored by the lower color, we consider the product $2\alpha_i\beta_i$.
The associated string produces a point with integer coordinates $M(t, \omega) \in \RR_{>0}^{n-1}$. 
For example, if only the first and last vertices of $t$ are blue, we obtain a point of the form
\[M(t, \omega) = \big(2\alpha_1\beta_1, \alpha_2\beta_2, \ldots, \alpha_{n-2}\beta_{n-2}, 2\alpha_{n-1}\beta_{n-1}\big)\in 
\RR_{>0}^{n-1}\ . \]
\begin{figure}[h!]
\[
\vcenter{\hbox{\begin{tikzpicture}[scale=1.5]
\draw[thick] (1,0)--(2,0);
\draw (1,0) node[above] {$\TreeBa$};
\draw (1.95,0) node[above] {$\TreeBb$};
\draw (1,0) node[below] {$1$};
\draw (2,0) node[below] {$2$};
\end{tikzpicture}}} \qquad \qquad
\vcenter{\hbox{
\begin{tikzpicture}[scale=1.5]
\draw (1,-0.05)--(1,0.05);
\draw (2,-0.05)--(2,0.05);
\draw (3,-0.05)--(3,0.05);
\draw (4,-0.05)--(4,0.05);
\draw (-0.05, 1)--(0.05,1);
\draw (-0.05, 2)--(0.05,2);
\draw (-0.05, 3)--(0.05,3);
\draw (-0.05, 4)--(0.05,4);
\draw[->] (0,0)--(5,0);
\draw[->] (0,0)--(0,5);
\draw (1,0) node[below] {$1$};
\draw (2,0) node[below] {$2$};
\draw (3,0) node[below] {$3$};
\draw (4,0) node[below] {$4$};
\draw (0,1) node[left] {$1$};
\draw (0,2) node[left] {$2$};
\draw (0,3) node[left] {$3$};
\draw (0,4) node[left] {$4$};
\draw[thick] (1,2)--(1,4)--(2,4)--(4,2)--(4,1)--(2,1)--cycle;
\draw (1,2) node[below left] {$\TreeLa$};
\draw (2,1) node[below left] {$\TreeRa$};
\draw (2,4) node[above right] {$\TreeLc$};
\draw (4,2) node[above right] {$\TreeRc$};
\draw (1,4) node[above left] {$\TreeLb$};
\draw (4,1) node[below right] {$\TreeRb$};
\end{tikzpicture}}}
\]
\caption{Examples of points associated to 2-colored atomic trees, with standard weight.}
\end{figure}
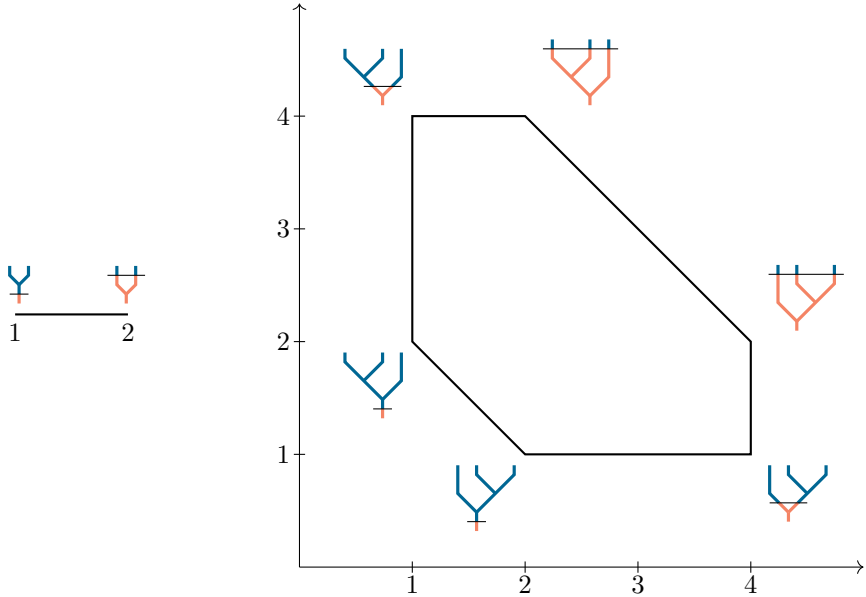

\begin{definition}[Forcey--Loday Realization] \label{def:ForceyLoday}
 The \emph{Forcey--Loday realization of weight $\omega$} of the $(n-1)$-dimensional multiplihedron is the  polytope
\[\J_\omega \coloneqq \conv \big\{M(t, \omega)\mid t\in \CMT{n} \big\}\subset \RR^{n-1}\ .\]
\end{definition}

The Forcey--Loday realization associated to the standard weight $(1, \ldots, 1)$ will simply be denoted by $\J_n$.
By convention, we define the polytope $\J_\omega$ with weight $\omega=(\omega_1)$ of length $1$  to be made up of one point labeled by the 2-colored  tree $\iBT\coloneqq \TreeIab$\ .
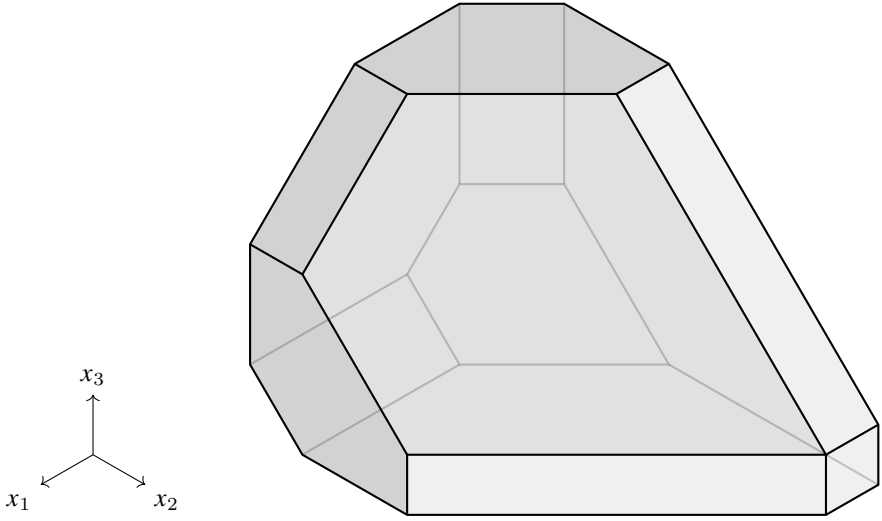
\begin{figure}[h]
\[
\begin{tikzpicture}[scale=0.8, J4]
\draw[->] (4,-4,-3)--(5,-4,-3) node[below left] {$x_1$};
\draw[->] (4,-4,-3)--(4,-3, -3) node[below right] {$x_2$};
\draw[->] (4,-4,-3)--(4,-4,-2) node[above] {$x_3$};

\draw[thick, opacity=0.2] (1,4,1)--(1,2,3)--(2,1,3)--(3,1,2)--(3,2,1)--cycle;
\draw[thick] (2,8,2)--(2,4,6)--(4,2,6)--(6,2,4)--(6,4,2)--cycle;
\draw[thick] (4,1,6)--(6,1,4)--(6,1,2)--(6,2,1)--(6,4,1)--(2,8,1)--(1,8,1)--(1,8,2)--(1,4,6)--(1,2,6)--(2,1,6)--cycle;
\draw[thick] (4,1,6)--(4,2,6);
\draw[thick] (6,1,4)--(6,2,4);
\draw[thick, opacity=0.2] (6,1,2)--(3,1,2);
\draw[thick, opacity=0.2] (6,2,1)--(3,2,1);
\draw[thick] (6,4,1)--(6,4,2);
\draw[thick] (2,8,1)--(2,8,2)--(1,8,2);
\draw[thick, opacity=0.2] (1,8,1)--(1,4,1);
\draw[thick] (1,4,6)--(2,4,6);
\draw[thick, opacity=0.2] (1,2,6)--(1,2,3);
\draw[thick, opacity=0.2] (2,1,6)--(2,1,3);

\draw[fill, opacity=0.12] (2,8,2)--(2,4,6)--(4,2,6)--(6,2,4)--(6,4,2)--cycle;
\draw[fill, opacity=0.18] (6,2,4)--(6,1,4)--(6,1,2)--(6,2,1)--(6,4,1)--(6,4,2)--cycle;
\draw[fill, opacity=0.18] (6,2,4)--(6,1,4)--(4,1,6)--(4,2,6)--cycle;
\draw[fill, opacity=0.18] (4,1,6)--(4,2,6)--(2,4,6)--(1,4,6)--(1,2,6)--(2,1,6)--cycle;
\draw[fill, opacity=0.06] (2,8,2)--(6,4,2)--(6,4,1)--(2,8,1)--cycle;
\draw[fill, opacity=0.06] (2,8,1)--(2,8,2)--(1,8,2)--(1,8,1)--cycle;
\draw[fill, opacity=0.06] (2,8,2)--(1,8,2)--(1,4,6)--(2,4,6)--cycle;
\end{tikzpicture}
\]
\caption{The Forcey--Loday realization of the multiplihedron $\J_4$ .}
\end{figure}

\begin{proposition}\label{prop:PropertiesKLoday}
The Forcey--Loday realization $\J_\omega$  satisfies the following properties. 
\begin{enumerate}[leftmargin=*]
\item Let $t\in \CMT{n}$ be a 2-colored atomic tree. 

\noindent For $p+q+r=n$, with $2\leq q\leq n$, the point $M(t, \omega)$ is contained in the half-space defined by the inequality
\begin{equation}\label{Eq:B}\tag{$\B$}
x_{p+1}+\cdots+x_{p+q-1}\geq \sum_{p+1\leq a<b\leq p+q} \omega_a \omega_b\ , 
\end{equation}
with equality if and only if the 2-colored atomic tree $t$ can be decomposed as $t=u\circ_{p+1} v$, where $u\in\CMT{p+1+r}$ and $v\in \Tam{q}$. 

\noindent For $i_1+\cdots+i_k=n$, with $i_1, \ldots,i_k\geq 1$ and $k\geq 2$, the point $M(t, \omega)$ is contained in the half-space defined by the inequality
\begin{equation}\label{Eq:T}\tag{$\T$}
x_{i_1}+x_{i_1+i_2}+\cdots+x_{i_1+\cdots+i_{k-1}}\leq 
2\sum_{1\leq j<l\leq k} \omega_{I_j} \omega_{I_l}\ , 
\end{equation}
where $I_j=[i_1+\cdots +i_{j-1}+1, \ldots, i_1+\cdots +i_j]$ and $\omega_{I_j}\coloneqq\sum_{a\in I_j} \omega_a$, with equality if and only if the 2-colored atomic tree $t$ can be decomposed as $t=u(v_1, \ldots, v_k)$, where $u\in\Tam{k}$ and $v_j\in \CMT{i_j}$, for $1\leq j\leq k$.

\item The polytope $\J_\omega$ is the intersection of the half-spaces defined in  \emph{(1)}. 

\item The face lattice $(\La(\J_\omega), \subset)$ is isomorphic to the lattice $(\CT{n}, \subset)$ of 2-colored trees with $n$ leaves.

\item Any face of a Forcey--Loday realization of a multiplihedron is isomorphic to a product of a Loday realization of an associahedron with possibly many Forcey--Loday realizations of multiplihedra, via a permutation of coordinates. 
\end{enumerate}
\end{proposition}

\begin{proof}

Points~(1)--(3) were proved in \cite{Forcey08}.  
We prove Point~(4) by induction on $n$. 
It clearly holds true for $n=1$. Let us suppose that it holds true up to $n-1$ and let us prove it for the polytopes $\J_\omega$, for any weight $\omega$ of length $n$.
We examine first facets. 
In the case of a facet of type $(\B)$ associated to $p+q+r=n$ with $2\leq q \leq n-1$, we consider the following two weights 
\[
\overline{\omega}\coloneqq (\omega_1, \ldots, \omega_{p}, \omega_{p+1}+\cdots+\omega_{p+q}, \omega_{p+q+1}, \ldots,  \omega_{n})
\quad \text{and} \quad 
\widetilde{\omega}\coloneqq (\omega_{p+1}, \ldots, \omega_{p+q})
\]
and the isomorphism 
\begin{align*}
\RR^{p+r} \times \RR^{q-1} &\xrightarrow{\Theta_{p,q,r}} \RR^{n-1} \\
(x_1, \ldots, x_{p+r}) \times (y_1, \ldots, y_{q-1}) &\mapsto
(x_1, \ldots, x_{p} , y_1, \ldots, y_{q-1}, x_{p+1}, \ldots, x_{p+r}) \ .
\end{align*}
The image of the vertices of $\J_{\overline{\omega}}\times \K_{\widetilde{\omega}}$ are sent to the vertices of the facet of $\J_\omega$
labelled by the 2-colored tree $c_{p+1+r}\circ_{p+1} c^\T_q$. 
In other words, the permutation of coordinates $\Theta$ sends bijectively $\J_{\overline{\omega}}\times \K_{\widetilde{\omega}}$ to $\J_\omega$. 
Similarly, in the case of a facet of type $(\T)$ associated to $i_1+\cdots+i_k=n$ with 
$i_1, \ldots,i_k\geq 1$ and $k\geq 2$, 
 we consider the following weights 
\[
\overline{\omega}\coloneqq \big(\sqrt{2}\omega_{I_1}, \ldots, \sqrt{2}\omega_{I_k}\big)
\quad \text{and} \quad 
\widetilde{\omega}_j\coloneqq (\omega_{i_1+\cdots+i_{j-1}+1}, \ldots, \omega_{i_1+\cdots+i_{j-1}+i_j}), \ \text{for}\ 1\leq j\leq k, 
\]
and the isomorphism 
\begin{align*}
\begin{array}{rccc}
\Theta^{i_1, \ldots, i_k}\  : &  \RR^{k-1}\times \RR^{i_1-1}\times \cdots \times \RR^{i_k-1} &\xrightarrow{\cong} &\RR^{n-1}
\end{array}
\end{align*}
which sends 
\[(x_1, \ldots, x_{k-1})\times (y_1^1, \ldots, y^1_{i_1-1})\times \cdots 
\times (y_1^k, \ldots, y^k_{i_k-1})\]
to 
\[(
y^1_1,\ldots, y^1_{i_1-1}, x_1, y^2_1, \ldots, y^2_{i_2-1}, x_2, y^3_1, \ldots, x_{k-1}, y^k_1, \ldots, y^k_{i_k-1}
)\ .\]
The image of the vertices of 
$\K_{\overline{\omega}}\times \J_{\widetilde{\omega}_1}\times \cdots \times \J_{\widetilde{\omega}_k}$ are sent to the vertices of the facet of $\J_\omega$
labelled by the 2-colored tree $c^\B_k(c_1, \ldots, c_k)$. In other words, the permutation of coordinates $\Theta$ sends bijectively $\K_{\overline{\omega}}\times \J_{\widetilde{\omega}_1}\times \cdots \times \J_{\widetilde{\omega}_k}$ to $\J_\omega$. 

We can finally conclude the proof with these decompositions of facets of $\J_\omega$, the induction hypothesis, and Point~(5) of \cite[Proposition~1]{MTTV19}.
\end{proof}


\subsection{Ardila-Doker realizations of the multiplihedra} 
\label{sec:generalizedpermutahedra}

\begin{definition}[Permutahedron] The \emph{$(n-1)$-dimensional permutahedron} is the polytope in $\RR^n$ equivalently defined as:
\begin{itemize}[leftmargin=*]
  \item the convex hull of the points $\displaystyle \sum_{i=1}^{n}i e_{\sigma(i)}$ for all permutations $\sigma \in \mathbb{S}_n$, or
  \item the intersection of the hyperplane $\displaystyle  \left\{x \in \RR^n \ \bigg| \ \sum_{i=1}^{n} x_i = \binom{n+1}{2}\right\}$ with the affine half-spaces \\ $\displaystyle \left\{x \in \RR^n \ \bigg| \ \sum_{i \in I} x_i \geq \binom{|I|+1}{2}\right\}$ for all $\emptyset\neq I \subseteq [n]$.
\end{itemize}
\end{definition}

For a face $F$ of a polytope $P\subset\RR^n$, the \emph{normal cone} of $F$ is the cone 
\[\mathcal{N}_P(F)\coloneqq \left\{ c \in (\RR^n)^{*} \ \bigg | \ F \subseteq \{ x \in P \ | \ c x =\max_{y \in P} c y \}\right\}  \ . \]  
The codimension of $\mathcal{N}_P(F)$ is equal to the dimension of $F$. 
The \emph{normal fan} of $P$ is the collection of the normal cones $\mathcal{N}_P \coloneqq \{\mathcal{N}_P(F) \ | \ F \in \mathcal{L}(P)\setminus\emptyset \}$. 
We refer to \cite[Chapter 7]{Ziegler95} for more details. 

\begin{definition}[Generalized permutahedron]
A \emph{generalized permutahedron} is a polytope equivalently defined as:
\begin{itemize}[leftmargin=*]
\item a polytope whose normal fan coarsens the one of the permutahedron, or 
\item the convex set \[ \left\{ x \in \RR^n \ : \ \sum_{i=1}^{n}x_i = z_{[n]} \ , \sum_{i \in I} x_i \geq z_I \text{ for all } I \subseteq [n] \right\} \ , \]
where $\{ z_I \}_{I \subseteq [n]}$ are real numbers which satisfy the inequalities $z_I+z_J \leq z_{I\cup J} + z_{I \cap J}$ for all $I,J \subseteq [n]$, and where $z_\emptyset =0$.
\end{itemize}
\end{definition}

Generalized permutahedra were introduced by A. Postnikov in \cite{Postnikov09}.
Loday realizations of the associahedra are all generalized permutahedra (see \cite[Corollary 2.16 (2)]{LA21}, specialized to nested linear trees), while Forcey--Loday realizations of the multiplihedra are not. 
However, F. Ardila and J. Doker introduced in \cite{AD13} realizations of the multiplihedra that are generalized permutahedra. 
They are obtained from the Loday realizations of the associahedra via the operation of \emph{$q$-lifting}. 
We will consider the special case $q=1/2$ of their construction.

\begin{definition}[Lifting of a generalized permutahedron {\cite[Definition 2.3]{AD13}}]
For a generalized permutahedron $P\subset \RR^n$, its \emph{$\tfrac{1}{2}$-lifting} $P \left(\tfrac{1}{2}\right) \subset \RR^{n+1}$ is defined by 
\[P \left(\tfrac{1}{2}\right) \coloneqq \left\{ x \in \RR^{n+1} \ : \ 
\sum_{i=1}^{n+1} x_i = z_{[n]} \ , 
\sum_{i \in I} x_i \geq \tfrac{1}{2}z_I \ ,
\sum_{i \in I \cup \{n+1\}} x_i \geq z_I 
\text{ for all } I \subseteq [n] \right\} \ . \]
\end{definition}

\begin{proposition}[{\cite[Proposition 2.4]{AD13}}] 
The $\tfrac{1}{2}$-lifting $P \left(\tfrac{1}{2}\right)$ of a generalized permutahedron is again a generalized permutahedron. 
\end{proposition}

\begin{proposition} 
The generalized permutahedron $\K_\omega\left(\tfrac{1}{2}\right)$ given by the $\tfrac{1}{2}$-lifting of the Loday realization of weight $\omega$ of the associahedron is a realization of the multiplihedron. 
\end{proposition}
\begin{proof} 
This is a particular case of \cite[Corollary 4.10]{AD13}.
\end{proof}

We call the lifting of the Loday associahedron $\K_\omega\left(\tfrac{1}{2}\right)$ the\emph{ Ardila--Doker realization} of the multiplihedron. It is related to the Forcey--Loday realization via the projection $\pi: \RR^{n+1} \to \RR^n$ which forgets the last coordinate. 
 
\begin{proposition} 
\label{prop:lifting} 
The Forcey--Loday realization of the multiplihedron is the image under the projection $\pi$ of the $\tfrac{1}{2}$-lifting of the Loday realization of the associahedron, scaled by $2$. 
That is, we have  \[ \J_\omega = \pi \left(2 \K_\omega\left(\tfrac{1}{2}\right)\right) \ . \]
\end{proposition}

\begin{proof} 
  This follows from the vertex description of $\tfrac{1}{2}$-lifting given in \cite[Definition 3.5.3]{Doker11}, together with the description of the projection from the permutahedron to the multiplihedron given in the proof of \cite[Theorem 3.3.6]{Doker11}. 
  The coordinates of a vertex in $2 \K_\omega$ are of the form $(2\alpha_1\beta_1, \ldots, 2\alpha_n\beta_n)$. 
  A coordinate $2\alpha_i\beta_i$ is then multiplied by $1/2$ in the lifting if and only if its associated vertex in the 2-colored atomic tree is of the upper color. 
  We thus recover the description of \cref{def:ForceyLoday}.
\end{proof}

In summary, we have the following diagram:
\medskip
\begin{equation*}
\begin{matrix}
  $ \small  \text{Loday}$ & & $ \small \text{Ardila--Doker}$ &  & $ \small \text{Forcey--Loday}$ \\
  $ \small  \text{associahedron}$ & & $ \small \text{multiplihedron}$ &  & $ \small \text{multiplihedron}$ \\
  & &  &  & \\
  \K_\omega & \hookrightarrow & \K_\omega \left(\tfrac{1}{2}\right) & \overset{\pi ( 2 \cdot ) }{\twoheadrightarrow} & \J_\omega \\
   & &  &  & \\
  \RR^n & \hookrightarrow & \RR^{n+1} & \twoheadrightarrow & \RR^n \\
  & &  &  & \\
  $ \small \text{Gen. permutahedron}$ & & $ \small \text{Gen. permutahedron}$ &  & $ \small \textit{Not}\text{ a gen. permutahedron}$
\end{matrix}
\end{equation*}


\section{Diagonal of the multiplihedra}
\label{sec:II}

In this section, we define a cellular approximation of the diagonal of the Forcey--Loday realizations of the multiplihedra, and we endow them with an operadic bimodule structure over the Loday realizations of the associahedra in the category $\PolySub$. 
We use the methods of \cite{MTTV19} and the general theory developed in \cite{LA21}.
Our construction of the cellular approximation relies crucially on the fact that the Forcey--Loday multiplihedra, are obtained from the Ardila--Doker multiplihedra by projection (\cref{prop:lifting}).

 
\subsection{The monoidal category $\PolySub$}

We recall the definition of the symmetric monoidal category $(\PolySub, \times)$ from \cite[Section~2.1]{MTTV19}.
\begin{description}
\item[{\sc Objects}] An object of $\PolySub$ is a $d$-dimensional  polytope $P$ in the $n$-dimensional Euclidian space $\RR^n$, for any $0\leq d\leq n$.
\item[{\sc Morphisms}] A morphism in $\PolySub$ is a continuous map  $f: P\to Q$ which sends  $P$ homeomorphically to the underlying set $|\mathcal{D}|$ of a polytopal subcomplex $\mathcal{D}\subset~\La(Q)$ of $Q$ 
such that $f^{-1}(\mathcal D)$ defines a polytopal subdivision of $P$.
\end{description}

We will use the notion of \textit{operad}, \textit{operadic bimodule} and \textit{Hadamard product} of operads and operadic bimodules in the rest of this paper. For the sake of concision, we refer respectively to \cite[Section 1.1.1]{mazuir-I}, \cite[Section 1.1.3]{mazuir-I} and \cite[Section 5.1.12]{LodayVallette12} for a complete definition of these notions. An operad will in particular be a non-symmetric operad in the language of \cite[Section 5.2.8]{LodayVallette12}. The fact that the category $\PolySub$ is monoidal will moreover allow us to define operads and operadic bimodules in polytopes. 

\subsection{Positively oriented polytopes and diagonal maps}

For a polytope $P$, we will denote by $\rho_z P \coloneqq 2z-P$ its reflection with respect to a point $z \in P$. 

\begin{definition}
A \emph{positively oriented polytope} $(P, \vec v)$ is a polytope $P \subset \RR^n$ together with a vector $\vec v\in \RR^n$ which is not perpendicular to any edge of $P\cap \rho_z P$, for any $z \in P$.
\end{definition}

For any point $z$ in a positively oriented polytope, the intersection $P \cap \rho_z P$ admits a unique vertex $\bm_{\vec v}(P \cap \rho_z P)$ which minimizes the euclidean scalar product with $\vec v$, and a unique vertex $\tp_{\vec v}(P \cap \rho_z P)$ which maximizes it. These two distinguished vertices in each intersection define a diagonal map
\begin{align*}
\begin{array}{rlcl}
\triangle_{(P,\vec v)}\  : & P &\to  &P \times P\\
&z & \mapsto& 
\bigl(\bm_{\vec v}(P\cap \rho_zP),\,  \tp_{\vec v}(P\cap \rho_z P)\bigr) \ .
\end{array}
\end{align*}
Such a map is a morphism in $\PolySub$, coincides with the usual thin diagonal $x\mapsto (x, x)$ on vertices, and is fiber-homotopic to it, see \cite[Proposition~5]{MTTV19} and \cite[Proposition 1.1]{LA21}.
Its cellular image, which consists of pairs of faces $(F,G)$ of $P$, admits a combinatorial description in terms of the fundamental hyperplane arrangement of $P$, as we will now recall.

\begin{definition}[Fundamental hyperplane arrangement]
  \label{def:fundamentalhyperplane} 
  An \emph{edge hyperplane} of $P$ is an hyperplane in $\RR^n$ which is orthogonal to the direction of an edge of $P\cap\rho_z P$ for some $z \in P$.
  The \emph{fundamental hyperplane arrangement} $\mathcal{H}_P$ of $P$ is the collection of all edge hyperplanes of $P$. 
\end{definition}

Recall that a face $F$ of a polytope $P \subset \RR^n$ is equal to the intersection of a family of facets $\{F_i\}$. 
If we choose an outward pointing normal vector $\vec F_i$ for each facet $F_i$ (see \cite[Definition 1.25]{LA21}), a basis $\{b_k\}$ of the orthogonal complement of the affine hull of $P$ in $\RR^n$, and use the canonical identification of $\RR^n$ with its dual $(\RR^n)^{*}$, we have that the normal cone of $F$ is given by $\mathcal{N}_P(F)=\cone(\{\vec F_i\} \cup \{b_k,-b_k\})$. 

\begin{proposition}[{\cite[Theorem 1.26]{LA21}}]
  \label{thm:universalformula} 
  Let $(P,\vec v)$ be a positively oriented polytope in $\RR^n$. For each $H\in\mathcal{H}_P$, we choose a normal vector $\vec d_H$ such that $\langle \vec d_H, \vec v \rangle >0$. We have 
\begin{eqnarray*}
  (F,G) \in \Ima \triangle_{(P,\vec v)} 
  &\iff&  \forall H \in \mathcal{H}_P , \ \exists i , \ \langle \vec F_i, \vec d_H \rangle < 0  \text{ or } \exists j , \ \langle \vec G_j, \vec d_H \rangle > 0 \ . 
\end{eqnarray*} 
\end{proposition}

We finally recall general facts from \cite[Section 1.6]{LA21}. 

\begin{definition}[Coarsening projection] 
  \label{def:coarseningprojection} 
  Let $P$ and $Q$ be two polytopes in $\RR^n$ such that the normal fan of $P$ refines the normal fan of $Q$. 
  The \emph{coarsening projection} from $P$ to $Q$ is the map $\theta : \mathcal{L}(P)\to\mathcal{L}(Q)$ which sends a face $F$ of $P$ to the face $\theta(F)$ of $Q$ whose normal cone $\mathcal{N}_Q(\theta(F))$ is the minimal cone with respect to inclusion which contains $\mathcal{N}_P(F)$.
\end{definition}

\begin{proposition} 
\label{prop:refinementofnormalfans}
Let $P$ and $Q$ be two polytopes such that the normal fan of $P$ refines the one of $Q$. 
If $P$ is positively oriented by $\vec v$, then so is $Q$. 
Moreover, the coarsening projection from $P$ to $Q$ commutes with the diagonal maps $\triangle_{(P,\vec v)}$ and $\triangle_{(Q,\vec v)}$, and we have 
\begin{eqnarray*}
  (F,G) \in \Ima \triangle_{(Q,\vec v)} 
  &\iff& \forall H \in \mathcal{H}_P , \ \exists i , \ \langle \vec F_i, \vec d_H \rangle < 0  \text{ or } \exists j , \ \langle \vec G_j, \vec d_H \rangle > 0 \ .
\end{eqnarray*} 
\end{proposition}

We will apply \cref{prop:refinementofnormalfans} to $P$ the permutahedron and $Q$ the Ardila--Doker multiplihedron, in order to define a diagonal map on the Forcey--Loday multiplihedron and to compute an explicit formula for its cellular image in \cref{thm:formuladiagonal}.


\subsection{Good orientation vectors and generalized permutahedra}

The projection $\pi : \RR^{n+1} \to \RR^n$ forgetting the last coordinate defines an affine isomorphism between any hyperplane $H$ of equation $\sum_{i=1}^{n+1} x_i = c \in \RR$, and $\RR^n$. 
The inverse map $(\pi_{| H})^{-1}$ is given by the assignment \[ (x_1, \ldots, x_n) \mapsto \left(x_1, \ldots, x_n, c- \sum_{i=1}^{n}x_i\right) \ . \]
If a polytope $P$ is contained in the hyperplane $H$, then the polytope $\pi(P)$ is affinely isomorphic to $P$, and the projection $\pi$ defines a bijection between the faces of $P$ and the faces of $\pi(P)$. Moreover, for every face $F$ of $P$, we have $\dim F = \dim \pi(F)$.

However, the projection $\pi$ does not preserve orthogonality in general, so if $P$ is positively oriented by $\vec v$, the projection $\pi(P)$ might not be positively oriented by $\pi(\vec v)$.  
We restrict our attention to a certain class of orientation vectors for which this property holds, in the case where $P$ is a generalized permutahedron.

\begin{definition} 
\label{def:goodvector}
A \emph{good orientation vector} is a vector $\vec v=(v_1, \ldots, v_{n+1})\in \RR^{n+1}$ satisfying \[v_{i}\geq2v_{i+1}\ , \ \text{for any}\  1\leq i\leq n\ , \quad \text{and}\quad  v_{n+1}>0 \ . \]
\end{definition}

Observe that the family of good orientation vectors is stable under the projection forgetting the last coordinate: if $\vec v$ is a good orientation vector, then so is $\pi(\vec v)$.
Being a good orientation vector is a more restrictive condition than being a principal orientation vector in the sense of \cite[Definition 3.15]{LA21}. Thus, a good orientation vector orients positively any generalized permutahedron. 

\begin{proposition} 
\label{prop:goodprojection}
Let $P \subset \RR^{n+1}$ be a generalized permutahedron, and let $\vec v \in \RR^{n+1}$ be a good orientation vector. 
Then, the polytope $\pi(P)$ is positively oriented by $\pi(\vec v)$. 
Moreover, the projection $\pi$ commutes with the diagonal maps of $P$ and $\pi(P)$, that is $\triangle_{(\pi(P),\pi(\vec v))}=(\pi \times \pi)\triangle_{(P,\vec v)}$.
\end{proposition}

\begin{proof} 
Since $P$ is a generalized permutahedron, the direction of the edges of the intersection $P\cap\rho_z P$, for any $z \in P$, are vectors with coordinates equal to $0,1$ or $-1$, and the same number of $1$ and $-1$ (combine Proposition 1.30 and Proposition 3.4 of \cite{LA21}). 
The direction $\vec d$ of such an edge satisfies $\langle \vec d, \vec v \rangle \neq 0$, since the first non-zero coordinate of $\vec d$ will contribute a greater amount than the sum of the remaining coordinates in the scalar product.  
For the same reason, we have $\langle \pi(\vec d), \pi(\vec v) \rangle \neq 0$. As $\pi(P\cap\rho_z P)=\pi(P)\cap\rho_{\pi(z)}\pi(P)$, we have in particular that the image of the edges of $P\cap\rho_z P$ under $\pi$ are the edges of $\pi(P)\cap\rho_{\pi(z)}\pi(P)$ and thus that $\pi(P)$ is positively oriented by $\pi(\vec v)$. 
For the last part of the statement, observe that $\pi$ preserves the orientation of the edges: if we have $\langle \vec d, \vec v \rangle >0$, then we have $\langle \pi(\vec d), \pi(\vec v) \rangle > 0$. 
Hence, the image of the vertex $\tp_{\vec v}(P\cap\rho_z P)$, which maximizes $\langle - ,\vec v \rangle$ over $P\cap\rho_z P$, under $\pi$ is equal to the vertex $\tp_{\pi(\vec v)}(\pi(P)\cap\rho_{\pi(z)} \pi(P))$ which maximizes $\langle - ,\pi(\vec v) \rangle$ over $\pi(P)\cap\rho_{\pi(z)} \pi(P)$. The argument for the minimum $\bm(P\cap\rho_z P)$ is the same.
\end{proof}

\begin{proposition}
Let $P\subset\RR^{n+1}$ be a generalized permutahedron. 
Any two good orientation vectors $\vec v, \vec w$ define the same diagonal maps on $P$ and $\pi(P)$, that is, we have $\triangle_{(P,\vec v)}=\triangle_{(P,\vec w)}$ and $\triangle_{(\pi(P),\pi(\vec v))}=\triangle_{(\pi(P),\pi(\vec w))}$.
\end{proposition}
\begin{proof}
Good orientation vectors are principal orientation vectors \cite[Definition 3.15]{LA21}. Since all principal orientation vectors live in the same chamber of the fundamental hyperplane arrangement of the permutahedron, they all define the same diagonal on the permutahedron \cite[Proposition 1.23]{LA21}, and thus the same diagonal on any generalized permutahedron (\cref{prop:refinementofnormalfans}). So, we have $\triangle_{(P,\vec v)}=\triangle_{(P,\vec w)}$. Finally, using \cref{prop:goodprojection}, we have $\triangle_{(\pi(P),\pi(\vec v))}=(\pi \times \pi)\triangle_{(P,\vec v)}=(\pi \times \pi)\triangle_{(P,\vec w)}=\triangle_{(\pi(P),\pi(\vec w))}$. 
\end{proof}

\subsection{Diagonal of the Forcey--Loday multiplihedra}
\label{sec:diagonal}

\begin{definition}
A \emph{well-oriented realization of the multiplihedron} is a positively oriented polytope which realizes the multiplihedron and such that the orientation vector induces the Tamari-type order on the set of vertices. 
\end{definition}

\begin{proposition}
\label{prop:OrientationVector}
Any good orientation vector induces a well-oriented realization $\left( \J_\omega, \vec v \right)$ of the Forcey--Loday multiplihedron, for any weight $\omega$. 
\end{proposition}

\begin{proof}
Using \cref{def:ForceyLoday}, we can compute that any edge of the realization of the multiplihedron $\J_\omega$ is directed, according to the Tamari type order, by either $ e_i$ or $ e_i- e_j$, for $i<j$.
Since $\vec v$ has strictly decreasing coordinates, the scalar product is in each case positive. 
It remains to show that $P\cap\rho_z P$ is oriented by $\vec v$, for any $z \in P$. 
This follows directly from \cref{prop:goodprojection}, and the fact that $\J_\omega$ arises as the projection under $\pi$ of a generalized permutahedron as shown in  \cref{prop:lifting}.
\end{proof}

Any good orientation vector therefore defines a diagonal map $\triangle_\omega : \J_\omega\to \J_\omega \times \J_\omega$, for any weight $\omega$.
These diagonal maps are all equivalent up to isomorphim in the category $\PolySub$. 

\begin{proposition}
\label{prop:transitionmap}
For any pair of weights $\omega$ and $\theta$ of length $n$, there exists a unique isomorphism 
$\tr=\tr_\omega^\theta : \J_\omega \to \J_\theta$ in the category $\PolySub$,  
which preserves homeomorphically the faces of the same type and which commutes with the respective diagonals.
\end{proposition}

\begin{proof}
The arguments of \cite[Sections~3.1-3.2]{MTTV19} hold in the present case using \cref{prop:PropertiesKLoday}. 
We note that the crucial condition above is that the map $\tr$ commutes with the respective diagonals: this makes the map $\tr$ unique and highly non-trivial to construct, see the proof of \cite[Proposition 7]{MTTV19}.
\end{proof}

\begin{definition} \label{def:diagonal-multipl-forcey-loday}
We define $\triangle_n : \J_n \to \J_n\times \J_n$ to be the diagonal induced by any good orientation vector for the Forcey--Loday realization of standard weight $\omega=(1, \ldots, 1)$.
\end{definition}


\subsection{Operadic bimodule structure on the Forcey--Loday multiplihedra}

We will use the transition maps $\tr$ of \cref{prop:transitionmap} above to endow the family of standard weight Forcey--Loday multiplihedra with an operadic bimodule structure over the standard weight Loday associahedra. 
The uniqueness property of the map $\tr$ will be used in a crucial way.

\begin{definition}[Action-composition maps] \label{def:action-composition}
For any $p+q+r=n\geq 1$, $k\geq 2$ and any $i_1,\ldots,i_k \geq 1$, we define the \emph{action-composition maps} by 
\[
\vcenter{\hbox{
\begin{tikzcd}[column sep = 16pt]
\circ_{p+1}\ : \ \J_{p+1+r}\times \K_q
\arrow[rr,  "\tr\times \id"]
& & 
\J_{(1,\ldots,q,\ldots,1)}\times \K_q 
\arrow[rr,hookrightarrow, "\Theta_{p,q,r}"]
&  &
\J_{n}\ \ \text{and}
\end{tikzcd}
}}
\]
\[
\vcenter{\hbox{
\begin{tikzcd}[column sep = 16pt]
\gamma_{i_1,\ldots,i_k}\ : \ \K_{k}\times \J_{i_1} \times \cdots \times \J_{i_k}
\arrow[rr,  "\tr\times \id"]
& &
\K_{(i_1,\ldots,i_k)} \times \J_{i_1} \times \cdots \times \J_{i_k} 
\arrow[rr,hookrightarrow, "\Theta^{i_1, \ldots , i_k}"]
& &
\J_{i_1+\cdots + i_k}\ , 
\end{tikzcd}
}}
\]
where the last inclusions are given by the block permutations of the coordinates introduced in the proof of \cref{prop:PropertiesKLoday}. 
\end{definition}

Recall from \cite[Theorem 1]{MTTV19} that the diagonal maps $\triangle_n : \K_n \to \K_n \times \K_n$ define a morphism of operads, where the operad $\{ \K_n \times \K_n \}$ is to be understood as the Hadamard product $\{ \K_n \} \times \{ \K_n \}$.
The next proposition shows that the diagonal maps $\triangle_n : \K_n \to \K_n \times \K_n$ and $\triangle_n : \J_n \to \J_n \times \J_n$ are compatible with the action-composition maps introduced in \cref{def:action-composition}.

\begin{proposition} 
\label{prop:thetacommutes}
The diagonal maps $\triangle_n$ commute with the maps $\Theta$.  
\end{proposition}

\begin{proof}
First observe that a good orientation vector has decreasing coordinates, thereby induces the diagonal maps $\triangle_n : \K_n \to \K_n \times \K_n$ and the operad structure on $\{\K_n\}$ defined in \cite{MTTV19}. 
Following \cite[Proposition 4.14]{LA21}, to prove the claim it suffices to show that the preimage under $\Theta^{-1}$ of a good orientation vector is still a good orientation vector for each associahedron and multiplihedron. 
This is easily seen to be the case from the definition of $\Theta$, in the proof of \cref{prop:PropertiesKLoday}. 
\end{proof}

\begin{samepage}
\begin{theorem}\label{thm:MainOperad}\leavevmode
\begin{enumerate}[leftmargin=*]
\item The collection $\{\J_n\}_{n\geq 1}$ together with the action-composition maps $\circ_i$ and $\gamma_{i_1,\ldots,i_k}$ form an operadic bimodule over the operad $\{\K_n\}$ in the category $\PolySub$. 

\item The maps $\{\triangle_n : \J_n \to \J_n\times \J_n\}_{n\geq 1}$ form a morphism of $(\{\K_n\},\{\K_n\})$-operadic bimodules in the category $\PolySub$. 
\end{enumerate}
\end{theorem}
\end{samepage}

\begin{proof}
Using \cref{prop:thetacommutes}, we can apply the proof of \cite[Theorem~1]{MTTV19} \emph{mutatis mutandis}. The uniqueness of the transition map $\tr$ is the key argument, as it forces the operadic axioms to hold. We also point out that $\{ \J_n\times \J_n \}$ is to be understood as the Hadamard product $\{ \J_n \} \times \{ \J_n \}$, and that its $(\{\K_n\},\{\K_n\})$-operadic bimodule structure is defined as the pullback of its natural $(\{\K_n \times \K_n\},\{\K_n \times \K_n\})$-operadic bimodule structure under the diagonal maps $\{ \triangle_n : \K_n \to \K_n \times \K_n \}$.
\end{proof}

Point (1) of \cref{thm:MainOperad} was already mentioned in \cite[Section 1.2]{mazuir-I}, where associahedra and multiplihedra are realized as compactifications of moduli spaces of metric trees and used to construct $\Ainf$-structures on the Morse cochains of a closed manifold.


\section{Cellular formula for the diagonal of the multiplihedra} \label{sec:III}

We compute in \cref{thm:formuladiagonal} an explicit cellular formula for the diagonal of the Forcey--Loday multiplihedra, using again the key fact that the Ardila--Doker multiplihedron is a generalized permutahedron to which one can apply \cref{prop:refinementofnormalfans} and the results of \cite{LA21}. We then explain geometrically why this formula necessarily has to differ from the "magical formula" computed for the associahedra in \cite{MTTV19}.


\subsection{2-colored nested linear graphs} \label{ss:2-col}

Let $\gra$ be a \emph{linear graph} with $n$ vertices, as represented in \cref{fig:bijections}.
We respectively write $V(\gra)$ and $E(\gra)$ for its sets of vertices and edges.
Any subset of edges $N\subset E(\gra)$ defines a subgraph of $\gra$ whose edges are $N$ and whose vertices are all the vertices adjacent to an edge in $N$. 
We call this graph the \emph{closure} of~$N$. 

\begin{definition}[Nest and nesting]
\leavevmode
\begin{itemize}[leftmargin=*]
\item A \emph{nest} of a linear graph $\gra$ with $n$ vertices is a non-empty set of edges $N \subset E(\gra)$ whose closure is a connected subgraph of $\gra$.  
\item A \emph{nesting} of a linear graph $\gra$ is a set $\mathcal{N}=\{N_i\}_{i\in I}$ of nests such that 
\begin{enumerate}[leftmargin=*]
    \item the \emph{trivial nest} $E(\gra)$ is in $\mathcal{N}$,
    \item for every pair of nests $N_i\neq N_j$, we have either $N_i \subsetneq N_j$, $N_j \subsetneq N_i$ or $N_i \cap N_j = \emptyset$, and
    \item if $N_i \cap N_j = \emptyset$ then no edge of $N_i$ is adjacent to an edge of $N_j$.
\end{enumerate}
\end{itemize}
\end{definition}

Two nests that satisfy Conditions (2) and (3) are said to be \textit{compatible}. 
We denote the set of nestings of $\gra$ by $\mathcal{N}(\gra)$. 
We naturally represent a nesting by circling the closure of each nest as in \cref{fig:bijections}. 
A nesting is moreover \emph{atomic} if it has maximal cardinality $|\mathcal{N}|=|E(\gra)|$.

\begin{definition}[2-colored nesting] 
A \emph{2-colored nesting} is a nesting where each nest is either colored in blue, red or both red and blue (that is, purple), and which satisfy the following properties: 
\begin{enumerate}[leftmargin=*]
\item if a nest $N$ is blue or purple, then all nests contained in $N$ are blue, and 
\item if a nest $N$ is red or purple, then all nests that contain $N$ are red.
\end{enumerate}
\end{definition}
We call \emph{monochrome} the nests that are either blue or red, and \emph{bicolored} the purple nests. 
We denote by $\mono(\mathcal{N})$ the set of monochrome nests of a 2-colored nesting $\mathcal{N}$, and by $\mathcal{N}_2(\gra)$ the set of 2-colored nestings of $\gra$.
A 2-colored nesting is moreover \emph{atomic} if it has maximal cardinality, and it is made of monochrome nests only. 

\begin{remark} 
The data of a 2-colored nesting on a graph is equivalent to the data of a marked tubing on its line graph, as defined in \cite{DevadossForcey08}. See also \cite[Remark 2.4]{LA21}.
\end{remark}

\begin{lemma} 
\label{lemma:bijection}
There is a bijection between (2-colored) trees with $n$ leaves and (2-colored) nested linear graphs with $n$ vertices. 
Under this map, (2-colored) atomic trees are in bijection with atomic (2-colored) nested linear graphs.
\end{lemma}

\noindent Under this bijection, vertices of 2-colored trees correspond to nests, and their colors agree under the previous conventions. 

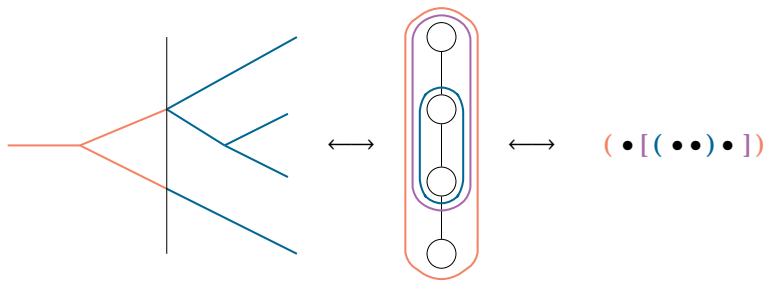
\begin{figure}[h!]
\resizebox{0.8\linewidth}{!}{
\begin{tikzpicture}

\node (b1) at (-2,3) {};
\node (b2)at (-2,2) {};
\node (b3)  at (-2,1) {};
\node (b4)  at (-2,0) {};
\node (b5)  at (-6,1.5) {};   

\draw[MidnightBlue,thick] (b2)--(-3,1.5) node {};
\draw[MidnightBlue,thick] (b3)--(-3,1.5) node {};
\draw[MidnightBlue,thick] (-3,1.5)--(-3.8,2) node {};

\draw[Red!60,thick] (-5,1.5)--(-3.8,2) node {};
\draw[Red!60,thick] (-5,1.5)--(-3.8,0.9) node {};
\draw[Red!60,thick] (-5,1.5)--(-6,1.5) node {};

\draw[MidnightBlue,thick] (-2,0)--(-3.8,0.9) node {};
\draw[MidnightBlue,thick] (-2,3)--(-3.8,2) node {};

\draw[-] (-3.8,3)--(-3.8,0) node {};

\node (B) at (-1.25,1.5) {$\longleftrightarrow$};

\node (x1) [circle,draw=none,minimum size=4mm,inner sep=0.1mm] at (0.15,2.55) {};
\node (x2) [circle,draw=none,minimum size=4mm,inner sep=0.1mm] at (0.15,1.5) {};
\node (x3) [circle,draw=none,minimum size=4mm,inner sep=0.1mm] at (0.15,0.38) {};

\node (t4)[circle,draw=black,minimum size=4mm,inner sep=0.1mm] at (-0,0) {};
\node (t3)[circle,draw=black,minimum size=4mm,inner sep=0.1mm] at (0,1) {};
\node (t2) [circle,draw=black,minimum size=4mm,inner sep=0.1mm] at (0,2) {};
\node (t1) [circle,draw=black,minimum size=4mm,inner sep=0.1mm] at (0,3) {};

\draw[-] (t4)--(t3) node {};
\draw[-] (t3)--(t2) node {};
\draw[-] (t2)--(t1) node {};

\draw [MidnightBlue,rounded corners,thick] (-0.15,0.7) -- (-0.3,0.9) -- (-0.3,2.1) -- (-0.15,2.3) -- (.15,2.3) -- (.3,2.1) -- (.3,.9) -- (.15,.7) -- cycle;
\draw [Purple!80,rounded corners,thick] (-0.15,0.6) -- (-0.4,0.8) -- (-0.4,3.1) -- (-0.15,3.3) -- (.15,3.3) -- (0.4,3.1) -- (.4,.8) -- (.15,.6) -- cycle;
\draw [Red!60,rounded corners,thick] (-0.15,-0.35) -- (-0.5,-0.1) -- (-0.5,3.15) -- (-0.15,3.4) -- (.15,3.4) -- (0.5,3.15) -- (0.5,-0.1) -- (0.15,-0.35) -- cycle;

\node (C) at (1.25,1.5) {$\longleftrightarrow$};

\node (D) at (3.35,1.5) {$\red{ \bullet \purple{\blue{\bullet \bullet} \bullet}}$};

\end{tikzpicture}}
\caption{Bijections between 2-colored trees, 2-colored nested linear graphs, and 2-colored parenthesizations.}
\label{fig:bijections}
\end{figure}

\subsection{Cellular formula for the diagonal} \label{ss:cellular-formula}

\begin{definition}
Let $(\gra,\mathcal{N})$ be a nested linear graph. 
We respectively denote by $B(\mathcal{N})$, $P(\mathcal{N})$ and $R(\mathcal{N})$ the set of blue, purple and red nests of $\mathcal{N}$. 
We define $Q(\mathcal{N})$ to be the set whose elements are the unions of nests
\[   
\bigcup_{i=1}^k R_i \cup \bigcup_{B \in B(\mathcal{N})} B \cup 
  \bigcup_{P \in P(\mathcal{N})} P 
\]
where $R_1,\ldots,R_k \in R(\mathcal{N})$, for any $0 \leq k \leq |R(\mathcal{N})|$, the case $\cup R_i = \emptyset$ being allowed, and where two unions that result in the same set are identified.
\end{definition}

We number the edges of the linear graph with $n$ vertices from bottom to top as represented in \cref{fig:bijections}, starting at $1$ and ending at $n-1$.  
To each blue nest $B \in B(\mathcal{N})$ in a 2-colored nesting $\mathcal{N}$ of a linear graph with $n$ vertices, we associate the \emph{characteristic vector} $\vec B\in \RR^n$ which has a $1$ in position $i$ if $i \in B$, $0$ in position $i$ if $i \notin B$ and 0 in position $n$.
To each union of nests $Q \in Q(\mathcal{N})$, we associate the characteristic vector $\vec Q \in \RR^n$ which has a $1$ in position $i$ if $i \in Q$, $0$ in position $i$ if $i \notin Q$ and 1 in position $n$. 
We denote moreover by $\vec n$ the vector $(1,\ldots,1) \in \RR^n$.

\begin{lemma} 
\label{lemma:normalcones}
The normal cone of the face of the Ardila--Doker realization of the multiplihedron labeled by the 2-colored nesting $\mathcal{N}$ is given by \[\cone\left(\{-\vec B\}_{B \in B(\mathcal{N})} \cup \{-\vec Q\}_{Q \in Q(\mathcal{N})} \cup \{\vec n, - \vec n\} \right) \ . \]
\end{lemma} 

\begin{proof} 
This follows from the description of the Ardila--Doker multiplihedron as a generalized permutahedron: 
the normal cone of a face of the multiplihedron is a union of normal cones of faces of the permutahedron, and these faces can be easily determined from the projection from the permutahedron to the multiplihedron, written down explicitly in the proof of \cite[Theorem 3.3.6]{Doker11}.
\end{proof}

We are now ready to compute the cellular formula for the diagonal of the Forcey--Loday multiplihedra. We introduce \[ D(n)\coloneqq \{(I,J) \ | \ I,J\subset\{1,\ldots,n\}, |I|=|J|, I\cap J=\emptyset, \min(I\cup J)\in I \}. \]
We number again the edges of the linear graph with $n$ vertices from bottom to top, starting at $1$ and ending at $n-1$. 
Blue nests and unions of blue, purple and red nests can then in particular be seen as subsets of $\{1,\ldots,n-1\}$, hence of $\{1,\ldots,n\}$.

\begin{samepage}
\begin{theorem}
\label{thm:formuladiagonal}
The cellular image of the diagonal map $\triangle_n : \J_n \to \J_n \times \J_n$ introduced in \cref{def:diagonal-multipl-forcey-loday} admits the following description.
For $\mathcal{N}$ and $\mathcal{N}'$ two 2-colored nestings of the linear graph with $n$ vertices, we have that
\begin{eqnarray*}
  (\mathcal{N},\mathcal{N}') \in \Ima\triangle_n 
  & \iff & \forall (I,J) \in D(n), \\
  && \exists B \in B(\mathcal{N}), |B\cap I|>|B\cap J| \text{ or }  \\
  && \exists Q \in Q(\mathcal{N}), |(Q\cup \{n\}) \cap I|>| (Q\cup \{n\}) \cap J| \text{ or }  \\
  && \exists B' \in B(\mathcal{N}'), |B'\cap I|<|B'\cap J| \text{ or } \nonumber \\
  && \exists Q' \in Q(\mathcal{N}'), |(Q'\cup \{n\}) \cap I|<| (Q'\cup \{n\}) \cap J| \ .
\end{eqnarray*}
\end{theorem}
\end{samepage}

\begin{proof} 
The essential ingredient is the computation of the fundamental hyperplane arrangement of the permutahedron, which was done in \cite[Section 3.1]{LA21}. The result follows in three steps:
\begin{enumerate}[leftmargin=*]
\item Since a good orientation vector $\vec v$ is also a principal orientation vector \cite[Definition 3.15]{LA21}, it orients positively the permutahedron. 
\item Using \cref{prop:refinementofnormalfans} and the description of the normal cones of the faces of the multiplihedron in \cref{lemma:normalcones}, we get the above formula for the Ardila--Doker realizations of the multiplihedra. 
\item \cref{prop:goodprojection} garantees that this formula holds for the Forcey--Loday realizations, which completes the proof.
\end{enumerate}
\end{proof}

We now make this formula explicit in dimension 1, 2 and 3. 
We write 2-colored nestings of a linear graph with $n$ vertices as 2-colored parenthesizations of a word with $n$ symbols $\bullet$, which are easier to read and shorter to type, see \cref{fig:bijections}. 
For the sake of readability, we use brackets instead of parentheses for the purple nests in this representation.
We moreover only write pairs of faces $(F,G)$ such that $\dim F + \dim G = \dim P$.
\begin{equation*}
  \begin{matrix}
      \triangle_2(\purple{\bullet \bullet}) & = & \blue{\bullet \bullet} \times \purple{\bullet \bullet} \cup \purple{\bullet \bullet} \times \red{\bullet \bullet}
  \end{matrix}
\end{equation*}
\[ \resizebox{\hsize}{!}{$\displaystyle{ 
 \renewcommand*{\arraystretch}{1.5}
  \begin{matrix}
      \triangle_3(\purple{\bullet \bullet \bullet}) 
      & = & \blue{\blue{\bullet \bullet} \bullet} \times \purple{\bullet \bullet \bullet} 
      & \cup & \purple{\bullet \bullet \bullet} \times \red{\bullet \red{\bullet \bullet}}
      & \cup & \blue{\bullet \bullet \bullet} \times \purple{\bullet \blue{\bullet \bullet}} \\
      & \cup & \blue{\bullet \bullet \bullet} \times \red{\bullet \purple{\bullet \bullet}}  
      & \cup & \purple{\bullet \blue{\bullet \bullet}} \times \red{\bullet \purple{\bullet \bullet}} 
      & \cup & \purple{\blue{\bullet \bullet} \bullet} \times \red{\purple{\bullet \bullet} \bullet} \\
      & \cup & \purple{\blue{\bullet \bullet} \bullet} \times \red{\bullet \bullet \bullet} 
      & \cup & \red{\purple{\bullet \bullet} \bullet} \times \red{\bullet \bullet \bullet} 
  \end{matrix} }$} \]

\[ \resizebox{\hsize}{!}{$\displaystyle{ 
  \renewcommand*{\arraystretch}{1.5}
  \begin{matrix}
      & & & & \triangle_4(\purple{\bullet \bullet \bullet \bullet}) =  \\
      &  & \blue{\blue{\blue{\bullet \bullet}\bullet}\bullet} \times \purple{\bullet \bullet \bullet \bullet} 
      & \cup & \purple{\bullet \bullet \bullet \bullet} \times \red{\bullet \red{\bullet \red{\bullet \bullet}}}
      & \cup & \blue{\blue{\bullet \bullet \bullet}\bullet} \times \purple{\bullet \blue{\bullet \bullet}\bullet} \\ 
      & \cup & \red{\purple{\bullet \bullet}\purple{\bullet \bullet}} \times \red{\bullet \bullet\red{\bullet \bullet}} 
      & \cup & \blue{\blue{\bullet \bullet \bullet}\bullet} \times \purple{\bullet \blue{\bullet \bullet \bullet}} 
      & \cup & \red{\purple{\bullet \bullet}\bullet \bullet} \times \red{\bullet \bullet\red{\bullet \bullet}} \\
      & \cup & \blue{\bullet \blue{\bullet \bullet}\bullet} \times \purple{\bullet \blue{\bullet \bullet \bullet}} 
      & \cup & \red{\purple{\bullet \bullet \bullet}\bullet} \times \red{\bullet \red{\bullet \bullet}\bullet}
      & \cup & \blue{\blue{\bullet \bullet}\bullet \bullet} \times \purple{\bullet \bullet \blue{\bullet \bullet}} \\
      & \cup & \red{\purple{\bullet \bullet \bullet}\bullet} \times \red{\bullet \red{\bullet \bullet \bullet}} 
      & \cup & \purple{\blue{\blue{\bullet \bullet}\bullet}\bullet} \times \red{\purple{\bullet \bullet \bullet}\bullet}
      & \cup & \purple{\bullet \bullet\blue{\bullet \bullet}} \times \red{\bullet \red{\bullet \purple{\bullet \bullet}}} \\
      & \cup & \purple{\blue{\bullet \bullet}\blue{\bullet \bullet}} \times \red{\purple{\bullet \bullet}\purple{\bullet \bullet}}
      & \cup & \purple{\bullet \blue{\bullet \bullet}\bullet} \times \red{\bullet \red{\purple{\bullet \bullet}\bullet}}
      & \cup & \blue{\blue{\bullet \bullet}\bullet \bullet} \times \red{\purple{\bullet \bullet}\purple{\bullet \bullet}} \\
      & \cup & \purple{\bullet \blue{\bullet \bullet} \bullet} \times \red{\bullet \red{\bullet \bullet \bullet}}
      & \cup & \purple{\bullet \blue{\blue{\bullet \bullet}\bullet}} \times \red{\bullet \purple{\bullet \bullet\bullet}} 
      & \cup & \purple{\blue{\bullet \bullet}\bullet \bullet} \times \red{\purple{\bullet \bullet}\red{\bullet \bullet}}\\
      & \cup & \blue{\bullet \blue{\bullet \bullet} \bullet} \times \red{\bullet \purple{\bullet\bullet \bullet}}
      & \cup & \blue{\blue{\bullet \bullet \bullet}\bullet} \times \red{\bullet \purple{\bullet \bullet \bullet}}
      & \cup & \purple{\blue{\bullet \bullet}\bullet \bullet} \times \red{\bullet \bullet\red{\bullet \bullet}} \\
      & \cup & \red{\purple{\blue{\bullet \bullet}\bullet}\bullet} \times \red{\purple{\bullet \bullet}\bullet \bullet}
      & \cup & \purple{\bullet \blue{\bullet \bullet \bullet}} \times \red{\bullet \purple{\bullet \blue{\bullet \bullet}}}
      & \cup & \purple{\blue{\blue{\bullet \bullet}\bullet}\bullet} \times \red{\purple{\bullet \bullet}\bullet \bullet} \\
      & \cup & \purple{\bullet \blue{\bullet \bullet \bullet}} \times \red{\bullet \red{ \bullet \purple{\bullet \bullet}}} 
      & \cup & \red{\bullet \purple{\bullet \bullet}\bullet} \times \red{\bullet \red{\bullet \bullet \bullet}}
      & \cup & \red{\red{\purple{\bullet \bullet}\bullet }\bullet} \times \red{\bullet \bullet \bullet \bullet} \\
      & \cup & \blue{\bullet \bullet \bullet \bullet} \times \purple{\bullet \blue{\bullet\blue{\bullet \bullet}}} 
      & \cup &   \red{\purple{\blue{\bullet \bullet}\bullet}\bullet} \times \red{\bullet \bullet \bullet \bullet} 
      & \cup & \blue{\bullet \bullet \bullet \bullet} \times \red{\bullet \purple{\bullet \blue{\bullet \bullet}}} \\
      & \cup & \purple{\blue{\blue{\bullet \bullet}\bullet}\bullet} \times \red{\bullet \bullet \bullet \bullet} 
      & \cup & \blue{\bullet \bullet \bullet \bullet} \times \red{\bullet\red{\bullet\purple{\bullet \bullet}}} 
      & \cup & \red{\purple{\bullet \bullet}\blue{\bullet \bullet}} \times \red{\bullet \bullet \purple{\bullet \bullet}} \\
      & \cup & \purple{\blue{\bullet \bullet \bullet}\bullet} \times \red{\purple{\bullet \blue{\bullet \bullet}}\bullet}
      & \cup & \purple{\blue{\bullet \bullet}\blue{\bullet \bullet}} \times \red{\bullet \bullet \purple{\bullet \bullet}}
      & \cup & \purple{\blue{\bullet \bullet \bullet}\bullet} \times \red{\bullet \red{\purple{\bullet \bullet}\bullet}} \\ 
      & \cup & \purple{\blue{\bullet \bullet \bullet}\bullet} \times \red{\bullet \blue{\bullet \bullet}\bullet} 
      & \cup & \blue{\blue{\bullet \bullet}\bullet \bullet} \times \red{\bullet \bullet\purple{\bullet \bullet}}
      & \cup & \purple{\blue{\bullet \bullet \bullet}\bullet} \times \red{\bullet \red{\bullet \bullet \bullet}} \\
      & \cup & \red{\purple{\bullet \blue{\bullet \bullet}}\bullet} \times \red{\bullet \purple{\bullet \bullet}\bullet}
      & \cup & \red{\blue{\bullet \bullet \bullet}\bullet} \times \red{\bullet \purple{\bullet \bullet}\bullet}  
      & \cup & \purple{\blue{\blue{\bullet \bullet}\bullet}\bullet} \times \red{\bullet \purple{\bullet \bullet}\bullet} 
  \end{matrix} }$} \]
  
We also compute in \cref{table:numerology} the number of faces of complementary dimensions and the number of pairs of vertices in the cellular image of the diagonal of the multiplihedra in dimensions $0$ to $6$.
They are compared with the diagonals induced by the same orientation vector on the Loday associahedra and the permutahedra. 
The two sequences of numbers that we obtain did not appear before in \cite{OEIS}.

\medskip

\begin{figure}[h]
\resizebox{\hsize}{!}{$\displaystyle{
\begin{tabular}{c|c|rrrrrrr|l}
\textbf{Pairs $(F,G) \in \Ima\triangle_{(P,\vec v)}$} & \textbf{Polytopes} & \textbf{0} & \textbf{1} & \textbf{2} & \textbf{3} & \textbf{4} & \textbf{5} & \textbf{6} & \textbf{\cite{OEIS}} \\
\hline
& \text{Assoc.} & 1 & 2 & 6 & 22 & 91 & 408 & 1938 & \OEIS{A000139}  \\
$\dim F + \dim G= \dim P$  & \text{Multipl.} & 1 & 2 & 8 & 42 & 254 & 1678 & 11790 &  to appear \\
 & \text{Permut.} & 1 & 2 & 8 & 50 & 432 & 4802 & 65536 &  \OEIS{A007334} \\
\hline
 & \text{Assoc.} & 1 & 3 & 13 & 68 & 399 & 2530 & 16965 &  \OEIS{A000260} \\
 $\dim F=\dim G=0$ & \text{Multipl.} & 1 & 3 & 17 & 122 & 992 & 8721 & 80920 & to appear \\
 & \text{Permut.} & 1 & 3 & 17 & 149 & 1809 & 28399 & 550297 &  \OEIS{A213507} 
\end{tabular}}$}
\caption{Number of pairs of faces in the cellular image of the diagonal of the associahedra, multiplihedra and permutahedra of dimension $0\leq \dim P \leq 6$, induced by any good orientation vector.}
\label{table:numerology}
\end{figure}

\subsection{About the cellular formula} 
\label{ss:about}

Given a face $F$ of a positively oriented polytope $(P, \vec v)$, the orientation vector $\vec v$ defines a unique vertex $\tp F$ (resp. $\bm F$) which maximizes (resp. minimizes) the scalar product $\langle - , \vec v \rangle$ over $F$. 
By \cite[Proposition 1.17]{LA21}, any pair of faces $(F,G) \in \Ima \triangle_{(P,\vec v)}$ satisfies $\tp F \leq \bm G$. 
In the case of the simplices, the cubes and the associahedra, the converse also holds: the image of the diagonal is given by the "magical formula" 
\begin{align}
\label{eq:magical-formula}
(F,G) \in \Ima \triangle_n \iff \tp F\leq \bm G \ . 
\end{align}
This formula, however, does not hold for the diagonal of the Forcey--Loday multiplihedra. 

\begin{example} \label{prop:pas-top-bot}
The diagonal on the multiplihedron $\J_4$ is such that
\[ \Ima \triangle_4 \subsetneq \{ (F,G) , \ \tp F\leq \bm G \} \ . \]
Indeed, the pairs of faces $(F,G)$ that satisfy $\dim F + \dim G = 3$ and $\tp F\leq \bm G$ include the four pairs 
\begin{equation} \label{eq:quatre-paires-inclues}
\begin{matrix}
\purple{\blue{\bullet \bullet \bullet}\bullet} \times \red{\purple{\bullet\blue{\bullet\bullet}}\bullet} & 
\red{\blue{\bullet\bullet\bullet}\bullet} \times \red{\bullet\purple{\bullet\bullet}\bullet}  \\
\purple{\blue{\bullet\bullet\bullet}\bullet} \times \red{\bullet\blue{\bullet\bullet}\bullet} &
\red{\purple{\bullet\blue{\bullet\bullet}}\bullet} \times \red{\bullet\purple{\bullet\bullet}\bullet}  
\end{matrix}
\end{equation}
and the four pairs
\begin{equation} \label{eq:quatre-paires-exclues}
\begin{matrix}
\purple{\blue{\bullet \bullet \bullet}\bullet} \times \red{\red{\bullet\purple{\bullet\bullet}}\bullet} & 
\blue{\blue{\bullet\bullet\bullet}\bullet} \times \red{\bullet\purple{\bullet\bullet}\bullet}  \\
\purple{\blue{\bullet\bullet\bullet}\bullet} \times \red{\bullet\red{\bullet\bullet}\bullet} &
\purple{\blue{\bullet\blue{\bullet\bullet}}\bullet} \times \red{\bullet\purple{\bullet\bullet}\bullet}  \ .
\end{matrix}
\end{equation}
While the image $\Ima \triangle_4$ contains the four pairs in (\ref{eq:quatre-paires-inclues}), it does \emph{not} include the four pairs in (\ref{eq:quatre-paires-exclues}), as can be checked directly from \cref{thm:formuladiagonal}.
\end{example}

\begin{remark}
We point out that Formula (\ref{eq:magical-formula}) also does not hold neither for the permutahedra nor the operahedra in general, as proven in \cite[Section 3.2]{LA21}. 
\end{remark}

The diagonal $\triangle_n$ being a section of the projection $\pi : \J_n \times \J_n \to \J_n , (x,y) \mapsto (x+y)/2$ \cite[Proposition 1.1]{LA21}, one can in fact represent its cellular image by projecting it to $\J_n$: for each pair of faces $(F,G) \in \Ima \triangle_n$, one draws the polytope $(F+G)/2$ in $\J_n$. This defines a polytopal subdivision of $\J_n$. The polytopal subdivision of $\J_3$ can be found in \cite[Figure 3]{LA21}, while the polytopal subdivision of $\J_4$ is illustrated on the first page of this article.

\cref{prop:pas-top-bot} can then be illustrated geometrically as follows. 
There are two distinct diagonals on $\J_4$ which agree with the Tamari-type order on the vertices. 
The first one, corresponding to the diagonal defined in this paper, is induced by the choice of any orientation vector $\vec v=(v_1,v_2,v_3,v_4)$ satisfying $v_1>v_2>v_3>v_4$ and $v_1 + v_4 > v_2+v_3$ (here we work with the Ardila--Doker realization of the multiplihedron).
Changing the last condition to $v_1 + v_4 < v_2+v_3$ gives the second choice of diagonal, which is in fact exactly the diagonal of Saneblidze--Umble \cite[Section 5]{SaneblidzeUmble04}. 
These two diagonals on $\J_4$ then differ by four pairs of faces, as represented in~\cref{fig:four-pairs}: the first diagonal includes the pairs of~(\ref{eq:quatre-paires-inclues}), while the second diagonal includes the pairs of~(\ref{eq:quatre-paires-exclues}).
Under the projection $\pi : \J_4 \times \J_4 \to \J_4, (x,y) \mapsto (x+y)/2$, these two families of faces induce two distinct polytopal subdivisions of the same "diamond" inside $\J_4$, represented in \cref{fig:diamonds}. 
We also refer to the last paragraph of \cref{ss:diagonals} for an algebraic counterpart of \cref{prop:pas-top-bot}.

\begin{remark}
  The two previous families of orientation vectors correspond to two adjacent chambers in the fundamental hyperplane arrangement of the permutahedron \cite[Theorem 3.6]{LA21}, separated by the hyperplane $x_1+x_4=x_2+x_3$, pictured in blue in \cite[Figure 12]{LA21}.
  A way to relate the diagonal constructed in this article to the diagonal of \cite[Section 5]{SaneblidzeUmble04} would possibly be to find further choices of chambers in the fundamental hyperplane arrangements of the permutahedra (or the multiplihedra) in all dimensions $n \geq 4$ recovering the latter diagonal, see also \cite[Remark~3.19]{LA21}.
\end{remark}

\begin{figure}[h]
\resizebox{0.7\linewidth}{!}{
\begin{tikzpicture}[scale=0.5, J4]

\draw (1,2,3) node {$\bullet$}; 
\draw (6,4,2) node {$\bullet$}; 

\draw[thick, opacity=0.2] (1,4,1)--(1,2,3)--(2,1,3)--(3,1,2)--(3,2,1)--cycle;
\draw[thick, opacity=0.2] (2,8,2)--(2,4,6)--(4,2,6)--(6,2,4)--(6,4,2)--cycle;
\draw[thick, opacity=0.2] (4,1,6)--(6,1,4)--(6,1,2)--(6,2,1)--(6,4,1)--(2,8,1)--(1,8,1)--(1,8,2)--(1,4,6)--(1,2,6)--(2,1,6)--cycle;
\draw[thick, opacity=0.2] (4,1,6)--(4,2,6);
\draw[thick, opacity=0.2] (6,1,2)--(3,1,2);
\draw[thick, opacity=0.2] (6,2,1)--(3,2,1);
\draw[thick, opacity=0.2] (6,4,1)--(6,4,2);
\draw[thick, opacity=0.2] (2,8,1)--(2,8,2)--(1,8,2);
\draw[thick, opacity=0.2] (1,8,1)--(1,4,1);
\draw[thick, opacity=0.2] (1,4,6)--(2,4,6);

\draw[very thick, blue] (4,1,6)--(6,1,4);
\draw[very thick, blue] (6,1,4)--(6,2,4);

\draw[very thick, blue] (1,2,6)--(1,2,3);
\draw[very thick, blue] (1,2,6)--(2,1,6);
\draw[very thick, blue] (2,1,6)--(2,1,3);
\draw[very thick, blue] (1,2,3)--(2,1,3);
\draw[fill=blue, opacity=0.12] (1,2,3)--(2,1,3)--(2,1,6)--(1,2,6)--cycle;
\end{tikzpicture}

\begin{tikzpicture}[scale=0.5, J4]

\draw[thick, opacity=0.2] (1,4,1)--(1,2,3)--(2,1,3)--(3,1,2)--(3,2,1)--cycle;
\draw[thick, opacity=0.2] (2,8,2)--(2,4,6)--(4,2,6)--(6,2,4)--(6,4,2)--cycle;
\draw[thick, opacity=0.2] (4,1,6)--(6,1,4)--(6,1,2)--(6,2,1)--(6,4,1)--(2,8,1)--(1,8,1)--(1,8,2)--(1,4,6)--(1,2,6)--(2,1,6)--cycle;
\draw[thick, opacity=0.2] (4,1,6)--(4,2,6);
\draw[thick, opacity=0.2] (6,1,2)--(3,1,2);
\draw[thick, opacity=0.2] (6,2,1)--(3,2,1);
\draw[thick, opacity=0.2] (6,4,1)--(6,4,2);
\draw[thick, opacity=0.2] (2,8,1)--(2,8,2)--(1,8,2);
\draw[thick, opacity=0.2] (1,8,1)--(1,4,1);
\draw[thick, opacity=0.2] (1,4,6)--(2,4,6);
\draw[thick, opacity=0.2] (2,1,6)--(2,1,3);
\draw[thick, opacity=0.2] (1,2,3)--(2,1,3);

\draw[very thick, blue] (4,1,6)--(6,1,4);
\draw[very thick, blue] (6,1,4)--(6,2,4);
\draw[very thick, blue] (4,1,6)--(4,2,6);
\draw[very thick, blue] (4,2,6)--(6,2,4);
\draw[fill=blue, opacity=0.12] (4,1,6)--(6,1,4)--(6,2,4)--(4,2,6)--cycle;

\draw[very thick, blue] (1,2,6)--(1,2,3);
\draw[very thick, blue] (1,2,6)--(2,1,6);
\end{tikzpicture}}
\[\]
\resizebox{0.7\linewidth}{!}{
\begin{tikzpicture}[scale=0.5, J4]

\draw[thick, opacity=0.2] (1,4,1)--(1,2,3)--(2,1,3)--(3,1,2)--(3,2,1)--cycle;
\draw[thick, opacity=0.2] (2,8,2)--(2,4,6)--(4,2,6)--(6,2,4)--(6,4,2)--cycle;
\draw[thick, opacity=0.2] (4,1,6)--(6,1,4)--(6,1,2)--(6,2,1)--(6,4,1)--(2,8,1)--(1,8,1)--(1,8,2)--(1,4,6)--(1,2,6)--(2,1,6)--cycle;
\draw[thick, opacity=0.2] (4,1,6)--(4,2,6);
\draw[thick, opacity=0.2] (6,1,2)--(3,1,2);
\draw[thick, opacity=0.2] (6,2,1)--(3,2,1);
\draw[thick, opacity=0.2] (6,4,1)--(6,4,2);
\draw[thick, opacity=0.2] (2,8,1)--(2,8,2)--(1,8,2);
\draw[thick, opacity=0.2] (1,8,1)--(1,4,1);
\draw[thick, opacity=0.2] (1,4,6)--(2,4,6);
\draw[thick, opacity=0.2] (6,1,4)--(6,2,4);

\draw[very thick, red] (4,1,6)--(4,2,6);
\draw[very thick, red] (4,2,6)--(6,2,4);

\draw[very thick, red] (1,2,6)--(1,2,3);
\draw[very thick, red] (1,2,6)--(2,1,6);
\draw[very thick, red] (2,1,6)--(2,1,3);
\draw[very thick, red] (1,2,3)--(2,1,3);
\draw[fill=red, opacity=0.12] (1,2,3)--(2,1,3)--(2,1,6)--(1,2,6)--cycle;
\end{tikzpicture}

\begin{tikzpicture}[scale=0.5, J4]

\draw[thick, opacity=0.2] (1,4,1)--(1,2,3)--(2,1,3)--(3,1,2)--(3,2,1)--cycle;
\draw[thick, opacity=0.2] (2,8,2)--(2,4,6)--(4,2,6)--(6,2,4)--(6,4,2)--cycle;
\draw[thick, opacity=0.2] (4,1,6)--(6,1,4)--(6,1,2)--(6,2,1)--(6,4,1)--(2,8,1)--(1,8,1)--(1,8,2)--(1,4,6)--(1,2,6)--(2,1,6)--cycle;
\draw[thick, opacity=0.2] (4,1,6)--(4,2,6);
\draw[thick, opacity=0.2] (6,1,2)--(3,1,2);
\draw[thick, opacity=0.2] (6,2,1)--(3,2,1);
\draw[thick, opacity=0.2] (6,4,1)--(6,4,2);
\draw[thick, opacity=0.2] (2,8,1)--(2,8,2)--(1,8,2);
\draw[thick, opacity=0.2] (1,8,1)--(1,4,1);
\draw[thick, opacity=0.2] (1,4,6)--(2,4,6);
\draw[thick, opacity=0.2] (2,1,6)--(2,1,3);
\draw[thick, opacity=0.2] (1,2,3)--(2,1,3);
\draw[thick, opacity=0.2] (1,2,6)--(1,2,3);
\draw[thick, opacity=0.2] (1,2,6)--(2,1,6);

\draw[very thick, red, -] (4,1,6)--(6,1,4);
\draw[very thick, red, -] (6,1,4)--(6,2,4);
\draw[very thick, red] (4,1,6)--(4,2,6);
\draw[very thick, red, -] (4,2,6)--(6,2,4);
\draw[fill=red, opacity=0.12] (4,1,6)--(6,1,4)--(6,2,4)--(4,2,6)--cycle;

\draw[very thick, red, -] (2,1,3)--(1,2,3);
\draw[very thick, red] (2,1,3)--(2,1,6);
\end{tikzpicture}}
\caption{The four pairs of~(\ref{eq:quatre-paires-inclues}) represented in blue on the two top copies of $\J_4$ and the four pairs of~(\ref{eq:quatre-paires-exclues}) represented in red on the two bottom copies of $\J_4$. 
The minimal (top right) and maximal (bottom left) vertices for the Tamari-type order are drawn in black, in the top left copy.}
\label{fig:four-pairs}
\end{figure}
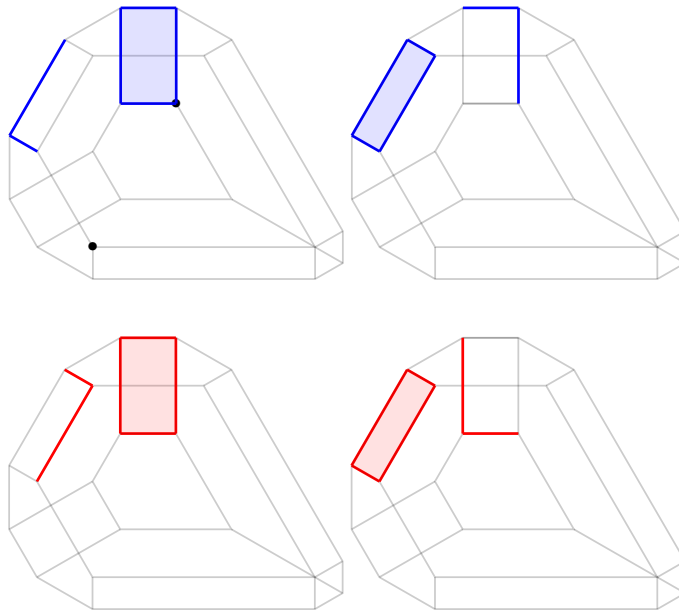

\begin{figure}[h!]
\centering
\begin{subfigure}{0.4\textwidth}
    \centering
        \includegraphics[width=0.6\linewidth]{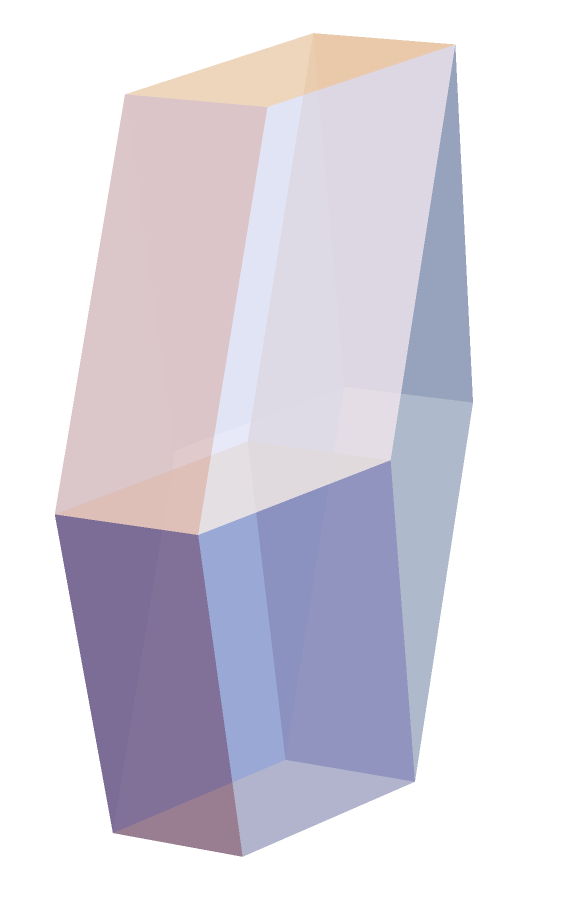} 
    \end{subfigure} ~
    \begin{subfigure}{0.4\textwidth}
    \centering
       \includegraphics[width=0.6\linewidth]{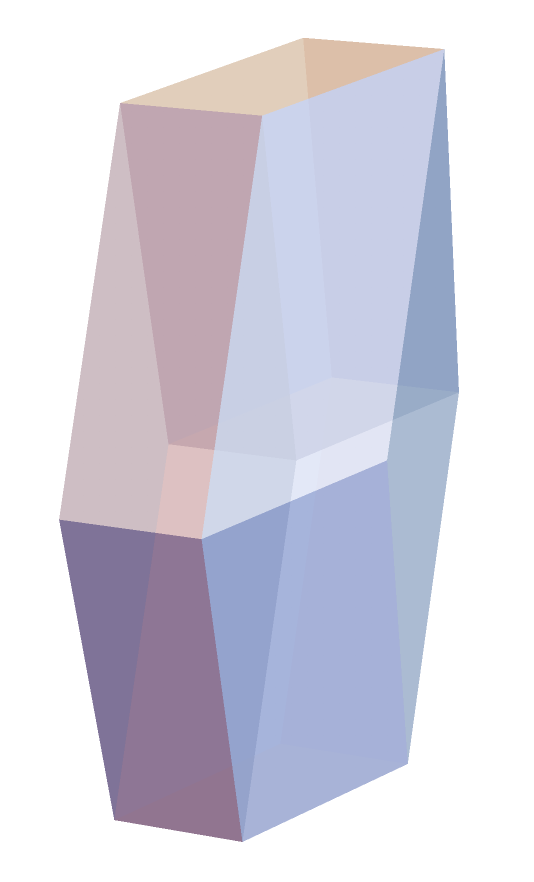}
    \end{subfigure}
\caption{The two distinct subdivisions of the same "diamond" in $\J_4$, respectively induced by sets of pairs~(\ref{eq:quatre-paires-inclues}) and~(\ref{eq:quatre-paires-exclues}) from \cref{prop:pas-top-bot}.}
\label{fig:diamonds}
\end{figure}


\section{Tensor product of \Ainf -morphisms and \Ainf -functors}
\label{sec:IV}

We begin by proving that for a certain choice of cellular orientation, the cellular chains functor maps the Loday associahedra to the operad \Ainf\ encoding \Ainf -algebras and maps the Forcey--Loday multiplihedra to the operadic bimodule \Minf\ encoding \Ainf -morphisms between them.
It then maps the respective geometric diagonals to algebraic ones, which can be used to define compatible tensor products of $\Ainf$-algebras and $\Ainf$-morphisms (with signs). 
Tensor product of \Ainf -categories and \Ainf -functors are defined in a similar fashion, and we relate them to the different notions of \Ainf -categories with identities.
We finally study coassociativity, cocommutativity and compatibility with composition of \Ainf -morphisms for these diagonals. 
We show that these properties are always satisfied up to homotopy, hinting at the idea that the category $\infAalg$ should possess some kind of \textit{homotopy} symmetric monoidal structure.

\subsection{\Ainf -algebras and \Ainf -morphisms} \label{ss:ainf-alg-ainf-morph}

\subsubsection{Definitions}

We work in the rest of this article with homological convention. 
We will refer to chain complexes as \emph{dg modules}, where the abbreviation dg stands for "differential graded", and their differential will always have degree $-1$.
For a dg module $(A,\partial)$, we endow the graded module $\Hom(A^{\otimes n}, A)$ with the differential $[\partial,f] := \partial f - (-1)^{|f|}f\partial$.

\begin{definition}[$\mathrm{A}_\infty$-algebra] \label{def:ainf-alg} An \emph{$\mathrm{A}_\infty$-algebra} is the data of a dg module $(A,\partial)$ together with operations \[ m_n : A^{\otimes n} \to A \ , \ n \geq 2 \] of degree $|m_n|=n-2$, satisfying the equations 
\[ [ \partial , m_n ] = - \sum_{\substack{p+q+r=n \\ 2 \leq q \leq n-1}} (-1)^{p+qr}m_{p+1+r}(\id^{\otimes p} \otimes m_q \otimes \id^{\otimes r}) \ , \ n\geq 2 \ . \]
\end{definition}

\begin{definition}[$\mathrm{A}_\infty$-morphism] \label{def:ainf-morph}
An \emph{$\mathrm{A}_\infty$-morphism} $F : A\rightsquigarrow B$ between two $\mathrm{A}_\infty$-algebras $(A,\{m_n\})$ and $(B,\{m_n'\})$ is a family of linear maps \[f_n : A^{\otimes n} \to B \ , \ n \geq 1\] of degree $|f_n|=n-1$, satisfying the equations  \[
 [ \partial , f_n] =  \sum_{\substack{p+q+r=n \\ q \geq 2}} (-1)^{p+qr}f_{p+1+r}(\id^{\otimes p} \otimes m_q \otimes \id^{\otimes r}) \ -  \sum_{\substack{i_1+\cdots+i_k=n \\ k \geq 2}} (-1)^{\varepsilon} m_k'(f_{i_1}\otimes\cdots\otimes f_{i_k})  \ , \] for $n \geqslant 1$, where $\varepsilon = \sum_{u=1}^{k}(k-u)(1-i_u)$.
\end{definition}

For three $\Ainf$-algebras $A$, $B$, $C$ and two $\Ainf$-morphisms $F : A \rightsquigarrow B$, $B \rightsquigarrow C$, their composition $G \circ F : A \rightsquigarrow C$ is the $\Ainf$-morphism whose operation of arity $n$ is given by the formula
\[ (G \circ F)_n := \sum_{i_1+\cdots+i_k=n} (-1)^{\varepsilon} g_k(f_{i_1}\otimes\cdots\otimes f_{i_k})  \ . \]
This composition is associative. We moreover point out that a standard \textit{dg (associative) algebra} can be defined as an \Ainf -algebra whose higher operations $m_n$ vanish for $n \geq 3$. For more details on these notions, we refer to \cite[Chapter 9]{LodayVallette12}. 

\begin{definition}
We denote by $\infAalg$ the category of $\Ainf$-algebras with $\Ainf$-morphisms.
\end{definition}

Representing the operations $m_n$ as corollae \arbreop{0.15} of arity $n$, the equations of \cref{def:ainf-alg} read as
\begin{equation}
    [ \partial , \arbreop{0.15} ] = - \sum_{\substack{p+q+r=n \\ 2 \leq q \leq n-1}} (-1)^{p+qr} \eqainf   \ .  \label{eq:ainf-alg}
\end{equation} 
Representing the operations $m_n$ in blue \arbreopbleu{0.15}, the operations $m'_n$ in red \arbreoprouge{0.15} and the operations $f_n$ by \arbreopmorph{0.15}, the equations of \cref{def:ainf-morph} can be rewritten as
\begin{align} 
[ \partial , \arbreopmorph{0.15} ] = \sum_{\substack{p+q+r=n \\ q \geq 2}} (-1)^{p+qr} \eqainfmorphun \ -  \sum_{\substack{i_1+\cdots+i_k=n \\ k \geq 2}} (-1)^{\varepsilon} \eqainfmorphdeux \ .  \label{eq:ainf-morph}
\end{align}
Finally, representing the operations $f_n$ by \arbreopmorphcompun\ and the operations $g_n$ by \arbreopmorphcompdeux, the formula for the composition of \Ainf -morphisms reads as
\begin{align}
     \sum_{i_1+\cdots+i_k=n} (-1)^{\varepsilon} \compainf \ .  \label{eq:ainf-comp}
\end{align}

\subsubsection{The operad \Ainf\ and the operadic bimodule \Minf} \label{sss:operad-ainf-operadic-bimod-minf}

\begin{definition}[Operad \Ainf]
The \emph{operad \Ainf} is the quasi-free dg operad generated in arity $n \geq 2$ by one operation $\arbreop{0.15}$ of degree $n-2$ 
\[ \Ainf := \left( \mathcal{T}( \arbreopdeux , \arbreoptrois, \arbreopquatre , \cdots ) , \partial \right) \ , \]
and whose differential is defined by Equations (\ref{eq:ainf-alg}).
\end{definition}

\begin{definition}[Operadic bimodule \Minf]
The operadic bimodule \Minf\ is the quasi-free $(\Ainf ,\Ainf )$-operadic bimodule generated in arity $n \geq 1$ by one operation $\arbreopmorph{0.15}$ of degree $n-1$ 
\[ \Minf :=  \left( \mathcal{T}^{\Ainf , \Ainf}(\arbreopunmorph , \arbreopdeuxmorph , \arbreoptroismorph , \arbreopquatremorph , \cdots ) , \partial \right) \ , \]
and whose differential is defined by Equations (\ref{eq:ainf-morph}).
\end{definition}

We denote by $\End_A$ the \textit{endomorphism operad} of a dg module $A$, i.e. the operad whose dg module of operations of arity $n$ is $\End_A(n) := \Hom (A^{\otimes n},A)$. 
An \Ainf -algebra structure on $A$ is then equivalent to the datum of a morphism of operads $\Ainf \rightarrow \End_A$. 
We denote similarly by $\Hom^A_B$ the $(\End_B , \End_A)$-operadic bimodule defined by $ \Hom^A_B(n) := \Hom (A^{\otimes n},B)$. 
An \Ainf -morphism between two \Ainf -algebras $A$ and $B$ is then equivalent to the datum of a morphism of operadic bimodules $\Minf \rightarrow \Hom^A_B$.

Composition of \Ainf -morphisms can also be formulated at the level of the operadic bimodule \Minf\ as a morphism of $(\Ainf , \Ainf)$-operadic bimodules $\Minf \rightarrow \Minf \circ_{\Ainf} \Minf$, where the notation $\circ_{\Ainf}$ denotes the \emph{relative composite product} \cite[Section 11.2.1]{LodayVallette12}.
We write the first factor of $\Minf \circ_{\Ainf} \Minf$ using green for the color above the diaphragm and red for the color below the diaphragm,
\[ \Minf :=  \mathcal{T}^{\Ainf , \Ainf}(\arbreopunmorphcompdeux , \arbreopdeuxmorphcompdeux , \arbreoptroismorphcompdeux , \arbreopquatremorphcompdeux , \cdots ) \ , \]
and its second factor using blue for the color above the diaphragm and green for the color below the diaphragm
\[ \Minf :=  \mathcal{T}^{\Ainf , \Ainf}(\arbreopunmorphcompun , \arbreopdeuxmorphcompun , \arbreoptroismorphcompun , \arbreopquatremorphcompun , \cdots ) \ . \]

\begin{definition}[Composition morphism]
The \emph{composition morphism} is defined to be the morphism of $(\Ainf ,\Ainf )$-operadic bimodules $\comp : \Minf \rightarrow \Minf \circ_{\Ainf} \Minf$ given on the generating operations of \Minf\ by 
\[ \comp \left( \arbreopmorph{0.15}  \right) =  \sum_{i_1+\cdots+i_k=n} (-1)^{\varepsilon} \compainf \ . \]
\end{definition}

\noindent The composition of two \Ainf -morphisms $A \rightsquigarrow B$ and $B \rightsquigarrow C$ is then equivalent to the following composition of morphisms of operadic bimodules
\[ \Minf \overset{\comp}{\longrightarrow} \Minf \circ_{\Ainf} \Minf \longrightarrow \Hom^B_C \circ_{\End_B} \Hom^A_B \longrightarrow \Hom^A_C \ . \]

\subsubsection{The Forcey--Loday multiplihedra realize the operadic bimodule \Minf} \label{sss:forcey--loday-realize}

\begin{definition}[Cellular orientation] 
\leavevmode
Let $P\subset\RR^n$ be a polytope, and let $F$ be a face of $P$. A \emph{cellular orientation of $F$} is a choice of orientation of its linear span. A \emph{cellular orientation of $P$} is a choice of cellular orientation for each face $F$ of $P$. 
\end{definition}

We respectively denote by $\mathsf{CW}$ and $\mathsf{dg-mod}$ the symmetric monoidal categories of CW complexes and of dg modules over $\mathbb{Z}$, and by $C_\bullet^{\mathrm{cell}} : \mathsf{CW} \rightarrow \mathsf{dg-mod}$ the cellular chains functor. 
A choice of a cellular orientation for every polytope $P \in \mathsf{Poly}$ defines an inclusion $\mathsf{Poly} \subset \mathsf{CW}$. 
Then, the strong symmetric monoidal functor $C_\bullet^{\mathrm{cell}}$ respectively sends operads and operadic bimodules in polytopes to dg operads and dg operadic bimodules. 

\begin{definition}[Left-levelwise order] \label{def:left-levelwise-tree}
Let $t$ be a (2-colored) tree $t$. The \emph{left-levelwise order} on the vertices of $t$ is defined by ordering them from bottom to top and from left to right, proceeding one level at a time.
\end{definition}

\begin{figure}[h!]
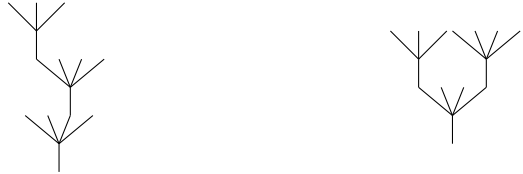


\centering

\begin{subfigure}{0.4\textwidth}
\centering
\exampleleftlevelwiseone
\end{subfigure}
\begin{subfigure}{0.4\textwidth}
\centering
\exampleleftlevelwisetwo
\end{subfigure}

\caption{The tree on the left decomposes as $(c_4\circ_3 c_4)\circ_3 c_3$ and the orientation on the face it labels is determined by the product $K_4 \times K_4 \times K_3$. 
The tree on the right decomposes as $ (c_4\circ_1 c_3)\circ_6 c_4$ and defines the orientation determined by the product $K_4 \times K_3 \times K_4$.} 
\label{fig:left-levelwise-order}
\end{figure}

Given a tree $t$, there is a unique decomposition $t=(\cdots ((c_{n_1} \circ_{i_1} c_{n_2})\circ_{i_2}c_{n_3})\cdots \circ_{i_k} c_{n_{k+1}})$ where the corollae $c_n$ are grafted according to this total order. Using the grafting operations defined in \cref{sss:grafting}, a 2-colored tree admits similarly a unique decomposition as a sequence of blue corollae, red corollae and 2-colored corollae ordered according to this total order. 
We can then make the same choices of cellular orientations as in \cite[Section 1.4]{mazuir-I}, illustrated in \cref{fig:left-levelwise-order} :
\begin{itemize}
  \item For the Loday associahedra $\K_n \subset \RR^{n-1}$ of \cite{MTTV19}, we choose the basis $\{e_1 - e_{j+1}\}_{1\leq j \leq n-2}$ as positively oriented basis of the top dimensional cell $\arbreop{0.15}$. We then choose the orientation of any other face $t$ of $\K_n$ to be the image of the positively oriented bases of the top cells of the polytopes $\K_{n_i}$ under the sequence of partial compositions following the left-levelwise order on $t$. 
  \item We choose the basis $\{- e_j\}_{1\leq j \leq n-1}$ as positively oriented basis of the top dimensional cell $\arbreopmorph{0.15}$ of the Forcey--Loday multiplihedra $\J_n \subset \RR^{n-1}$. We then choose the orientation of any other face $t$ of $\J_n$ to be the image of the positively oriented bases of the top cells of the polytopes $\K_{n_i}$ and $\J_{n_j}$ under the sequence of action-compositions maps, following the left-levelwise order on $t$.
\end{itemize}

\begin{proposition} 
\label{prop:cellular-chains}
These cellular orientations on the Loday associahedra and the Forcey--Loday multiplihedra provide an isomorphism of dg operads $C_\bullet^{\mathrm{cell}}(\{\K_n\})\cong \mathrm{A}_\infty$ and an isomorphism of dg operadic bimodules $C_\bullet^{\mathrm{cell}}(\{\J_n\})\cong \Minf$. 
\end{proposition}

\begin{proof}
The choice of a cellular orientation endows the $\K_n$ and $\J_n$ with a natural CW structure (see \cite[Proposition 4.22]{LA21}).
The choice of the left-levelwise order on trees ensures that we recover precisely the usual sign conventions for the partial compositions of the quasi-free operad $\Ainf$ and for the action-composition maps of the quasi-free operadic bimodule $\Minf$. 
The signs for the respecive differentials were computed in \cite[Section 1.4]{mazuir-I}.
\end{proof}

\subsection{Tensor product of \Ainf -algebras and \Ainf -morphisms}

\subsubsection{Diagonals on the operad \Ainf\ and on the operadic bimodule \Minf}

\begin{definition}[Operadic diagonals] $ $
\begin{enumerate}[leftmargin=*]
    \item A \emph{diagonal on the operad \Ainf} is a morphism of dg operads $\triangle : \Ainf \rightarrow \Ainf \otimes \Ainf$ which satisfies $\triangle (\arbreopdeux) = \arbreopdeux \otimes \arbreopdeux$. 
    \item Given a diagonal on the operad \Ainf, a \emph{diagonal on the operadic bimodule \Minf} is a morphism of dg-operadic bimodules $\triangle : \Minf \rightarrow \Minf \otimes \Minf$ which satisfies $\triangle ( \arbreopunmorph ) = \arbreopunmorph \otimes \arbreopunmorph$, and where $\Minf \otimes \Minf$ is endowed with its $(\Ainf , \Ainf)$-operadic bimodule structure induced by the diagonal on \Ainf .
\end{enumerate}
\end{definition}

Diagonals provide an adapted framework to define tensor products of \Ainf -algebras and tensor products of \Ainf -morphisms.
Given a diagonal $\Ainf \to \Ainf \otimes \Ainf$ and two \Ainf -algebras $A$ and $B$, one can define an \Ainf -algebra structure on $A \otimes B$ by considering the following composition
\[ \Ainf \longrightarrow \Ainf \otimes \Ainf \longrightarrow \End_A \otimes \End_B \longrightarrow \End_{A \otimes B} \ . \]
Given similarly a diagonal $\Minf \to \Minf \otimes \Minf$ and two \Ainf -morphisms $F_1 : A_1 \rightsquigarrow B_1$ and $F_2 : A_2 \rightsquigarrow B_2$, one can define an \Ainf -morphism $F_1 \otimes F_2 : A_1 \otimes A_2 \rightsquigarrow B_1 \otimes B_2$ by the following composition
\[ \Minf \rightarrow \Minf \otimes \Minf \rightarrow \Hom^{A_1}_{B_1} \otimes \Hom^{A_2}_{B_2} \rightarrow  \Hom^{A_1 \otimes A_2}_{B_1 \otimes B_2} \ . \]
We moreover point out that the conditions $\triangle (\arbreopdeux) = \arbreopdeux \otimes \arbreopdeux$ and $\triangle ( \arbreopunmorph ) = \arbreopunmorph \otimes \arbreopunmorph$ respectively imply that these constructions recover the standard tensor product of dg algebras and the standard tensor product of ordinary morphisms between dg algebras.

\subsubsection{Admissible edges and permutations}

We fix a (2-colored) nested linear graph $(\gra,\mathcal{N})$.
We denote by $N_i$ the unique minimal nest of $\mathcal{N}$ with respect to nest inclusion, which contains the edge $i$.  

\begin{definition}[Admissible edge]
For a nested linear graph $(\gra,\mathcal{N})$, an edge $i$ is \emph{admissible} with respect to $\mathcal{N}$ if $i \neq \min N_i$. 
For a 2-colored nested linear graph $(\gra,\mathcal{N})$, an edge $i$ is \emph{admissible} with respect to $\mathcal{N}$ when $N_i$ is bicolored, or if $i \neq \min N_i$ when $N_i$ is monochrome.
We denote the set of admissible edges of $\mathcal{N}$ by $\mathrm{Ad}(\mathcal{N})$. 
\end{definition}

\begin{definition} [Left-levelwise order]
\label{def:left-levelwise-graph}
The \emph{left-levelwise order} on $\mathcal{N}$ is defined by ordering the nests by decreasing order of cardinality, and ordering two nests of the same cardinality according to the increasing order on their minimal elements. 
\end{definition}

\noindent  Under the bijection of \cref{lemma:bijection}, the left-levelwise order on the nesting of a nested linear graph is equivalent to the left-levelwise order on the vertices of the corresponding tree $t$, as defined in \cref{def:left-levelwise-tree} .

Consider the left-levelwise order $N^1<N^2<\cdots < N^k$ on the nesting $\mathcal{N}=\{N^j\}_{1\leq j \leq k}$.
We endow the set $\mathrm{Ad}(\mathcal{N})$ with a total order, by ordering the admissible edges of $N_1 \setminus \cup_{2\leq j \leq k} N_j$ in increasing order, then the admissible edges of $N_2 \setminus \cup_{3\leq j \leq k} N_j$ in increasing order, and so on. 
Given two nestings $\mathcal{N}, \mathcal{N}'$ of $\gra$, we endow the set $\mathrm{Ad}(\mathcal{N})\sqcup \mathrm{Ad}(\mathcal{N}')$ with the total order given by following the total order on $\mathrm{Ad}(\mathcal{N})$ and then the total order on $\mathrm{Ad}(\mathcal{N}')$. 
We denote by $\triangle^K$ and $\triangle^J$ the algebraic diagonals obtained from the polytopal ones by applying the cellular chains functor, see \cref{prop:diagonal-polytopale-a-infini,prop:diagonal-polytopale-m-infini} below. 
The proofs of these two propositions include the proofs of the following two lemmas. 

\begin{lemma} 
\label{prop:signs-ass}
For a pair of nestings of complementary dimensions $(\mathcal{N}, \mathcal{N}')\in \Ima\triangle^K$, the function $\sigma_{\mathcal{N}\mathcal{N}'}: \mathrm{Ad}(\mathcal{N})\sqcup \mathrm{Ad}(\mathcal{N}') \to (1,2,\ldots,|\mathrm{Ad}(\mathcal{N})\sqcup \mathrm{Ad}(\mathcal{N}')|)$ defined on $i \in \mathrm{Ad}(\mathcal{N})$ by 
\begin{equation*}
  \sigma_{\mathcal{N}\mathcal{N}'}(i)= 
  \begin{cases}
    \min N_i -1 & \text{ if } i \in \mathrm{Ad}(\mathcal{N})\cap \mathrm{Ad}(\mathcal{N}') \text{ and } 1 \neq \min N_i < \min N_i' \\ 
    i-1 & \text{ otherwise ,} 
  \end{cases}
\end{equation*}
and similarly on $i \in \mathrm{Ad}(\mathcal{N}')$ by reversing the roles of $\mathcal{N}$ and $\mathcal{N}'$, induces a permutation of the set $$\{1,2,\ldots,|\mathrm{Ad}(\mathcal{N})\sqcup \mathrm{Ad}(\mathcal{N}')|\}$$ that we will still denote by $\sigma_{\mathcal{N}\mathcal{N}'}$.
\end{lemma}

\begin{lemma} 
\label{prop:signs-mul}
For a pair of 2-colored nestings of complementary dimensions $(\mathcal{N},\mathcal{N}')\in \Ima\triangle^J$, the function $\sigma_{\mathcal{N}\mathcal{N}'}: \mathrm{Ad}(\mathcal{N})\sqcup \mathrm{Ad}(\mathcal{N}') \to (1,2,\ldots,|\mathrm{Ad}(\mathcal{N})\sqcup \mathrm{Ad}(\mathcal{N}')|)$ defined on $i \in \mathrm{Ad}(\mathcal{N})$ by
\begin{equation*}
  \sigma_{\mathcal{N}\mathcal{N}'}(i)= 
  \begin{cases}
    \min N_i & \text{ if } i \in \mathrm{Ad}(\mathcal{N})\cap \mathrm{Ad}(\mathcal{N}') , N_i \text{ is monochrome and } N_i' \text{ is not} \\
    \min N_i & \begin{array}{l}
         \text{ if } i \in \mathrm{Ad}(\mathcal{N})\cap \mathrm{Ad}(\mathcal{N}'), N_i \text{ and } N_i' \text{ are monochrome} \\
         \text{ and } \min N_i < \min N_i' \ , 
    \end{array}   \\ 
    i & \text{ otherwise ,} 
  \end{cases}
\end{equation*}
and similarly on $i \in \mathrm{Ad}(\mathcal{N}')$ by reversing the roles of $\mathcal{N}$ and $\mathcal{N}'$, induces a permutation of the set $$\{1,2,\ldots,|\mathrm{Ad}(\mathcal{N})\sqcup \mathrm{Ad}(\mathcal{N}')|\}$$ that we will still denote by $\sigma_{\mathcal{N}\mathcal{N}'}$.
\end{lemma}

\subsubsection{The polytopal diagonals on \Ainf\ and \Minf}
\label{ss:diagonals}

We use nested linear graphs introduced in \cref{ss:2-col} to work with the operad \Ainf\ and the operadic bimodule \Minf . The generating operation of arity $n$ of \Ainf\ corresponds to the trivial nested linear graph with $n$ vertices $\black{ \bullet \cdots \bullet }$, while the generating operation of arity $n$ of \Minf\ is represented by the trivial 2-colored nested linear graph with $n$ vertices $\purple{\bullet \cdots \bullet}$.

\begin{proposition}
\label{prop:diagonal-polytopale-a-infini}
The image under the functor $C_\bullet^{\mathrm{cell}}$ of the diagonal of the Loday associahedra constructed in \cite{MTTV19} defines a diagonal on the operad \Ainf , that we denote \diagainf . It is determined by the formula 
\[ \diagainf \left( \black{ \bullet \cdots \bullet } \right) = 
\sum_{\substack{
  \mathcal{N},\mathcal{N}' \in \mathcal{N}_n \\ 
  \tp(\mathcal{N}) \leq \bm(\mathcal{N'}) \\
  |\mathcal{N}|+|\mathcal{N}'|=n
}}
(-1)^{|\mathrm{Ad}(\mathcal{N})\cap \mathrm{Ad}(\mathcal{N}')|}\mathrm{sgn}(\sigma_{\mathcal{N}\mathcal{N}'})\mathcal{N} \otimes \mathcal{N}' \ , \] 
where $\bullet \cdots \bullet$ stands for the linear graph with $n$ vertices. 
\end{proposition}

\begin{proof}
The image of the diagonal on the Loday associahedra under the functor $C_\bullet^{\mathrm{cell}}$ defines a diagonal on the operad \Ainf\ as this functor is strong monoidal. 
This diagonal $\diagainf : \Ainf \rightarrow \Ainf \otimes \Ainf$ is determined by the image of the generating operations of the quasi-free operad \Ainf , which are the trivially nested linear graphs. 
The signs arise from the choices of cellular orientations on the Loday associahedra made in \cref{sss:forcey--loday-realize} as follows. 
As explained in the proof of \cite[Proposition 4.27]{LA21}, the computation of the signs boils down to the computation of the determinant of the bases $e_j^{F}, e_j^{G}$ determining the cellular orientations of the faces $F$ and $G$ associated to the nestings $\mathcal{N}$ and $\mathcal{N}'$, expressed in the basis $e_j$ of the top dimensional cell of $\K_n$. 
The second part of the proof of \cite[Theorem 1.26]{LA21} shows that  $\dim(F\cap \rho_z G)=0$, for any $z \in (\mathring F+ \mathring G)/2$.
Combined with the fact that $\dim F + \dim G = \dim \K_n$, this implies that the two bases $e_j^F, e_j^G$ form together a basis of the linear span of $\K_n$. 
Writing horizontally the $e_j^F$ and then the $e_j^G$ in the basis $e_j$ defines a square matrix.
The positions of the rightmost non-zero entries of each line are given by the admissible edges of $\mathcal{N}$ and $\mathcal{N}'$. 
The permutation $\sigma_{\mathcal{N}\mathcal{N}'}$ corresponds to a permutation of the lines of this matrix, sending these righmost entries to the diagonal, except for one case: when $\mathcal{N}$ and $\mathcal{N}'$ share the same admissible edge. 
In this case, linear independence guarantees that the two vectors differ in another place.
We moreover point out that that the $-1$ term in the definition of the permutation $\sigma_{\mathcal{N}\mathcal{N}'}$ in \cref{prop:signs-ass} stems from the fact that $\K_n$ is defined in $\RR^{n-1}$ but has dimension $n-2$. 
\end{proof}

We compute in particular 

\[ \resizebox{\hsize}{!}{$\displaystyle{
\begin{matrix}
\diagainf ( \black { \bullet \bullet } )
&=& & \black{\bullet \bullet} \otimes \black{\bullet \bullet} \ , & & \\
\diagainf ( \black { \bullet \bullet \bullet } ) 
&=& & \black{\black{\bullet \bullet} \bullet} \otimes \black{\bullet \bullet \bullet} &+& \black{\bullet \bullet \bullet} \otimes \black{ \bullet \black{\bullet \bullet}} \ , \\
\diagainf ( \black{ \bullet \bullet \bullet \bullet} ) 
&=& & 
\black{\bullet \bullet \bullet \bullet} \otimes \black{\bullet \black{ \bullet \black{ \bullet \bullet }} }  
&+& \black{ \black{ \black{\bullet \bullet} \bullet } \bullet }  \otimes \black{ \bullet \bullet \bullet \bullet } \\
& & -& \black{ \black{\bullet \bullet } \bullet \bullet } \otimes \black{ \bullet \bullet \black{ \bullet \bullet }}  &+& \black{ \black{ \bullet \bullet \bullet } \bullet }  \otimes \black{ \bullet \black{ \bullet \bullet } \bullet }  \\ 
& & +& \black{ \black{ \bullet \bullet \bullet } \bullet }  \otimes \black{ \bullet \black{ \bullet \bullet \bullet }} &+& \black{ \bullet \black{ \bullet \bullet } \bullet } \otimes \black{ \bullet \black{ \bullet \bullet \bullet }} \ .
\end{matrix}}$} \]

\begin{remark}
\cref{prop:diagonal-polytopale-a-infini} completes the work of \cite{MTTV19}, by explicitly computing the signs for the polytopal diagonal on the dg level. This formula corresponds in fact to the formula originally computed in \cite{MarklShnider06} (up to signs verification). We also conjecture that this diagonal is equal to the diagonal constructed in \cite{SaneblidzeUmble04}.
\end{remark}

\begin{definition}[Tensor product of \Ainf -algebras] \label{def:tensor-product-ainf-alg}
Given $A$ and $B$ two \Ainf -algebras, their tensor product as \Ainf -algebras is defined to be the dg module $A \otimes B$ endowed with the \Ainf -algebra structure induced by the diagonal \diagainf .
\end{definition}

\begin{proposition}
\label{prop:diagonal-polytopale-m-infini}
The image under the functor $C_\bullet^{\mathrm{cell}}$ of the diagonal on the Forcey--Loday multiplihedra constructed in this paper defines a diagonal on the operadic bimodule \Minf , that we denote \diagminf . It is determined by the formula
\[ \diagminf \left( \purple{\bullet \cdots \bullet} \right) = 
\sum_{
\mathcal{N},\mathcal{N}'} 
(-1)^{|\mathrm{Ad}(\mathcal{N})\cap \mathrm{Ad}(\mathcal{N}')|}
\mathrm{sgn}(\sigma_{\mathcal{N}\mathcal{N}'})
\mathcal{N} \otimes \mathcal{N}' \ ,\]
where the sum runs over the pairs $\mathcal{N},\mathcal{N}' \in \mathcal{N}^2_n$ such that $\ |\mono(\mathcal{N})|+|\mono(\mathcal{N}')|=n-1$ and which satisfy the conditions in \cref{thm:formuladiagonal}.
\end{proposition}

\begin{proof}
The proof is similar to the proof of \cref{prop:diagonal-polytopale-a-infini}. 
Note that in this case, there is no $-1$ term in the definition of the permutation $\sigma_{\mathcal{N}\mathcal{N}'}$ in \cref{prop:signs-mul} since $\J_n$ is full-dimensional. 
\end{proof}

We compute in particular 
\[
\begin{matrix}
\diagminf ( \purple{ \bullet } )  
&=&  & \purple{ \bullet } \otimes \purple{ \bullet } \ , & &  \\
\diagminf ( \purple{ \bullet \bullet } ) 
&=& & \blue{\bullet \bullet} \otimes \purple{\bullet \bullet} & + & \purple{\bullet \bullet} \otimes \red{\bullet \bullet} \ , \\
\diagminf (\purple{\bullet \bullet \bullet}) 
&=& 
& \blue{\blue{\bullet \bullet}\bullet} \otimes \purple{\bullet \bullet \bullet} 
& + & \purple{\bullet \bullet \bullet} \otimes \red{\bullet \red{\bullet \bullet}} \\
& &- &  \blue{\bullet \bullet \bullet} \otimes \purple{\bullet \blue{\bullet \bullet}}    &  - & \blue{\bullet \bullet \bullet} \otimes \red{\bullet \purple{\bullet \bullet}}  \\  &   &+ & \purple{\bullet \blue{\bullet \bullet}} \otimes \red{\bullet \purple{\bullet \bullet}} 
  &  - & \purple{\blue{\bullet \bullet} \bullet} \otimes \red{\purple{\bullet \bullet} \bullet} \\ & &  + & \purple{\blue{\bullet \bullet} \bullet} \otimes \red{\bullet \bullet \bullet}   &  + & \red{\purple{\bullet \bullet} \bullet} \otimes \red{\bullet \bullet \bullet} \ .
  \end{matrix}
\]

\begin{definition}[Tensor product of \Ainf -morphisms] \label{def:tensor-ainf-morph}
Let $F_1 : A_1 \rightsquigarrow B_1$ and $F_2 : A_2 \rightsquigarrow B_2$ be two \Ainf -morphisms between \Ainf-algebras.
Their tensor product is defined to be the \Ainf -morphism $F_1 \otimes F_2 : A_1 \otimes A_2 \rightsquigarrow B_1 \otimes B_2$ induced by the diagonal \diagminf on \Minf \ .
\end{definition}

One can ask whether the dg "magical formula" for the diagonal on the operad \Ainf\
also defines a diagonal on the operadic bimodule \Minf, i.e. if by relaxing the conditions of \cref{thm:formuladiagonal} to the condition $\tp(\mathcal{N}) \leq \bm(\mathcal{N'})$, the formula of \cref{prop:diagonal-polytopale-m-infini} still defines a diagonal on \Minf \ . A simple computation in arity 4 shows that the answer to this question is negative, see \cref{prop:pas-top-bot}. In other words, it is not possible to naively extend the "magical formula" for the tensor product of \Ainf -algebras to define a tensor product of \Ainf -morphisms.

\subsection{Categorification}

\subsubsection{Tensor product of \Ainf -categories and \Ainf -functors}

As categories can be viewed as associative algebras with several objects, \Ainf -categories can be viewed in a similar fashion as \Ainf -algebras with several objects. One says that the notions of \Ainf -category and \Ainf -functor are the horizontal categorifications of the notions of \Ainf -algebra and \Ainf -morphism, respectively. 
We refer to \cite[Chapter 1]{Seidel08} for the definitions of these two notions. 
We borrow the notations from \cite{Seidel08} and will moreover use the sign conventions of \cref{ss:ainf-alg-ainf-morph}. 

\begin{definition}[Tensor product of \Ainf -categories] \label{def:tensor-product-ainf-cat}
The \emph{tensor product} of two $\mathrm{A}_\infty$-categories $\cat{A}$ and $\cat{B}$ is given by 
\begin{itemize}[leftmargin=*]
  \item the set of objects $\mathrm{Ob}(\cat{A}\otimes \cat{B})\coloneqq \mathrm{Ob}(\cat{A})\times\mathrm{Ob}(\cat{B})$,
  \item for each pair of objects $X_1\times Y_1,X_2\times Y_2 \in \mathrm{Ob}(\cat{A}\otimes \cat{B})$, the dg module of morphisms \[\cat{A}\otimes \cat{B}(X_1\times Y_1,X_2\times Y_2)\coloneqq \cat{A}(X_1,X_2)\otimes\cat{B}(Y_1,Y_2) \ , \]
\end{itemize}
and by defining the higher compositions $m_n$ as in \cref{prop:diagonal-polytopale-a-infini}.
\end{definition}

\begin{samepage}
\begin{definition}[Tensor product of \Ainf -functors]
The \emph{tensor product} of two $\mathrm{A}_\infty$-functors $\cat{F}:\cat{A}_1 \rightsquigarrow \cat{B}_1$ and $\cat{G}:\cat{A}_2 \rightsquigarrow \cat{B}_2$ is given by the function 
\[ \mathrm{Ob}(\cat{F}\otimes \cat{G})\coloneqq \mathrm{Ob}(\cat{F})\times \mathrm{Ob}(\cat{G}) : \mathrm{Ob}(\cat{A}_1\otimes\cat{B}_1) \to \mathrm{Ob}(\cat{A}_2\otimes\cat{B}_2) \ , \]
and by defining the operations $(\cat{F} \otimes \cat{G})_n$ as in \cref{prop:diagonal-polytopale-m-infini}.
\end{definition}
\end{samepage}

\subsubsection{Identities}

The category $H_*(\cat{A})$ associated to an \Ainf -category $\cat{A}$ does not necessarily have identity morphisms. As explained in \cite[Section 1.2]{Seidel08}, there exist three notions of \Ainf -category with identity morphisms : \textit{strictly unital \Ainf -category}, \textit{cohomologically unital \Ainf -category} and \textit{homotopy unital \Ainf -category}. 
\begin{enumerate}[leftmargin=*]
\item A \textit{cohomologically unital} \Ainf -category is an \Ainf -category $\cat{A}$ which is such that $H_*(\cat{A})$ has identity morphisms.
\item A \textit{strictly unital} \Ainf -category is an \Ainf -category together with an element $e_X \in \cat{A} (X ,X )$ for every $X \in \mathrm{Ob}(\cat{A})$ such that $\partial (e_X) = 0$, $m_2 (e , \cdot ) = m_2 (\cdot , e ) = \ide$ and $m_n ( \cdots , e , \cdots ) = 0 \text { for } n \geq 3$.
\item A \textit{homotopy unital} \Ainf -category is defined to be an \Ainf -category together with elements $e_X \in \cat{A} (X ,X )$ and endowed with additional operations encoding the fact that the previous relations on the $m_n$ and the $e_X$ are satisfied only up to higher coherent homotopies, see also \cite[Section 6.1]{HirshMilles12}.
\end{enumerate}
We have in particular that
\[  \text{unital} \Rightarrow \text{homotopy unital} \Rightarrow \text{cohomologically unital} \ . \]
The proof of the following proposition is straightforward.

\begin{proposition} $ $
\begin{enumerate}[leftmargin=*]
\item If $\cat{A}$ and $\cat{B}$ are cohomologically unital \Ainf -categories, the tensor \Ainf -category $\cat{A} \otimes \cat{B}$ is again cohomologically unital. 
\item If $\cat{A}$ and $\cat{B}$ are strictly unital \Ainf -categories, the tensor \Ainf -category $\cat{A} \otimes \cat{B}$ is again strictly unital, with identity morphisms $e_{X \times Y} := e_X \otimes e_Y$ for  $X \in \mathrm{Ob}(\cat{A})$ and $Y \in \mathrm{Ob}(\cat{B})$.
\end{enumerate}
\end{proposition}

If $\cat{A}$ and $\cat{B}$ are homotopy unital \Ainf -categories, we have to define the additional operations associated to the fact that the elements $e_X \otimes e_Y$ are identity morphisms up to homotopy in order to endow the \Ainf -category $\cat{A} \otimes \cat{B}$ with a homotopy unital \Ainf -category structure. In other words, we have to define a diagonal on the operad \uAinf\ encoding homotopy unital \Ainf -algebras, which has not been done yet to the authors knowledge. An idea would be to define a diagonal on the unital associahedra, which are CW-complexes constructed by Muro and Tonks in \cite{MuroTonks} and which form an operad whose image under the cellular chains is the operad \uAinf\ . However, not all unital associahedra are polytopes, meaning that the present techniques cannot be directly applied to them. 

\subsection{Homotopy properties of diagonals on \Ainf\ and \Minf } \label{ss:homotopy-properties}
The goal of this section is to show that the category of \Ainf -algebras is symmetric monoidal only up to homotopy.

\subsubsection{The 2-colored viewpoint}

The operad \Ainf\ together with the operadic bimodule \Minf\ define the quasi-free 2-colored operad 
\[ A_\infty^2 := \left( \mathcal{T} (\arbreopdeuxcol{Red!60} , \arbreoptroiscol{Red!60} , \arbreopquatrecol{Red!60}, \cdots, \arbreopdeuxcol{MidnightBlue} , \arbreoptroiscol{MidnightBlue} , \arbreopquatrecol{MidnightBlue} , \cdots, \arbreopunmorph , \arbreopdeuxmorph , \arbreoptroismorph , \arbreopquatremorph , \cdots )  , \partial \right) \ , \]
whose differential is given by the equations of \cref{def:ainf-alg} and \cref{def:ainf-morph}.
We refer to \cite[Section 11]{yau-colored} for a complete definition of a 2-colored operad.
The data of \Ainf -algebra structures on two dg modules $A$ and $B$ together with an \Ainf -morphism $A \rightsquigarrow B$ between them is equivalent to a morphism of 2-colored operads $\Ainfdeux \longrightarrow \End ( A \text{\hspace{2pt}} ; B) $, where $\End ( A ; B)$ is the \textit{endomorphism 2-colored operad} naturally associated to $A$ and $B$.
The data of a diagonal on the operad \Ainf\ and of a diagonal on the operadic bimodule \Minf\ is moreover equivalent to the datum of a morphism of 2-colored operads $\Ainfdeux \longrightarrow \Ainfdeux \otimes \Ainfdeux$, while the composition of \Ainf -morphisms can be defined by a morphism of 2-colored operads $\Ainfdeux \longrightarrow \Ainfdeux \circ_{\Ainf} \Ainfdeux $.

\subsubsection{Coassociativity and cocommutativity} \label{sss:coassoc-cocomm}

First, we would like to know whether given three \Ainf -algebras $A$, $B$ and $C$, the two \Ainf -algebra structures $( A \otimes B) \otimes C$ and $A \otimes ( B \otimes C)$ on the dg module $A \otimes B \otimes C$ are the same. 
In operadic terms, this amounts to ask if the diagonal on \Ainf\ is coassociative.

\begin{proposition} $ $
  \label{prop:nocoassoc}
  \begin{enumerate}[leftmargin=*]
  \item There is no diagonal on the operad \Ainf\ which is coassociative. 
  \item There is no diagonal on the operadic bimodule \Minf\ which is coassociative.
  \end{enumerate}
\end{proposition} 

\begin{proof}
The non-existence of a coassociative diagonal on the operad \Ainf\ was already proven in \cite[Section 6]{MarklShnider06}.
The non-existence of a coassociative diagonal on the operad \Ainf\ implies the non-existence of a coassociative diagonal on the operadic bimodule \Minf . 
Given indeed diagonals $\triangle^{\Ainf}$ and $\triangle^{\Minf}$, it is not possible to compare the two morphisms of dg operadic bimodules $ ( \triangle^{\Minf} \otimes \ide^{\Minf} ) \triangle^{\Minf}$ and $(\ide^{\Minf} \otimes \triangle^{\Minf} )\triangle^{\Minf}$, as the $(\Ainf , \Ainf)$-operadic bimodule structures induced on $\Minf^{\otimes 3}$ by $ ( \triangle^{\Ainf} \otimes \ide^{\Ainf} ) \triangle^{\Ainf}$ and $(\ide^{\Ainf} \otimes \triangle^{\Ainf} ) \triangle^{\Ainf}$ do not coincide.
We can in fact prove a stronger result:  for any diagonal $\triangle : \Minf \to \Minf \otimes \Minf$, we have that
\[ \left( (\ide \otimes \triangle ) \triangle - (\triangle \otimes \ide) \triangle \right) \left( \purple{\bullet \bullet \bullet} \right) \neq 0 \ . \]
The proof of this result involves computations identical to the ones of \cite[Section 6]{MarklShnider06}, that we do not include for the sake of concision. 
\end{proof}

This proposition implies in particular that a diagonal on the 2-colored operad \Ainfdeux\ is never coassociative.
In the specific cases of \diagainf\ and \diagminf\ we compute moreover that
  \begin{align*}
      &\left( (\ide \otimes \diagainf ) \diagainf - (\diagainf \otimes \ide) \diagainf \right) \left( \black{ \bullet \bullet \bullet \bullet } \right)  \\
      = \  & - \partial \left( \black{ \black{ \bullet \bullet \bullet } \bullet }  \otimes \black{ \bullet \black{ \bullet \bullet } \bullet } \otimes  \black{ \bullet \black{ \bullet \bullet \bullet } } \right) \ ,
  \end{align*}
and that 
\begin{align*}
      &\left( (\ide \otimes \diagminf ) \diagminf - (\diagminf \otimes \ide) \diagminf \right) \left( \purple{\bullet \bullet \bullet} \right)  \\
      = \ &\partial \left( \blue{\bullet \bullet \bullet} \otimes \purple{ \bullet \blue{\bullet \bullet}} \otimes \red{\bullet \purple{\bullet \bullet}}
    - \purple{\blue{\bullet \bullet} \bullet} \otimes \red{\purple{\bullet \bullet} \bullet} \otimes \red{\bullet \bullet \bullet}
  \right) \ .
\end{align*} 
  
Given two \Ainf -algebras $A$ and $B$, we would also like to know whether the \Ainf -algebra structure on $B \otimes A$ can simply be obtained from the maps defining the \Ainf -algebra structure on $A \otimes B$ \[ m_n^{A \otimes B} : ( A \otimes B)^{\otimes n} \rightarrow A \otimes B \] by rearranging $(A \otimes B)^{\otimes n}$ into $(B \otimes A)^{\otimes n}$ and $A \otimes B$ into $B \otimes A$. 
In operadic terms, this amounts to ask if the diagonal on \Ainf\ is cocommutative or not. 

\begin{proposition} \label{prop:not-cocomm}
The diagonals \diagainf\ and \diagminf\ are not cocommutative.
\end{proposition} 

\begin{proof}
We compute indeed that
\[ \left( \diagainf - \tau \diagainf \right) \left( \black{ \bullet \bullet \bullet } \right) = \partial \left( \black{ \bullet \bullet \bullet } \otimes \black{ \bullet \bullet \bullet } \right) \ , \]
where $\tau$ acts by the permutation $(1 \ 2)$ on the operad $\Ainf \otimes \Ainf$.
We also compute that
\[ \left( \diagminf - \tau \diagminf \right) \left( \purple{\bullet \bullet} \right) = \partial \left( \purple{\bullet \bullet} \otimes \purple{\bullet \bullet} \right) \ . \]
\end{proof}

\noindent We conjecture in fact that \cref{prop:not-cocomm} holds for any diagonal on the operad \Ainf\ and for any diagonal on the operadic bimodule \Minf .
 
\subsubsection{Compatibility with the composition} \label{sss:comp-composition}

We would finally like to know whether the tensor product is functorial with respect to the composition of \Ainf -morphisms. 
In other words, if given four \Ainf -morphisms $F_1 : A_1 \rightsquigarrow B_1$, 
$G_1 : B_1 \rightsquigarrow C_1$, $F_2 : A_2 \rightsquigarrow B_2$ and
$G_2 : B_2 \rightsquigarrow C_2$ they satisfy the following equality
\[ ( G_1 \otimes F_1) \circ (G_2 \otimes F_2) = (G_1 \otimes G_2) \circ (F_1 \otimes F_2) \ . \]
In operadic terms, this amounts to ask if the diagonal $\triangle$ on \Minf\ together with the composition morphism \comp\ of \cref{sss:operad-ainf-operadic-bimod-minf} satisfy the following equality 
\[ ( \comp \otimes \comp ) \triangle = (\triangle \circ_{\Ainf} \triangle ) \comp   \ . \]

\begin{proposition} 
  \label{thm:nofunctorial}
  There is no diagonal on the operadic bimodule \Minf\ which is compatible with the composition of \Ainf -morphisms.  
\end{proposition}

\begin{proof} 
  Let $\triangle$ be a diagonal $\Minf \rightarrow \Minf \otimes \Minf$.
  The compatibility with the differential implies that $\triangle$ is necessarily of the form
  \[
      \triangle(\purple{\bullet }) = \purple{\bullet } \otimes \purple{\bullet } \]
and
\[ \begin{matrix}
      \triangle (\purple{\bullet \bullet }) &= &\alpha (\blue{\bullet \bullet }\otimes \purple{\bullet \bullet } + \purple{\bullet \bullet }\otimes\red{\bullet \bullet }) \\ & &+ \ (1-\alpha)(\red{\bullet \bullet }\otimes \purple{\bullet \bullet }+\purple{\bullet \bullet}\otimes \blue{\bullet \bullet}) \ ,
  \end{matrix} \]
  where $\alpha \in \mathbb{Z}$.
  We compute that if the equality
  \[ ( \comp \otimes \comp ) \triangle ( \purple{\bullet \bullet} ) = (\triangle \circ_{\Ainf} \triangle ) \comp  ( \purple{\bullet \bullet} ) \]
  holds, we necessarily have that $\alpha = 0$ and that $\alpha =1$, which is not possible.
\end{proof}

\noindent In the case of the diagonals \diagainf\ and \diagminf , we compute that 
  \[ \left( \comp  \circ \diagminf - (\diagminf \circ_{\Ainf} \diagminf ) \circ \comp \right) \left( \arbreopdeuxmorph  \right) = \partial \left( \arbreopcompun \otimes \arbreopcompdeux \right)  \ . \]

\subsubsection{Homotopy properties}

While coassociativity, compatibility with the composition and cocommutativity are not satisfied by the diagonals \diagainf\ and \diagminf , we will now prove that a diagonal on the 2-colored operad \Ainfdeux\ always satisfies these properties up to homotopy.
We use the notion of homotopy between morphisms of 2-colored operads as defined in \cite[Section 3.10]{MSS}. 

\begin{proposition} 
\label{th:homotopy-properties}
Let $\triangle$ be a diagonal on the 2-colored operad \Ainfdeux . 
\begin{enumerate}
    \item The morphisms of operads $(\triangle \otimes \ide ) \triangle)$ and $(\ide \otimes \triangle) \triangle)$ are homotopic. In other words, a diagonal on \Ainfdeux\ is always coassociative up to homotopy.
    \item The morphisms of operads $\triangle$ and $\tau \triangle$ are homotopic. In other words, a diagonal on \Ainfdeux\ is always cocommutative up to homotopy.
    \item The morphisms of operads $\comp  \circ \diagminf$ and $( \diagminf \circ_{\Ainf} \diagminf ) \circ \comp$ are homotopic. In other words, a diagonal on \Ainfdeux\ is always compatible with the composition of \Ainf -morphisms up to homotopy.
\end{enumerate}
\end{proposition}

\begin{proof}
The proof of this proposition is a simple adaptation of the results of \cite[Section 2]{MarklShnider06} in the context of 2-colored dg operads, applied to the minimal model \Ainfdeux\ for the 2-colored dg operad $As^2$ encoding pairs of dg algebras together with morphisms between them.
\end{proof}

While \cref{thm:nofunctorial} shows that it is not possible to endow the category $\infAalg$ with a symmetric monoidal category structure using the viewpoint of diagonals, \cref{th:homotopy-properties} exhibits a first level of homotopies that could be involved in the definition of some kind of \textit{homotopy symmetric monoidal} category structure on \infAalg . 
This question will be studied in a future work by D. Poliakova and the two authors of this paper. As a first step towards solving that problem, we will inspect in particular which higher coherent homotopies arise from the lack of coassociativity of $\triangle^{K_n}$ and $\triangle^{J_n}$ on the level of polytopes. 


\section{Further applications} \label{sec:V}

We first prove that a diagonal on the dg operad \Ainf\ is equivalent to a retraction of the bar-cobar resolution $\AAinf$ onto the operad \Ainf\ . 
We then explain how to associate a convolution \Ainf -algebra to an \Ainf -coalgebra and an \Ainf -algebra, as well as \Ainf -morphisms between convolution \Ainf -algebras, using diagonals on \Ainf\ and \Minf .
We finally describe two possible applications of our results in symplectic topology: in the context of Heegard Floer homology, and to study tensor products of Fukaya categories/algebras and \Ainf -functors between them.

\subsection{Retractions and diagonals} \label{ss:retract-diag}

Recall that the operad \Ainf\ is the minimal model $\Ainf =\Omega As^{\text{!`}}$ of the dg operad $As$ encoding associative algebras. 
Another cofibrant replacement of the operad $As$ is given by the bar-cobar (or Boardman-Vogt) resolution $\AAinf := \Omega B As$, which is defined as the quasi-free operad 
\[ \AAinf := \left( \mathcal{T} (\premiertermecobarbarA , \premiertermecobarbarD , \premiertermecobarbarB , \premiertermecobarbarC , \cdots , \mathrm{PT_n} \text{\hspace{2pt}}, \cdots ) , \partial \right) \ , \] 
where $\mathrm{PT_n}$ is the set of planar rooted trees of arity $n$ and the degree of a tree is defined as the number of its internal edges.
We refer to \cite[Section 9.3]{LodayVallette12} for a complete study of the operad $\AAinf$, and in particular for a definition of its differential.
There exists an explicit embedding of dg operads $\Ainf \rightarrow \AAinf$, as constructed in \cite[Section 4]{MarklShnider06} and in \cite[Section 1.3.1.5]{mazuir-I}.
The problem of the construction of an explicit morphism of dg operads $\AAinf \rightarrow \Ainf$ is more complicated and is the subject of the following proposition. 

\begin{definition}[Retraction]
A morphism of dg operads $\AAinf \rightarrow \Ainf$ sending $\premiertermecobarbarA$ to $\premiertermecobarbarA$ will be called a \emph{retraction of the operad $\AAinf$ onto the operad \Ainf }.
\end{definition}

\begin{proposition}
\label{prop:retract}
The datum of a diagonal on the operad \Ainf\ is equivalent to the datum of a retraction $r : \AAinf \rightarrow \Ainf$.
\end{proposition}

\begin{proof}
We apply the general theory of operadic twisting morphisms \cite[Section 6.4]{LodayVallette12} to prove the following sequence of isomorphisms: 
\begin{eqnarray*}
  \Hom_{\mathsf{Op}} (\Omega As^{\text{!`}}, \Omega As^{\text{!`}} \otimes \Omega As^{\text{!`}}) & \cong & \mathrm{Tw}(As^{\text{!`}}, \Omega As^{\text{!`}} \otimes \Omega As^{\text{!`}}) \\
  & \cong & \mathrm{Tw}(B As,\Omega As^{\text{!`}}) \\
  & \cong & \Hom_{\mathsf{Op}} (\Omega B As, \Omega As^{\text{!`}}) \ . 
\end{eqnarray*}
The first and last isomorphisms are given by the bar-cobar adjunction. We thus only need to explain the second isomorphism. 
A twisting morphism $As^{\text{!`}}\to \Omega As^{\text{!`}} \otimes \Omega As^{\text{!`}}$ is by definition a Maurer--Cartan element in the convolution pre-Lie algebra associated to the convolution dg operad $\Hom (As^{\text{!`}}, \Omega As^{\text{!`}} \otimes \Omega As^{\text{!`}})$.
This convolution dg operad is in turn isomorphic to the desuspension $\mathcal{S}^{-1}(\Omega As^{\text{!`}} \otimes \Omega As^{\text{!`}})$.
Since the cooperad $As^{\text{!`}}$ is 1-dimensional in every arity, and since the arity-wise linear dual dg cooperad of the desuspended dg operad $\mathcal{S}^{-1}(\Omega As^{\text{!`}})$ is isomorphic to the bar construction $B As$, we have that 
the desuspension $\mathcal{S}^{-1}(\Omega As^{\text{!`}} \otimes \Omega As^{\text{!`}})$ is isomorphic to the convolution dg operad $\Hom (B As, \Omega As^{\text{!`}})$. 
We hence have the following isomorphisms of dg operads
\[ \Hom (As^{\text{!`}}, \Omega As^{\text{!`}} \otimes \Omega As^{\text{!`}}) \cong \mathcal{S}^{-1}(\Omega As^{\text{!`}} \otimes \Omega As^{\text{!`}}) \cong \Hom (B As, \Omega As^{\text{!`}}) \ . \]
This implies an isomorphism on the level of the Maurer--Cartan elements of the associated dg pre-Lie algebras, that is
\[ \mathrm{Tw}(As^{\text{!`}}, \Omega As^{\text{!`}} \otimes \Omega As^{\text{!`}}) \cong \mathrm{Tw}(B As,\Omega As^{\text{!`}}) \ . \]
We finally check that the condition $\triangle (\premiertermecobarbarA) = \premiertermecobarbarA \otimes \premiertermecobarbarA$ is equivalent to the condition $r(\premiertermecobarbarA)=\premiertermecobarbarA$.
\end{proof}

\cref{prop:retract} clarifies in particular the construction of the diagonal on the operad \Ainf\ given in \cite{MarklShnider06}. 
The operad $\AAinf$ can indeed be seen as the cellular chains on the cubical realization of the associahedra \cite[Section 9.3.1]{LodayVallette12}. 
It comes with an elementary diagonal $\AAinf \rightarrow \AAinf \otimes \AAinf$ defined using the Serre cubical diagonal of \cite{Serre51}.
M. Markl and S. Shnider then define a retraction $r:\AAinf \rightarrow \Ainf$ and deduce a diagonal on the operad \Ainf\ as the composite
\[ \Ainf \longrightarrow \AAinf \longrightarrow \AAinf \otimes \AAinf \overset{r \otimes r}{\longrightarrow} \Ainf \otimes \Ainf \ . \]
Their choice of retraction recovers the diagonal constructed directly on the level of the associahedra in \cite[Theorem 2]{MTTV19}.
A similar proof would however not adapt to the case of the multiplihedra, as they are not simple polytopes hence do not admit a cubical realization.

\begin{remark}
\label{rem:Morse}
As observed in \cite[Remark 1.6]{LA21}, the methods used to construct our cellular approximation of the diagonal could be related to the Fulton--Sturmfels formula \cite[Theorem 4.2]{FultonSturmfels97}, appearing in the study of the intersection theory on toric varieties.
We also expect an interpretation of \cref{prop:retract} in terms of Morse theory, in the vein of \cite{FriedmanMardonesSinha21,Frankland07}. 
There should also be an interpretation in terms of discrete Morse theory as in \cite[Section 1.1.4]{Thorngren18} for the case of the standard simplices.  
\end{remark}

\subsection{Convolution \Ainf -algebra} \label{ss:conv-ainf-alg}

\subsubsection{Standard convolution algebra}

Given a dg algebra $A$ and a dg coalgebra $C$, recall from  \cite[Section 1.6]{LodayVallette12} that one can define the \textit{convolution algebra} of $C$ and $A$ as the dg algebra $(\Hom (C,A) , [ \partial , \cdot ] , \star)$, where $\Hom (C,A)$
is the dg module of maps $C \rightarrow A$, endowed with the convolution product $f \star g := \mu_A \circ ( f \otimes g) \circ \Delta_C$. 
The convolution algebra construction is in fact functorial, i.e. fits into a bifunctor $\mathsf{(dg-cog)^{op}} \times \mathsf{dg-alg} \rightarrow \mathsf{dg-alg}$ defined on objects as $(C,A) \mapsto \Hom (C,A)$.
A Maurer-Cartan element $\alpha$ of $\Hom (C,A)$, i.e. a map $\alpha : C \rightarrow A$ such that 
$[ \partial , \alpha ] + \alpha \star \alpha  = 0$,
is then called a \emph{twisting morphism}. 
Twisting morphisms define twisted differentials on the tensor product $C \otimes A$ via the formula
\[ \partial_\alpha := \partial_{C \otimes A} + (\ide \otimes \mu_A ) ( \ide \otimes \alpha \otimes \ide ) ( \Delta_C \otimes \ide ) \ . \]

Twisted differentials appear in the computation of the singular homology of fiber spaces \cite{Brown59}. 
Given a fibration $F \rightarrow X \rightarrow B$ satisfying some mild assumptions, the singular homology of $X$ can then be computed as the homology of the tensor product $C_*(B) \otimes C_*(F)$ endowed with a twisted differential, where $C_*(F)$ is seen as a dg module over the dg algebra $C_*(\Omega B)$.

\subsubsection{Convolution \Ainf -algebra} \label{sss:conv-ainf-alg}

One defines an \textit{\Ainf -coalgebra} structure on a dg module $C$ to be a morphism of dg operads $\Ainf \rightarrow \coEnd_C$, where $\coEnd_C(n) = \Hom ( C , C^{\otimes n} )$. 
Put differently, it is the structure dual to the structure of \Ainf -algebra, i.e. it corresponds to a collection of operations $c_n : C \rightarrow C^{\otimes n}$ of degree $n-2$ satisfying the equations obtained by inverting inputs and outputs in the equations for \Ainf -algebras. 
The notion of an \Ainf -morphism between \Ainf -coalgebras is defined in a similar fashion: either in terms of operations $f_n : C \rightarrow D^{\otimes n}$ of degree $n-1$ and satisfying the equations dual to the equations for \Ainf -morphisms, or equivalently as a morphism of dg operadic bimodules $\Minf \rightarrow \coHom^{C_1}_{C_2}$.
Our results allow us to extend the convolution algebra construction when $C$ is an \Ainf -coalgebra and $A$ is an \Ainf -algebra.

\begin{proposition} 
\label{prop:convolution-ainf} $ $
\begin{enumerate}[leftmargin=*]
    \item Let $C$ be an \Ainf -coalgebra and $A$ be an \Ainf -algebra. 
A diagonal on the operad \Ainf\ yields an \Ainf -algebra structure on the dg module $(\Hom (C,A) , [\partial,\cdot])$. 
We call this \Ainf -algebra the \emph{convolution \Ainf -algebra of $C$ and $A$}.
\item Let $F : A_1 \rightsquigarrow A_2$ be an \Ainf -morphism between two \Ainf -algebras $A_1$ and $A_2$ and $G : C_2 \rightsquigarrow C_1$ be an \Ainf -morphism between two \Ainf -coalgebras $C_2$ and $C_1$. A diagonal on the operad \Minf\ yields  an \Ainf -morphism between the convolution \Ainf -algebras $\Hom (C_1,A_1)$ and $\Hom (C_2,A_2)$. 
\end{enumerate}
\end{proposition}

\begin{proof} $ $
\begin{enumerate}[leftmargin=*]
\item Given a diagonal $\Ainf \rightarrow \Ainf \otimes \Ainf$, the following composite of morphism of operads defines the \Ainf -algebra structure on $\Hom(C,A)$ : 
\[ \Ainf \to \Ainf \otimes \Ainf \to \coEnd_C\otimes \End_A \to \End_{\Hom(C,A)} \ , \]
where the morphism of dg operads $\coEnd_C \otimes \End_A \to \End_{\Hom(C,A)}$ is straightforward to define.
\item Given a diagonal $\Minf \rightarrow \Minf \otimes \Minf$, we consider in a similar fashion the composite of morphism of operadic bimodules
\[ \Minf \to \Minf \otimes \Minf \to \coHom^{C_2}_{C_1} \otimes \Hom^{A_1}_{A_2} \to \Hom^{\Hom(C_1,A_1)}_{\Hom(C_2, A_2)} \ . \] 
\end{enumerate}
\end{proof}

\begin{proposition}
  \label{coroll:nobifunctor}
For any diagonal on $\Ainf$ and for any diagonal on $\Minf$, the convolution $\Ainf$-algebra $\Hom(C,A)$ does not define a bifunctor $(\infAcog)^{\mathrm{op}} \times \infAalg \rightarrow \infAalg$.
\end{proposition}
\begin{proof}
  This is a direct corollary to \cref{thm:nofunctorial}.
\end{proof}

\cref{prop:convolution-ainf} implies in particular that for an \Ainf -coalgebra $C$ and an \Ainf -algebra $A$, it is still possible to define the notion of a \textit{twisting morphism} $\alpha : C \rightarrow A$ as a Maurer-Cartan element in the \Ainf -algebra $\Hom (C,A)$, see \cite[Equation (4.1), p.72]{dotsenko2018twisting} for instance.
It also implies that the \Ainf -morphism $\Hom (C_1,A_1) \rightsquigarrow \Hom (C_2,A_2)$ defined by the \Ainf -morphism $F : A_1 \rightsquigarrow A_2$ and $G : C_2 \rightsquigarrow C_1$, sends a twisting morphism $C_1 \rightarrow A_1$ to a twisting morphism $C_2 \rightarrow A_2$.
We will use this key property in order to pursue the work of Brown \cite{Brown59} and \cite{Proute86} on the homology of fibered spaces in a forthcoming paper.

\subsubsection{Diagonals as twisting morphisms}
\label{sec:RNW}

The results of \cref{sss:conv-ainf-alg} can be interpreted in a more general framework, developed by D. Robert-Nicoud and F. Wierstra in \cite{RobertNicoudWierstraI,RobertNicoudWierstraII}. 

\begin{proposition}
  \label{coroll:twisting}
  The datum of a diagonal on $\Ainf$ is equivalent to the datum of a twisting morphism $\alpha \in \mathrm{Tw}(B As,\Omega As^{\text{!`}})$ sending $\premiertermecobarbarA$ to $\premiertermecobarbarA$. 
\end{proposition}

\begin{proof}
This result was proven in the proof of \cref{prop:retract}.
\end{proof}

Setting $\mathcal{C}=B As$ and $\mathcal{P}=\Omega As^{\text{!`}}$ and working in the context of non-symmetric operads where the operad $\Linf$ of \cite{RobertNicoudWierstraI,RobertNicoudWierstraII} is replaced by the operad $\Ainf$, we recover \cref{coroll:twisting} (and thus \cref{prop:retract}) via \cite[Theorem 7.1]{RobertNicoudWierstraI} and Point~(1) of \cref{prop:convolution-ainf} via \cite[Theorem 4.1]{RobertNicoudWierstraI}.
We denote by $\Aalg$ the category of $\Ainf$-algebras and their \emph{strict} morphisms \cite[Section 10.2.1]{LodayVallette12}.
It is shown in \cite[Corollary 5.4]{RobertNicoudWierstraI} that the assignments
\begin{eqnarray}
  \Hom(-,\id) &:& (\infAcog)^{\mathrm{op}} \times \Aalg \to \Aalg \label{eq:bif1} \\
  \Hom(\id, -) &:& (\Acog)^{\mathrm{op}} \times \infAalg \to \Aalg \label{eq:bif2}
\end{eqnarray}
given by the convolution $\Ainf$-algebra extend to bifunctors. 
The authors also show that these two bifunctors do \emph{not} extend in general to a bifunctor 
\begin{eqnarray}
  \Hom(-,-) &:& \mathsf{(\infAcog)^{op}} \times \infAalg \to \infAalg \label{eq:bifunctor}
\end{eqnarray}
as this assignment is not compatible with the composition of $\Ainf$-morphisms \cite[Theorem 6.6]{RobertNicoudWierstraI}.
Point~(2) of \cref{prop:convolution-ainf} allows us to define the assignment (\ref{eq:bifunctor}) directly, and \cref{coroll:nobifunctor} can be seen as a stronger version of \cite[Theorem 6.6]{RobertNicoudWierstraI}, in the special case of $\Ainf$-algebras. 

The main result of \cite{RobertNicoudWierstraII} says that if a twisting morphism $\alpha \in \mathrm{Tw}(B As,\Omega As^{\text{!`}})$ is Koszul, then the possible compositions of the two bifunctors (\ref{eq:bif1}) and (\ref{eq:bif2}) are homotopic and that they extend to a bifunctor on the level of the homotopy categories \cite[Theorem 3.6 and Corollary 3.8]{RobertNicoudWierstraII}. 
This should be seen as a statement analogous to Point (3) of \cref{th:homotopy-properties}. It would be interesting to know how the results of \cite{RobertNicoudWierstraI,RobertNicoudWierstraII} can be interpreted from the viewpoint of diagonals, and if they admit an interpretation on the level of polytopes.

\subsection{Diagonals in symplectic topology} \label{ss:diag-symp}

\subsubsection{The work of Lipshitz, {Oszv\'ath} and Thurston}

In \cite{LOT20}, R. Lipshitz, P. Oszv\'ath and D. Thurston also study diagonals on the dg operad \Ainf\ and on the dg operadic bimodule \Minf . They however work exclusively on the dg level, constructing abstract diagonals by using the fact that \Ainf\ and \Minf\ are contractible, and do not provide explicit formulae for these diagonals as in \cref{prop:diagonal-polytopale-a-infini} and \cref{prop:diagonal-polytopale-m-infini}. The goal of their work is to study bordered Heegaard Floer homology of 3-manifolds.
Given a 3-manifold $Y$ with two boundary components, they aim to construct a \emph{bimodule twisted complex} $CFDD^-(Y)$, also called a \emph{type $DD$-bimodule}. The definition of such an object uses a diagonal on the dg operad \Ainf . A diagonal on \Minf\ is then needed in order to relate the categories of bimodules defined with different diagonals on \Ainf , which in turn is needed for properties like the associativity of tensor products. They also expect that diagonals on \Minf\ could be needed in a distant future to define \Ainf -morphisms between bimodule twisted complexes arising from a cobordism between 3-manifolds $Y_1$ and $Y_2$.
Thus, the explicit formula for the diagonal defined in this paper could be used to compute invariants of 3 and 4-manifolds, via implementation in a computer program for instance.

\subsubsection{K\"unneth theorems in Lagrangian Floer theory}
\label{sss:amorim-fukaya}

Let $(M,\omega)$ be a closed symplectic manifold, i.e. a closed manifold $M$ together with a closed non-degenerate 2-form $\omega$ on $M$. 
The \emph{Fukaya category} $\mathrm{Fuk}(M,\omega)$ of $(M,\omega)$ is defined to be the (curved filtered unital) \Ainf -category whose objects are (unobstructed) Lagrangian submanifolds of $M$ and higher compositions are defined by counting pseudo-holomorphic disks with Lagrangian boundary conditions and marked points on their boundary, as represented in \cref{fig:pseudo-hol-disk-bord-lagrang}. 
We refer for instance to~\cite{smith-prolegomenon}~and~\cite{auroux-fukaya} for introductions to this subject.
Given a closed spin Lagrangian submanifold $L \subset M$, K. Fukaya also constructs in \cite{fukaya-cyclic-symmetry} a strictly unital \Ainf -algebra $\mathcal{F}(L)$, the \emph{Fukaya algebra} of the Lagrangian $L$, whose higher multiplications are again defined by counting pseudo-holomorphic disks. 

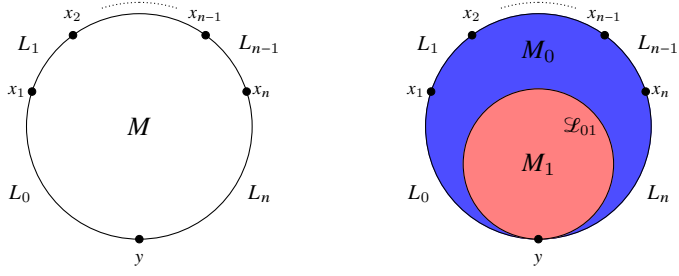
\begin{figure}[h]
\centering

\begin{subfigure}{0.4\textwidth}
\centering

\begin{tikzpicture}[scale = 0.5]

\draw (0,0) circle (3) ; 

\draw (360/20 : 3) node[scale = 0.1]{\pointbullet};
\draw (3*360/20 : 3) node[scale = 0.1]{\pointbullet};
\draw (9*360/20 : 3) node[scale = 0.1]{\pointbullet};
\draw (7*360/20 : 3) node[scale = 0.1]{\pointbullet};
\draw (270 : 3) node[scale = 0.1]{\pointbullet};

\draw (360/20 : 3) node[right,scale=0.7]{$x_n$};
\draw (3*360/20 : 3) node[above=3pt,scale=0.7]{$x_{n-1}$};
\draw (9*360/20 : 3) node[left,scale=0.7]{$x_1$};
\draw (7*360/20 : 3) node[above=3pt,scale=0.7]{$x_2$};
\draw (270 : 3) node[below=3pt,scale=0.7]{$y$};

\draw (2*360/20 : 3) node[above right,scale=0.8]{$L_{n-1}$};
\draw (8*360/20 : 3) node[above left,scale=0.8]{$L_1$};
\draw (11.5*360/20 : 3) node[below left,scale=0.8]{$L_0$};
\draw (18.5*360/20 : 3) node[below right,scale=0.8]{$L_n$};

\draw[densely dotted] (0,3.3) arc (90:108:3.3) ;
\draw[densely dotted] (0,3.3) arc (90:72:3.3) ;

\node at (0,0) {$M$} ;

\end{tikzpicture}

\end{subfigure}
\begin{subfigure}{0.4\textwidth}
\centering

\begin{tikzpicture}[scale = 0.5]

\draw[fill,blue!70] (0,0) circle (3) ;
\draw (0,0) circle (3) ;

\draw (360/20 : 3) node[scale = 0.1]{\pointbullet};
\draw (3*360/20 : 3) node[scale = 0.1]{\pointbullet};
\draw (9*360/20 : 3) node[scale = 0.1]{\pointbullet};
\draw (7*360/20 : 3) node[scale = 0.1]{\pointbullet};

\draw (360/20 : 3) node[right,scale=0.7]{$x_n$};
\draw (3*360/20 : 3) node[above=3pt,scale=0.7]{$x_{n-1}$};
\draw (9*360/20 : 3) node[left,scale=0.7]{$x_1$};
\draw (7*360/20 : 3) node[above=3pt,scale=0.7]{$x_2$};
\draw (270 : 3) node[below=3pt,scale=0.7]{$y$};

\draw (2*360/20 : 3) node[above right,scale=0.8]{$L_{n-1}$};
\draw (8*360/20 : 3) node[above left,scale=0.8]{$L_1$};
\draw (11.5*360/20 : 3) node[below left,scale=0.8]{$L_0$};
\draw (18.5*360/20 : 3) node[below right,scale=0.8]{$L_n$};

\draw[densely dotted] (0,3.3) arc (90:108:3.3) ;
\draw[densely dotted] (0,3.3) arc (90:72:3.3) ;

\node at (0,2) {$M_0$} ;

\draw[fill,red!50] (0,-1) circle (2) ;
\draw (0,-1) circle (2) ;

\node at (0,-1) {$M_1$} ;

\draw (30 : 2) + (0,-1) node[xshift=-0.3cm,scale=0.8]{$\mathcal{L}_{01}$};

\draw (270 : 3) node[scale = 0.1]{\pointbullet};

\end{tikzpicture}

\end{subfigure}

\caption{On the left, a pseudo-holomorphic disk defining the \Ainf -category structure on $\mathrm{Fuk}(M)$. On the right, a pseudo-holomorphic quilted disk defining an \Ainf -functor $\mathrm{Fuk}(M_0)\rightsquigarrow\mathrm{Fuk}(M_1)$} \label{fig:pseudo-hol-disk-bord-lagrang}
\end{figure}

In \cite{amorim-lagrangian}, L. Amorim shows that given two symplectic manifolds $M_1$ and $M_2$ together with Lagrangians $L_i \subset M_i$, the Fukaya algebra of the product Lagrangian $L_1 \times L_2$ is quasi-isomorphic to the tensor product of their Fukaya algebras, i.e. $\mathcal{F}(L_1 \times L_2) \simeq \mathcal{F}(L_1) \otimes \mathcal{F}(L_2)$. His proof relies on a theorem that he proves in~\cite{amorim-tensor}, giving a criterion for an \Ainf -algebra $C$ to be quasi-isomorphic to the tensor \Ainf -algebra $A \otimes B$ (see \cref{def:tensor-product-ainf-alg}) of two commuting \Ainf -subalgebras $A \subset C$ and $B \subset C$, which he then applies to the two \Ainf -subalgebras $\mathcal{F}(L_1) \subset \mathcal{F}(L_1 \times L_2)$ and $\mathcal{F}(L_2) \subset \mathcal{F}(L_1 \times L_2)$.
Fukaya generalizes this result in \cite{fukaya-unobstructed}, working this time on the level of Fukaya categories. He proves that for two closed symplectic manifolds $M_0$ and $M_1$ there exists a unital \Ainf -functor
\[ \mathrm{Fuk}(M_0) \otimes \mathrm{Fuk}(M_1) \longrightarrow \mathrm{Fuk}(M_0^- \times M_1) \]
which is a homotopy equivalence to its image. 

Let now $M_0$ and $M_1$ be two compact symplectic manifolds. Define a \emph{Lagrangian correspondence} from $M_0$ to $M_1$ to be a Lagrangian submanifold $\mathcal{L} \subset M_0^{-} \times M_1$.
In \cite{mau-wehrheim-woodward}, S. Mau, K. Wehrheim and C. Woodward associate to a Lagrangian correspondence $\mathcal{L}$ (with additional technical assumptions) an \Ainf -functor $\Phi_{\mathcal{L}} : \mathrm{Fuk}(M_0) \rightsquigarrow \mathrm{Fuk}(M_1)$.  
It is defined on objects as 
\[ \Phi_{\mathcal{L}} (L_0) := \pi_{M_1} ( L_0 \times_{M_0} \mathcal{L} ) \ , \]
where $\pi_{M_1}$ denotes the projection $M_0 \times M_0^{-} \times M_1 \rightarrow M_1$ and $\times_{M_0}$ is the fiber product over $M_0$. The operations of $\Phi_{\mathcal{L}}$ are defined by counting pseudo-holomorphic quilted disks with Lagrangian boundary conditions, seam condition on $\mathcal{L}$ and marked points on their boundary, as represented in \cref{fig:pseudo-hol-disk-bord-lagrang}. 
The tensor product of $\Ainf$-functors defined in the present paper allows one to consider the $\Ainf$-functor $\Phi_{\mathcal{L}_M} \otimes \Phi_{\mathcal{L}_N}$ associated to a pair of Lagrangian correspondences, raising the following question. 

\begin{samepage}
\begin{problem} \label{problem}
Does the diagram
\begin{center} 
\begin{tikzcd}[column sep = 12ex]
\mathrm{Fuk}(M_0) \otimes \mathrm{Fuk}(N_0) \arrow[d,squiggly] \arrow[r,"\Phi_{\mathcal{L}_M} \otimes \Phi_{\mathcal{L}_N}",squiggly] & \mathrm{Fuk}(M_1) \otimes \mathrm{Fuk}(N_1) \arrow[d,squiggly] \\
 \mathrm{Fuk}(M_0 \times N_0) \arrow[r,below,"\Phi_{\tau ( \mathcal{L}_M \times \mathcal{L}_N)}",squiggly] & \mathrm{Fuk}(M_1 \times N_1)
\end{tikzcd}
\end{center}
commute up to homotopy of \Ainf -functors? 
\end{problem}
\end{samepage}

In this diagram, $\mathcal{L}_M \subset M_0^{-} \times M_1$, $\mathcal{L}_N \subset N_0^- \times N_1$ and the symplectomorphism $\tau$ is defined by rearranging the factors of $M_0^{-} \times M_1 \times N_0^- \times N_1$ into the factors of $M_0^{-} \times N_0^- \times M_1 \times N_1$. In other words, we would like to know whether the \emph{algebraic (tensor) product} of geometric \Ainf -functors between Fukaya categories defined in this paper is homotopic to the \Ainf -functor defined by the \emph{geometric product} of the Lagrangian correspondences. 
We refer to \cite[Section 13]{fukaya-unobstructed} for a discussion on two definitions of the notion of a homotopy between \Ainf -functors.

\newpage

\bibliographystyle{amsalpha}
\bibliography{bib}





\end{document}